\documentclass[a4paper, reqno]{amsart}
\usepackage[utf8]{inputenc}
\usepackage[english]{babel}
\usepackage[T1]{fontenc}
\usepackage{amsmath}
\usepackage{amsthm}
\usepackage{amssymb}
\usepackage{tikz}
\usetikzlibrary{arrows}
\usetikzlibrary{decorations.markings}
\usetikzlibrary{patterns}
\usepackage{enumerate}
\usepackage{float}
\usepackage{hyperref}
\usepackage{todonotes}
\usetikzlibrary{calc}
\usepackage{array}

\makeatletter
\newcommand{\thickhline}{%
    \noalign {\ifnum 0=`}\fi \hrule height 1pt
    \futurelet \reserved@a \@xhline
}
\newcolumntype{"}{@{\hskip\tabcolsep\vrule width 1pt\hskip\tabcolsep}}
\makeatother

\newcommand{\alp}{\mathsf{alph}}

\newcommand{\dM}{\mathbb{M}}
\newcommand{\N}{\mathbb{N}}
\newcommand{\Z}{\mathbb{Z}}
\newcommand{\E}{e}%{\mathcal{E}}
\newcommand{\NF}{\mathsf{NF}}
\newcommand{\IRR}{\mathsf{IRR}}
\newcommand{\BR}{\mathsf{BR}}
\newcommand{\dG}{G}
\newcommand{\Sol}{\mathsf{sol}}
\newcommand{\magn}{\mathsf{mag}}
\newcommand{\DHB}{H_3(\Z)}

\renewcommand{\mod}{\operatorname{mod}}

\newcommand{\norm}[1]{|\!| #1 |\!|}
\newcommand{\rest}{\mathord\restriction}

\theoremstyle{definition}
\newtheorem{defi}{Definition}[section]

\newtheorem{bem}[defi]{Remark}
\theoremstyle{plain}
\newtheorem{theorem}[defi]{Theorem}
\newtheorem{proposition}[defi]{Proposition}

\newtheorem{lemma}[defi]{Lemma}

\newtheorem{convention}[defi]{Convention}

\begin{document}

\title{Closure properties of knapsack semilinear groups}

\author[M.~Figelius]{Michael Figelius}
\email{figelius@eti.uni-siegen.de}
\author[M.~Lohrey]{Markus Lohrey}
\email{lohrey@eti.uni-siegen.de}
\author[G.~Zetzsche]{Georg Zetzsche}
\email{georg@mpi-sws.org}

\address[Michael~Figelius, Markus~Lohrey]{Universit{\"a}t Siegen, Germany}
\address[Georg~Zetzsche]{Max Planck Institute for Software Systems, Kaiserslautern, Germany}
\thanks{This work has been supported by the DFG research project
	LO 748/12-1}

\begin{abstract}
	A \emph{knapsack equation} in a group $G$ is an equation of the form
	$g_1^{x_1} \cdots g_k^{x_k} = g$ where $g_1,\ldots,g_k,g$ are elements
	of $G$ and $x_1,\ldots,x_k$ are variables that take values in the
	natural numbers. We study the class of groups $G$ for which all
	knapsack equations have effectively semilinear solution sets.  We show
	that the following group constructions preserve effective
	semilinearity: graph products, amalgamated free products with finite
	amalgamated subgroups, HNN-extensions with finite associated subgroups,
	and finite extensions.  Moreover, we study a complexity measure, called
	magnitude, of the resulting semilinear solution sets. More precisely,
	we are interested in the growth of the magnitude in terms of the length
	of the knapsack equation (measured in number of generators).  We
	investigate how this growth changes under the above group operations.
\end{abstract}

\maketitle

\section{Introduction}

The study of algorithmic problems has a long tradition in combinatorial group theory, going back to the work of Dehn
\cite{Dehn11} on the word and conjugacy problem in finitely generated groups. 
Myasnikov, Nikolaev, and Ushakov initiated in  \cite{MyNiUs14} the systematic investigation of a new class of algorithmic problems
that have their origin in discrete optimization problems over the integers. One of these problems is the 
{\em knapsack problem}. Myasnikov et al.~proposed the following definition for the knapsack problem in a 
finitely generated group $G$:
The input is a sequence of group elements $g_1, \ldots, g_k, g \in G$ (specified
by finite words over the generators of $G$) and it is asked whether there exist natural numbers
$x_1, \ldots, x_k \in \mathbb{N}$
such that $g_1^{x_1} \cdots g_k^{x_k} = g$ in $G$. 
For the particular case $G = \mathbb{Z}$  (where the additive notation 
$x_1 \cdot g_1 + \cdots + x_k \cdot g_k = g$ is usually preferred)
this problem  is {\sf NP}-complete 
if the numbers $g_1,\ldots, g_k,g \in \mathbb{Z}$ are given in binary notation \cite{Karp72,Haa11}.\footnote{Karp in his seminal paper \cite{Karp72}
defined knapsack in a slightly different way. {\sf NP}-completeness of the above version was shown in \cite{Haa11}.}
On the other hand, if $g_1,\ldots, g_k,g$ are given in unary notation, then the knapsack problem for the integers was shown to be complete for the circuit
complexity class $\mathsf{TC}^0$ \cite{ElberfeldJT11}. Note that the unary notation for integers corresponds to the case where an integer is given by 
a word over a generating set $\{t,t^{-1}\}$. In one particular case, the knapsack problem was studied for a non-commutative group before
the work of Myasnikov et al.: in \cite{BabaiBCIL96}, it was shown that the
knapsack problem for commutative matrix groups over algebraic number fields can be solved in polynomial time.

Let us give a brief survey of the results that were obtained for the knapsack problem in \cite{MyNiUs14}  and successive papers:
\begin{itemize}
\item Knapsack can be solved in polynomial
time for every hyperbolic group \cite{MyNiUs14}. In \cite{FrenkelNU16} this result was extended to free products of any
finite number of hyperbolic groups and finitely generated abelian groups. A further generalization was obtained
in \cite{LOHREY2019}, where the smallest class of groups that can be obtained from hyperbolic groups using the operations of free 
products and direct products with $\mathbb{Z}$ was considered. It was shown that for every group in this class the knapsack
problem belongs to the complexity class ${\sf LogCFL}$ (a subclass of ${\sf P}$).
\item There are nilpotent groups of class $2$ for which knapsack is undecidable. Examples
are direct products of sufficiently many copies of the discrete Heisenberg group $H_3(\mathbb{Z})$ \cite{KoenigLohreyZetzsche2015a},
and free nilpotent groups of class $2$ and sufficiently high rank \cite{MiTr17}.
\item Knapsack for $H_3(\mathbb{Z})$ is decidable \cite{KoenigLohreyZetzsche2015a}. In particular, 
together with the previous point it follows that decidability
of knapsack is not preserved under direct products.
\item Knapsack is decidable for every co-context-free group~\cite{KoenigLohreyZetzsche2015a}, i.e., groups where the set of all words over the generators that do not represent the identity
is a context-free language. Lehnert and Schweitzer \cite{LehSch07} have shown that the Higman-Thompson groups are
co-context-free.
\item Knapsack belongs to ${\sf NP}$ for all virtually special groups (finite extensions of subgroups
of graph groups) \cite{LohreyZetzsche2016a}. The class of virtually special groups is very rich. It contains all Coxeter groups, 
one-relator groups with torsion, fully residually free groups, and  fundamental groups of hyperbolic 3-manifolds.
For graph groups (also known as right-angled Artin groups) a complete classification of the complexity
of knapsack was obtained in \cite{LohreyZ18}: If the underlying graph contains an induced path or cycle on 4 nodes, then knapsack
is ${\sf NP}$-complete; in all other cases knapsack can be solved in polynomial time (even in {\sf LogCFL}).
\item Knapsack is {\sf NP}-complete for non-abelian free solvable groups \cite{FigeliusGLZ20} and solvable Baumslag-Solitar groups $\mathsf{BS}(1,q)$ \cite{LohreyZ20} with $q>1$. 
For Baumslag-Solitar groups $\mathsf{BS}(p,q)$ with $p \neq 1 \neq q$ and $\mathrm{gcd}(p,q) = 1$, decidability of knapsack was shown in \cite{DudTre18}.
\item Decidability of knapsack is preserved under finite extensions, HNN-exten\-sions over finite associated subgroups 
and amalgamated free products over finite subgroups \cite{LohreyZetzsche2016a}.
\item In \cite{BGZ21}, there is a characterization of those wreath products $G \wr H$ for which the knapsack
problem is decidable. The characterization is in terms of (i)~decidability properties of the groups $G$ and $H$ and (ii)~whether $G$ is abelian.
\end{itemize}

In this work, we initiate the systematic study of solution sets of equations $g_1^{x_1}\cdots g_k^{x_k}=g$, which we call \emph{knapsack equations}.
The \emph{solution set} of this equation is $\{ (n_1, \ldots, n_k) \in \mathbb{N}^k \mid g_1^{n_1} \cdots g_k^{n_k} = g \text{ in } G\}$. In the papers
\cite{LOHREY2019,KoenigLohreyZetzsche2015a,LohreyZ18} it turned out that in many groups the solution
set of every knapsack equation is a {\em semilinear set}. Recall that a subset $S \subseteq \mathbb{N}^k$ is 
semilinear if it is a finite union of linear sets, and a subset $L \subseteq \mathbb{N}^k$ is linear if there are vectors
$v_0, v_1, \ldots, v_\ell \in \mathbb{N}^k$ such that $L = \{ v_0 + \lambda_1 v_1 + \cdots + \lambda_\ell v_\ell \mid \lambda_1, \ldots, \lambda_\ell \in \mathbb{N}\}$.
Semilinear sets play a prominent role in many areas of computer science and mathematics, e.g. in automata theory and logic.
It is known that the class of semilinear sets is closed under Boolean operations and that the semilinear sets are exactly
the sets that are first-order definable in Presburger arithmetic (i.e., the structure $(\mathbb{N},+)$)~\cite{GS66}.

We say that a group is {\em knapsack-semilinear} if for every knapsack equation the set of all solutions is semilinear. 
Note that in any group $G$ the set of solutions on an equation $g_1^x = g$ is periodic and hence semilinear. 
Moreover, every finitely generated abelian group is semilinear (since solution sets of linear equations are Presburger definable).
Nontrivial examples of knapsack-semilinear groups are hyperbolic groups \cite{LOHREY2019}, graph groups \cite{LohreyZ18}, and
co-context free groups \cite{KoenigLohreyZetzsche2015a}.\footnote{Knapsack-semilinearity of co-context free groups is not stated in 
\cite{KoenigLohreyZetzsche2015a} but follows immediately from the proof for the decidability of knapsack.}

Obviously, every finitely generated subgroup of a finitely generated knapsack-semilinear group is knapsack-semilinear as well.
Moreover, it was shown in \cite{GanardiKLZ18} that the class of knapsack-semilinear groups is closed under wreath products.
In this paper we prove the closure of the class of knapsack-semilinear groups under
\begin{itemize}
\item finite extensions,
\item graph products,
\item amalgamated free products with finite amalgamated subgroups, and
\item HNN-extensions with finite associated subgroups.
\end{itemize}
The operation of graph product interpolates between direct products and free products. It is specified
by a finite graph $(V,E)$, where every node $v \in V$ is labelled with a group $G_v$. One takes the free
product of the groups $G_v$ ($v \in V$) modulo the congruence that allows elements from adjacent groups
to commute. Amalgamated free products and HNN-extensions are fundamental constructions in combinatorial 
group theory; see Section~\ref{sec-HNN+amalgamated} for references.

In order to get complexity bounds for the knapsack problem, the sole concept of knapsack-semilinearity is not useful.
For this purpose, we need a quantitative measure for semilinear sets; see also
\cite{chistikov_et_al:LIPIcs:2016:6263}: For a semilinear set 
$$L = \bigcup_{1 \leq i \leq n} \{ v_{i,0} + \lambda_1 v_{i,1} + \cdots + \lambda_{\ell_i} v_{i,\ell_i} \mid \lambda_1, \ldots, \lambda_{\ell_i} \in \mathbb{N}\}$$
we call the tuple of all vectors $v_{i,j}$ a {\em semilinear representation} for $L$. The {\em magnitude} of this semilinear representation is the largest
number that occurs in some of the vectors $v_{i,j}$. Finally, the magnitude of a semilinear set $L$ is the smallest magnitude among all semilinear 
representations of $L$.

Our proofs showing that the above group constructions preserve knapsack-semi\-linearity also yield
upper bounds for the magnitude of solution sets in terms of (i) the total length of the knapsack equation (measured in the total number
of generators) and (ii) the number of variables in the knapsack equation. For this, we introduce a function 
$\mathsf{K}_G(n,m)$ that yields the maximal magnitude of a solution set for a knapsack equation over $G$ of total
length at most $n$ and at most $m$ variables.
Roughly speaking, it turns out 
that finite extensions, amalgamated free products with finite amalgamated subgroups, and
HNN-extensions with finite associated subgroups only lead to a polynomial blowup for the function $\mathsf{K}_G(n,m)$
(actually, this function also depends on the generating set for $G$),
whereas graph products can lead to an exponential blowup. On the other hand, if we bound the number of variables
by a constant, then also graph products only lead to a polynomial blowup for the function $\mathsf{K}_G(n,m)$.

\section{Words, monoids and groups}

Fix a non-empty set $\Sigma$, which is also called an alphabet in the following. Its elements are also called symbols.
 A word over $\Sigma$ is a finite sequence $w = a_1 a_2 \ldots a_n$ of elements $a_1, \ldots, a_n \in \Sigma$.
We write $|w|=n$ for the length $w$ and $\alp(w) = \{ a_1, a_2, \ldots, a_n\}$ for the set of symbols that occur in $w$.
For $a\in \Sigma$, we write $|w|_{a}$ to denote the number of occurrences of $a$ in $w$.
The {\em free monoid} $\Sigma^*$ consists of all finite words over $\Sigma$ and the monoid operation is the concatenation of words.
The concatenation of words $u,v \in \Sigma^*$ is simply denoted with $uv$. The identity element of the free monoid $\Sigma^*$ is the empty 
word, which is usually denoted with $\varepsilon$. Here, we prefer to denote the empty word with $1$ according to the following convention:
 
\begin{convention}
For every monoid $M$ we denote the identity element of $M$ with the symbol $1$; even in cases where we deal with several monoids.
\end{convention}
So intuitively, all monoids that we deal with share the same identity element $1$. This convention will simplify our notations.

For a set $\Omega$ we denote with $F(\Omega)$ the free group generated by $\Omega$. Formally, it can be defined as follows: 
Let $\Omega^{-1} = \{ a^{-1} \mid a \in \Omega\}$
be a disjoint copy of $\Omega$ (the set of formal inverses) and let $\Sigma = \Omega \cup \Omega^{-1}$.
Then the {\em free group} $F(\Omega)$ can be identified with the set of all words $w \in \Sigma^*$ that do not contain a factor of the form 
$a a^{-1}$ or $a^{-1} a$ for $a \in \Omega$ (so called irreducible words). The  product of two irreducible words $u,v$ is the unique irreducible
word obtained from $uv$ by replacing factors of the form $a a^{-1}$ or $a^{-1} a$ ($a \in \Omega$) by the empty word as long as possible.
For a set $R \subseteq \Sigma^*$ of irreducible words (the relators) we denote with $\langle \Omega \mid R \rangle$ the quotient group $F(\Omega)/N$, where
$N$ is the smallest normal subgroup of $F(\Omega)$ that contains $R$. Every group is isomorphic to a group $\langle \Omega \mid R \rangle$. If $\Omega$ is finite,
then $\langle \Omega \mid R \rangle$ is called {\em finitely generated}. In other words: a group $G$ is finitely generated if there exists a finite subset 
$\Sigma \subseteq G$ such that every element of $G$ is a product of 
elements of $\Sigma$. If for every $a \in \Sigma$ also $a^{-1}$ belongs to $\Sigma$ then $\Sigma$ is called a {\em symmetric generating set} for $G$. If $\Sigma$ is a symmetric generating set and $u\in\Sigma^*$ with $u=u_1\cdots u_n$, $u_1,\ldots,u_n\in\Sigma$, then by $u^{-1}$, we denote the word $u_n^{-1}\cdots u_1^{-1}$.

Let $\Sigma$ be an alphabet. An \emph{automaton over $\Sigma$} is a tuple
$\mathcal{A}=(Q,E,q_0,q_f)$, where $Q$ is a finite set of \emph{states},
$E\subseteq Q\times (\Sigma\cup \{1\})\times Q$ is a finite set of
\emph{edges}, $q_0\in Q$ is its \emph{initial state}, and $q_f\in Q$ is its
\emph{final state}. Note that here, $1$ denotes the empty word over $\Sigma$.
If there is an edge $(p,w,q)\in E$, we also denote this by $p\xrightarrow{w}q$.
A word $w\in\Sigma^*$ is \emph{accepted by $\mathcal{A}$} if we can write
$w=a_1\cdots a_n$ with $a_1,\ldots,a_n\in\Sigma\cup\{1\}$ such that there are
states $p_0,\ldots,p_n\in Q$, edges $p_{i-1}\xrightarrow{a_i}p_i$ for $1\le
i\le n$, and $p_0=q_0$ and $p_n=q_f$. By $L(\mathcal{A})$, we denote the set
of all words in $\Sigma^*$ accepted by $\mathcal{A}$, which is also called its
\emph{language}. The \emph{size} of $\mathcal{A}$ is $|Q|$, the number of its
states.

\section{Semilinear sets}

Fix a dimension $d \ge 1$.
For a vector $v=(v_1, \dots , v_d)^{\mathsf{T}}\in \Z^d$ we define its norm $\norm{v}=\max \{ |v_i| \mid 1 \leq i \leq d \}$ and for a 
matrix $M \in \Z^{c \times d}$ with entries $m_{i,j}$ ($1 \le i \le c$, $1 \le j \le d$) we define the norm
$\norm{M} = \max \{|m_{i,j}| \mid 1 \le i \le c, 1 \leq j \leq d \}$. Finally, for a finite set of vectors $A \subseteq \N^d$ let
$\norm{A} = \max \{ \norm{a} \mid a \in A \}$.

We extend the operations of vector addition and multiplication of a vector by a matrix to sets of vectors in the obvious way.
A \emph{linear subset} of $\N^d$ is a set of the form
$$
L= L(b,P) = b + P \cdot \N^k
$$
where $b \in \N^d$ and $P \in \N^{d \times k}$.
We call a set $S\subseteq \N^d$ \emph{semilinear}, if it is a finite union of linear sets. 
The class of semilinear sets is known be closed under Boolean operations, a set is semilinear if and only if it is first-order definable in Presburger arithmetic~\cite{GS66}.

If a semilinear set $S$ is given as a union $\bigcup_{i=1}^k L(b_i,P_i)$, we call the tuple $\mathcal{R} = (b_1, P_1, \ldots, b_k, P_k)$
a \emph{semilinear representation} of $S$.
For a semilinear representation $\mathcal{R} = (b_1, P_1, \ldots, b_k, P_k)$ we define 
$\norm{\mathcal{R}} = \max \{ \norm{b_1}, \norm{P_1} \ldots, \norm{b_k}, \norm{P_k}\}$.
The \emph{magnitude} of a semilinear 
set $S$, $\magn(S)$ for short, 
is the smallest possible value for $\norm{\mathcal{R}}$ among all  semilinear representations $\mathcal{R}$ of $S$.

For a linear set $L(b,P) \subseteq \N^d$
we can assume that all columns of $P$ are different.
Hence, if the magnitude of $L(b,P)$ is bounded by $s$ then 
we can bound the number of columns of $P$ by $(s+1)^d$
(since there are only $(s+1)^d$ vectors in $\N^d$ of norm at most $s$).
No better upper bound is known, but if we allow to split $L(b,P)$ into
several linear sets, we get the following lemma from \cite{EisenbrandS06}:

\begin{lemma}[c.f.~\mbox{\cite[Theorem~1]{EisenbrandS06}}] \label{lemma-number-periods}
Let $L = L(b,P) \subseteq \N^d$ be a linear set of magnitude $s = \magn(L)$.
Then $L = \bigcup_{i\in I} L(b,P_i)$ such that every $P_i$ consists of at most $2 d \log(4ds)$
columns from $P$ (and hence, $\magn(L(b,P_i)) \leq s$).
\end{lemma}
We also need the following bound on the magnitude for the intersections of semilinear sets:

\begin{proposition}[c.f.~\mbox{\cite[Theorem 4]{Beier}}] \label{intersection-semilinear-sets}
Let $K = L(b_1,P_1)$ and $L = L(b_2,P_2) \subseteq \N^d$ be linear sets of magnitude at most $s \geq 1$.
Then the intersection $K \cap L$ is semilinear and $\magn(K \cap L) \le (12 d^2 \log^2(4ds) d^{d/2} s^{d+1}+1)s 
\le \mathcal{O}(d^{d/2+3} s^{d+3})$.
\end{proposition}

\begin{proof}
By Lemma~\ref{lemma-number-periods} we can write $K = \bigcup_{i \in I_1}  L(b_1,P_{1,i})$ and
$L = \bigcup_{i \in I_2}  L(b_2,P_{2,i})$ where every $P_{1,i}$ ($P_{2,i}$) consists of at most 
$2 d \log(4ds)$ columns from $P_1$ ($P_2$). We have $K \cap L = \bigcup_{(i,j) \in I_1\times I_2}  L(b_1,P_{1,i}) \cap L(b_2,P_{2,i})$.
From \cite[Theorem 4]{Beier} we get the upper bound $(12 d^2 \log^2(4ds) d^{d/2} s^{d+1}+1)s$
for the magnitude of each intersection $L(b_1,P_{1,i}) \cap L(b_2,P_{2,i})$.
\end{proof}
In the context of knapsack problems (which we will introduce in the next section), we will consider 
semilinear subsets as sets of mappings $f : \{x_1, \ldots, x_d\} \to \N$ for a finite set of variables
$U = \{x_1, \ldots, x_d\}$. Such a mapping $f$ can be identified with the vector $(f(x_1), \ldots, f(x_d))^{\mathsf{T}}$.
This allows to use all vector operations (e.g. addition and scalar multiplication) on the set
$\N^U$ of all mappings from $U$ to $\N$. The pointwise product $f \cdot g$ of two mappings
$f,g \in \N^U$ is defined by $(f \cdot g)(x) = f(x) \cdot g(x)$ for all $x \in U$. Moreover, for mappings
$f \in \N^U$, $g \in \N^V$ with $U \cap V = \emptyset$ we define $f \oplus g \in \N^{U \cup V}$
by $(f \oplus g)(x) = f(x)$ for $x \in U$ and  $(f \oplus g)(y) = g(y)$ for $y \in V$. All operations on
$\N^U$ will be extended to subsets of $\N^U$ in the standard pointwise way.
Note that $\magn(K \oplus L) \le \max\{ \magn(K), \magn(L)\}$ for
semilinear sets $K,L$. If $L \subseteq \N^U$ is semilinear and $V \subseteq U$ then we denote with
$L\rest_{V}$ the semilinear set $\{ f\rest_{V} \mid f \in L \}$ obtained by restricting every function $f \in L$ to the 
subset $V$ of its domain.
Clearly, $L \rest_{V}$ is semilinear too and $\magn(L\rest_{V}) \le \magn(L)$.

\section{Knapsack and exponent equations}

Let $G$ be a finitely generated group with the finite symmetric generating set $\Sigma$.
Moreover, let $X$ be a set of formal variables that take values
from $\N$. For a subset $U\subseteq X$, we call a mapping $\sigma : U \to \N$ a \emph{valuation} for $U$.
An \emph{exponent expression} over $\Sigma$ is a formal expression of the form 
$\E = u_1^{x_1} v_1 u_2^{x_2} v_2  \cdots u_k^{x_k} v_k$
with $k \geq 1$, words $u_i, v_i \in \Sigma^*$ and variables $x_1, \ldots, x_k \in X$. 
Here, we allow $x_i = x_j$ for $i \neq j$.
The words $u_i$ are called the {\em periods} of $\E$, and we can assume that $u_i \neq 1$ for all $1 \le i \le k$.
If every variable in an exponent expression occurs at most once, it is called a \emph{knapsack expression}. 
Let $X_\E= \{ x_1, \ldots, x_k \}$ be the set of variables that occur in $\E$.
 For a valuation $\sigma : U \to \N$ such that $X_\E \subseteq U$ (in which case we also say
 that $\sigma$ is a valuation for $\E$), we define 
 $\sigma(\E) = u_1^{\sigma(x_1)} v_1 u_2^{\sigma(x_2)} v_2  \cdots u_k^{\sigma(x_k)} v_k \in \Sigma^*$.
We say that $\sigma$ is a \emph{$G$-solution} of the equation $\E=1$ if $\sigma(\E)$ evaluates to the identity element $1$ of $G$.
With $\Sol_G(\E)$ we denote the set of all $G$-solutions $\sigma : X_\E \to \N$ of $\E$. We can view $\Sol_G(\E)$ as a subset
of $\N^k$.
The \emph{length} of $\E$ is defined as $\norm{\E} =\sum_{i=1}^k |u_i|+|v_i|$, whereas $k \le \norm{\E}$ is its \emph{degree}, $\deg(\E)$ for short.
We define {\em solvability of exponent equations over $G$} as the following decision problem:
\begin{description}
\item[Input] A finite list of exponent expressions $\E_1,\ldots,\E_n$ over $G$.
\item[Question] Is $\bigcap_{i=1}^n \Sol_G(\E_i)$ non-empty?
\end{description}
The {\em knapsack problem for $G$} is the following 
decision problem:
\begin{description}
\item[Input] A single knapsack expression $\E$ over $G$.
\item[Question] Is $\Sol_G(\E)$ non-empty?
\end{description}
It is easy to observe that the concrete choice of the generating set $\Sigma$ has no influence
on the decidability and complexity status of these problems.

One could also allow exponent expressions of the form $\E = v_0 u_1^{x_1} v_1 u_2^{x_2} v_2  \cdots u_k^{x_k} v_k$.
However, since then $\Sol_G(\E) = \Sol_G(u_1^{x_1} v_1 u_2^{x_2} v_2  \cdots u_k^{x_k} v_kv_0)$, this would result in the same class of solution sets. Moreover, we could also restrict to exponent expressions of the form 
$\E = u_1^{x_1} u_2^{x_2} \cdots u_k^{x_k} v$: for $\E = u_1^{x_1} v_1 u_2^{x_2} v_2   \cdots u_k^{x_k} v_k$
and $\E' = u_1^{x_1} (v_1 u_2 v_1^{-1})^{x_2} (v_1v_2 u_3 v_2^{-1} v_1^{-1})^{x_3} \cdots (v_1\cdots v_{k-1} u_3 v_{k-1}^{-1} \cdots v_1^{-1})^{x_k}
v_1 \cdots v_{k-1} v_k$
we have $\Sol_G(\E) = \Sol_G(\E')$.

\subsection{Knapsack-semilinear groups}

The group $G$ is called {\em knapsack-semilinear} if for every knapsack expression $\E$ over $\Sigma$,
the set $\Sol_G(\E)$ is a semilinear set of vectors and a semilinear representation can be effectively computed from $\E$.
This implies that for every exponent expression $\E$ over $\Sigma$,
the set $\Sol_G(\E)$ is semilinear as well and a semilinear representation can be effectively computed from $\E$.
To see this, consider an exponent expression $\E = u_1^{x_1} v_1 u_2^{x_2} v_2  \cdots u_k^{x_k} v_k$ over $\Sigma$.
Choose pairwise different variables $y_1, y_2, \ldots, y_k$ such that $X_{\E} = \{ x_1, \ldots, x_k \} \subseteq \{ y_1, \ldots, y_k \}$
and consider the knapsack expression
$\E' = u_1^{y_1} v_1 u_2^{y_2} v_2  \cdots u_k^{y_k} v_k$. Moreover, define the equivalence relation 
$R = \{ (i,j) \mid 1 \le i,j \le k, x_i = x_j \}$. We get
$$
\Sol_G(\E) = (\Sol_G(\E') \cap \{ \sigma \mid \sigma : \{ y_1, \ldots, y_k \} \to \N, \forall (i,j) \in R : \sigma(y_i) = \sigma(y_j) \})\rest_{X_{\E}}.
$$
Since semilinear sets are effectively closed under intersection and restriction, the effective semilinearity of 
$\Sol_G(\E')$ yields the effective semilinearity of  $\Sol_G(\E)$.

Also notice that solvability of exponent equations is decidable for every knapsack-semilinear group.
As mentioned in the introduction, the class of knapsack-semilinear groups is very rich.
An example of a group $G$, where knapsack is decidable but solvability of exponent equations
is undecidable is the Heisenberg group $\DHB$ 
(which consists of all upper triangular $(3 \times 3)$-matrices over the integers, where all diagonal entries
are $1$), see \cite{KoenigLohreyZetzsche2015a}. In particular, $\DHB$ is not knapsack-semilinear in a strong sense:
there are knapsack expressions $\E$ such that $\Sol_{\DHB}(\E)$ is not semilinear.

\begin{bem}
	The requirement that the semilinear representation of the solution set
	can be computed effectively is important: There are groups where every
	knapsack equation has a semilinear solution set, but the semilinear
	representation cannot be computed. For example, consider a finitely
	generated torsion group $G$ with an undecidable word
	problem~\cite{Adian2010}. Then every knapsack expression over $G$ has a
	semilinear solution set.  However, computing a semilinear
	representation for $\{n\in\N\mid u^n=1\}$ for a given word $u$ would
	allow us to check whether $u=1$ in $G$.
\end{bem}

For a knapsack-semilinear group $G$ and a finite generating set $\Sigma$ for $G$ we define two growth functions. For $n,m \in \N$ with $m \leq n$ let
$\mathsf{Exp}(n,m)$ be the finite set of all exponent expressions $\E$ over $\Sigma$ such that
(i) $\Sol_G(\E) \neq \emptyset$, (ii) $\norm{\E} \leq n$ and (iii) $\deg(\E) \leq m$.
Moreover, let $\mathsf{Knap}(n,m) \subseteq \mathsf{Exp}(n,m)$ be the set of all knapsack
expressions in $\mathsf{Exp}(n,m)$.
We define the mappings $\mathsf{E}_{G,\Sigma} : \{(n,m) \mid m,n \in \N, m \leq n\} \rightarrow \N$ and $\mathsf{K}_{G,\Sigma} :  \{(n,m) \mid m,n \in \N, m \leq n\} \rightarrow \N$ as follows:
\begin{itemize}
\item $\mathsf{E}_{G,\Sigma}(n,m) = \max \{ \magn(\Sol_G(\E)) \mid \E \in \mathsf{Exp}(n,m) \}$,
\item $\mathsf{K}_{G,\Sigma}(n,m) = \max \{ \magn(\Sol_G(\E)) \mid \E \in \mathsf{Knap}(n,m) \}$.
\end{itemize}
Clearly, if $\Sol_G(\E) \neq \emptyset$  and $\magn(\Sol_G(\E)) \le N$ then $\E$ has a $G$-solution $\sigma$
such that $\sigma(x) \leq N$ for all variables $x\in X_\E$. Therefore, if $G$ has a decidable word problem and
we have a computable bound on the function $\mathsf{E}_{G,\Sigma}$ then we obtain a nondeterministic  
algorithm for solvability of exponent equations over $G$:
given an exponent expression $\E$ we can guess $\sigma : X_\E \to \N$ with $\sigma(x) \le N$ for all variables
$x$ and then verify (using an algorithm for the word problem), whether $\sigma$ is indeed a solution.

Let  $\Sigma$ and $\Sigma'$ be two generating sets for the group $G$. Then there is a constant $c$ such that
$\mathsf{E}_{G,\Sigma}(n,m) \le \mathsf{E}_{G,\Sigma'}(cn,m)$ and 
$\mathsf{K}_{G,\Sigma}(n,m) \le \mathsf{K}_{G,\Sigma'}(cn,m)$. To see this, note that 
for every $a\in \Sigma$ there is a word $w_a\in (\Sigma')^*$ such that $a$ and $w_a$ are representing the same element in $G$. 
Then we can choose $c=\max \{ |w_a| \mid a\in \Sigma\}$.

\section{Part 1: Exponent equations in graph products} \label{sec-part1}

In this section we introduce graph products of groups (Section~\ref{subsect:graph products}) and show that
every graph product of knapsack-semilinear groups is again knapsack-semilinear (Section~\ref{sec-main}).
Our definition of graph products is based on trace monoids (also known as partially commutative monoids), which
we discuss first. 

\subsection{Trace monoids} \label{sec-traces}

In the following we introduce some notions from trace theory, see
\cite{Die90lncs,DieRoz95} for more details.
An {\em independence alphabet} is an undirected graph  $(\Sigma,\mathrel{I})$
(without loops). Thus, $\mathrel{I}$ is a symmetric and irreflexive
relation on $\Sigma$. The set $\Sigma$ may be infinite. 
The \emph{trace monoid} $\dM(\Sigma,I)$ is defined as the quotient $\dM(\Sigma,I)=\Sigma^*/\{ab=ba\mid (a,b)\in I\}$ with
concatenation as operation and the empty trace $1$ as the neutral
element. Its elements are called
\emph{traces}. We denote by $[w]_{\mathrel{I}}$ the trace represented by the word
$w\in \Sigma^*$. Let $\alp([w]_{\mathrel{I}}) = \alp(w)$ and $|[w]_{\mathrel{I}}| = |w|$.
The {\em dependence alphabet} associated with $(\Sigma,I)$ is 
$(\Sigma,D)$, where $D=(\Sigma\times \Sigma)\setminus I$. 
Note that the relation $D$ is reflexive.
For $a\in \Sigma$ let $I(a)=\{b\in\Sigma\mid (a,b)\in I\}$ be the letters that commute with $a$.
For traces $u,v\in \dM(\Sigma,I)$ we denote with
$u \mathrel{I} v$
the fact that $\alp(u)\times \alp(v)\subseteq I$. The trace $u$ is {\em connected}
if we cannot write $u = v w$ in $\dM(\Sigma,I)$ such that $v \neq 1 \neq w$
and $v \mathrel{I} w$.

An {\em independence clique} is a subset 
$\Delta \subseteq \Sigma$ such that $(a,b) \in I$ for all $a,b \in \Delta$
with $a \neq b$. A {\em finite}
independence clique $\Delta$ is identified with the trace
$[a_1 a_2 \cdots a_n]_{\mathrel{I}}$, where $a_1, a_2, \ldots, a_n$
is an arbitrary enumeration of $\Delta$. 

The following lemma, which is known as Levi's lemma, is one of the most fundamental 
facts for trace monoids, see e.g.~\cite{DieRoz95}. 
\begin{lemma} \label{lemma-levi}
Let $u_1, \ldots, u_m, v_1, \ldots, v_n\in \dM(\Sigma,I)$. Then 
\[ u_1u_2 \cdots u_m = v_1 v_2 \cdots  v_n\] if and only if
there exist $w_{i,j} \in \dM(\Sigma,I)$ $(1 \leq i \leq m$, $1 \leq j \leq n)$ such that
\begin{itemize}
\item $u_i = w_{i,1}w_{i,2}\cdots w_{i,n}$ for every $1 \leq i \leq m$,
\item $v_j = w_{1,j}w_{2,j}\cdots w_{m,j}$ for every $1 \leq j \leq n$, and
\item $(w_{i,j}, w_{k,\ell})\in I$ if $1 \leq i < k \leq m$ and $n \geq j > \ell \geq 1$.
\end{itemize}
\end{lemma}
\noindent
The situation in the lemma will be visualized by a 
diagram of the following kind. The $i$-th column
corresponds to $u_i$, the $j$-th row
corresponds to $v_j$, and the intersection of the $i$-th
column and the $j$-th row represents $w_{i,j}$.
Furthermore $w_{i,j}$ and $w_{k,\ell}$ are independent
if one of them is left-above the other one. 

\medskip

\begin{center}
  \begin{tabular}{c"c|c|c|c|c|}\hline
  $v_n$  & $w_{1,n}$ & $w_{2,n}$ & $w_{3,n}$ & \dots  & $w_{m,n}$ \\ \hline
  \vdots & \vdots    & \vdots    & \vdots    & \vdots & \vdots    \\ \hline
  $v_3$  & $w_{1,3}$ & $w_{2,3}$ & $w_{3,3}$ & \dots  & $w_{m,3}$ \\ \hline
  $v_2$  & $w_{1,2}$ & $w_{2,2}$ & $w_{3,2}$ & \dots  & $w_{m,2}$ \\ \hline
  $v_1$  & $w_{1,1}$ & $w_{2,1}$ & $w_{3,1}$ & \dots  & $w_{m,1}$ \\ \thickhline
         & $u_1$     & $u_2$     & $u_3$     & \dots  & $u_m$
  \end{tabular} 
\end{center}

\medskip
\noindent
A consequence of Levi's lemma is that
trace monoids are cancellative, i.e., $usv=utv$ implies $s=t$ for all
traces $s,t,u,v\in\dM(\Sigma,I)$. 

A \emph{trace rewriting system} $R$ over $\mathbb{M}(\Sigma,I)$ is
just a subset of $\mathbb{M}(\Sigma,I) \times
\mathbb{M}(\Sigma,I)$ \cite{Die90lncs}. We define the
\emph{one-step rewrite relation} $\to_R \;\subseteq
\mathbb{M}(\Sigma,I) \times \mathbb{M}(\Sigma,I)$ by: $x \to_R y$ if
and only if there are $u,v \in \mathbb{M}(\Sigma,I)$ and $(\ell,r) \in
R$ such that $x = u\ell v$ and $y = u r v$. With $\xrightarrow{*}_R$
we denote the reflexive transitive closure of $\rightarrow_R$. The
notion of a confluent and terminating trace rewriting system is
defined as for other types of rewriting systems \cite{BoOt93}: A trace
rewriting system $R$ is called \emph{confluent} if for all $u,v,v'\in
\mathbb{M}(\Sigma,I)$ with $u \xrightarrow{*}_R v$ and
$u\xrightarrow{*}_R v'$ there exists a trace $w$ with $v
\xrightarrow{*}_R w$ and $v'\xrightarrow{*}_R w$.  It is called
\emph{terminating} if there does not exist an infinite chain
$u_0\rightarrow_R u_1 \rightarrow_R u_2 \cdots$.  A trace $u$ is
\emph{$R$-irreducible} if no trace $v$ with $u \to_R v$ exists. The
set of all $R$-irreducible traces is denoted
with $\IRR(R)$. If $R$ is terminating and confluent, then for every
trace $u$, there exists a unique \emph{normal form} $\NF_R(u) \in
\IRR(R)$ such that $u \xrightarrow{*}_R \NF_R(u)$~\cite{Huet1980}.

\subsection{Graph products}
\label{subsect:graph products}

Let us fix a {\em finite} 
independence alphabet $(\Gamma,E)$ 
and finitely generated groups ${\dG_i}$ for $i\in \Gamma$. 
Let $\alpha$ be the size of a largest clique of the independence alphabet $(\Gamma, E)$.
As usual $1$ is the identity element for each of the groups $\dG_i$.
Let $\Sigma_i$ be a finite and symmetric generating set of $\dG_i$ such that 
$\Sigma_i \cap \Sigma_j = \emptyset$ for $i\neq j$.
We define a (possibly infinite) 
independence alphabet as in \cite{DiLo08IJAC,KuLo05ijac}: Let
$$
A_i={\dG_i}\setminus \{1\} \quad\text{and}\quad A =
\bigcup_{i \in \Gamma} A_i.
$$ 
We assume that $A_i \cap A_j = \emptyset$ for $i\neq j$. We fix the independence relation 
$$
I =  \bigcup_{(i,j)\in E}A_i\times A_j
$$ 
on $A$. The independence alphabet $(A,I)$ is the only independence alphabet in this paper which may be infinite.
Recall that for a trace $t \in \dM(A,I)$, $\alp(t) \subseteq A$ is the set of symbols that occur in $t$.
We define the {\em $\Gamma$-alphabet} of $t$ as 
$$\alp_\Gamma(t) = \{ i \in \Gamma \mid \alp(t) \cap A_i \neq \emptyset \}.
$$
Note that whether $u \mathrel{I} v$ (for $u,v \in \dM(A,I)$) only depends on $\alp_{\Gamma}(u)$ and $\alp_{\Gamma}(v)$.

Every independence clique of $(A,I)$ 
has size at most $\alpha$ and hence can be identified with a trace from $\dM(A,I)$. 
Let $C_1$ and $C_2$ be independence cliques. We say that 
$C_1$ and $C_2$ are \emph{compatible}, if $\alp_{\Gamma}(C_1) = \alp_{\Gamma}(C_2)$.
In this case we can write $C_1 = \{ a_1, \ldots, a_m \}$ and $C_2 = \{ b_1, \ldots, b_m \}$
for some $m \leq \alpha$ such that for every $1 \leq i \leq m$ there exists $j_i \in \Gamma$ with $a_i, b_i \in A_{j_i}$. 
Let $c_i = a_i b_i$ in  the group $\dG_{j_i}$. If $c_i \neq 1$ for all $1 \leq i \leq m$, then $C_1$ and $C_2$ are 
{\em strongly compatible}. In this case we define the independence clique $C_1 C_2 = \{ c_1, \ldots, c_m\}$.
Note that $\alp_{\Gamma}(C_1) = \alp_{\Gamma}(C_2) = \alp_{\Gamma}(C_1C_2)$.

We will work with traces $t \in \dM(A,I)$. For such a trace we need two length measures. The ordinary length of $t$ is $|t|$ as defined in
Section~\ref{sec-traces}: If $t = [a_1 \cdots a_k]_{\mathrel{I}}$ with $a_j \in A$ then $|t|=k$.
On the other hand, if we deal with computational problems, we need a finitary representations of the elements $a_j$.
Assume that $a_j \in A_{i_j}$. Then, $a_j$ can be written as a word over the alphabet $\Sigma_{i_j}$. Let $n_j=\norm{a_j}$
denote the length of a shortest word over $\Sigma_{i_j}$ that evaluates to $a_j$ in the group $\dG_{i_j}$; 
this is also called the {\em geodesic length}
of the group element $a_j$. Then we define 
$\norm{t} = n_1 + n_2 + \cdots + n_k$.

A trace $a \in A$ (i.e., a generator of $\dM(A,I)$) is also called {\em atomic}, or an {\em atom}.
For an atom $a \in A$ that belongs to the group $\dG_i$, we write $a^{-1}$ for the inverse of $a$ in $\dG_i$;
it is again an atom.
On $\dM(A,I)$ we define the trace rewriting system
\begin{equation}\label{system-R}
R=\bigcup_{i\in \Gamma} \bigg(\{([aa^{-1}]_{\mathrel{I}},1)\mid a\in A_i\} \cup
\{([ab]_{\mathrel{I}},[c]_{\mathrel{I}})\mid a,b,c\in A_i,ab=c\;\mathrm{in}\;\dG_i\} \bigg).
\end{equation}
The following lemma was shown in \cite{KuLo05ijac}:

\begin{lemma}\label{R-confluent}
The trace rewriting system $R$ is confluent.
\end{lemma}
Since $R$ is length-reducing, it is also terminating and hence defines 
unique normal forms. We define the {\em graph product} 
$\dG(\Gamma,E,({\dG_i})_{i\in \Gamma})$ as the quotient monoid 
$$\dG(\Gamma,E,({\dG_i})_{i\in \Gamma}) = \dM(A,I)/R.$$
Here we identify $R$ with the smallest congruence relation on $\dM(A,I)$ that contains 
all pairs from $R$.
In the rest of Section~\ref{sec-part1}, we write $\dG$ for $\dG(\Gamma,E,({\dG_i})_{i\in \Gamma})$.
It is easy to see that $\dG$ is a group. The \emph{inverse} of a trace $t= [a_1 a_2 \cdots a_k]_{\mathrel{I}} \in \dM(A,I)$ with $a_i \in A$ is the trace 
$t^{-1} = [a_k^{-1} \cdots a_2^{-1} a_1^{-1}]_{\mathrel{I}}$.
Note that $t$ is well defined: If $[a_1 a_2 \cdots a_k]_{\mathrel{I}}  = [b_1 b_2 \cdots b_k]_{\mathrel{I}}$ then $[a_k^{-1} \cdots a_2^{-1} a_1^{-1}]_{\mathrel{I}} = [b_k^{-1} \cdots b_2^{-1} b_1^{-1}]_{\mathrel{I}}$.
We can apply this notation also to an independence clique $C$ of $(A,I)$ which yields the independence clique 
$C^{-1} = \{ a^{-1} \mid a \in C\}$.

Note that $\dG$ is finitely generated by $\Sigma = \bigcup_{i\in \Gamma}\Sigma_i$.
If $E = \emptyset$, then $\dG$ is the free product of the groups $\dG_i$ ($i \in \Gamma$) and if $(\Gamma,E)$ is a complete graph,
then $\dG$ is the direct product of the groups $\dG_i$ ($i \in \Gamma$). In this sense, the graph product
construction generalizes free and direct products.

For traces $u,v \in \dM(A,I)$ or words $u,v \in \Sigma^*$ we write $u =_{\dG} v$ if $u$ and $v$
represent the same element of the group $\dG$. In this case, we also say that $u=v$ \emph{in $\dG$}.
The following lemma is important for solving the word problem in
a graph product $\dG$:

\begin{lemma}\label{rewriting=product}
Let $u,v\in \dM(A,I)$. Then
$u=_{\dG} v$ if and only if $\NF_R(u)=\NF_R(v)$. In particular
we have $u=_{\dG} 1$ if and only if $\NF_R(u)=1$.
\end{lemma}

\begin{proof}
The if-direction is trivial. Let on the other hand $u,v\in \dM(A,I)$ and
suppose that $u=v$ in $\dG$. By definition this is the case if and only if
$u$ and $v$ represent the same element from $\dM(A,I)/R$ and
are hence congruent with respect to $R$. Since $R$ produces a normal form for elements from
the same congruence class, this implies that $\NF_R(u)=\NF_R(v)$.
\end{proof}

Graph products of copies of $\mathbb{Z}$ are also known as \emph{graph groups} or \emph{right-angled Artin groups}.
Graph products of copies of  $\mathbb{Z}/2\mathbb{Z}$ are known
as {\em right-angled Coxeter groups},  see \cite{Dro2003} for more details. 

For the rest of the paper we fix the graph product 
$\dG = \dG(\Gamma,E,({\dG_i})_{i\in \Gamma})$. Moreover, $\Sigma_i$, $A_i$ ($i \in \Gamma$), $\Sigma$, $A$, $I$, and $R$
will have the meaning defined in this section.

\subsection{Results from \cite{LohreyZ18}}

In this section we state a small modification of results from \cite{LohreyZ18},
where the statements are made for finitely generated trace monoids $\dM(\Sigma,I)$.
We need the corresponding statements for the non-finitely generated trace monoid
$\dM(A,I)$ from Section~\ref{subsect:graph products}. The proofs are exactly the same
as in \cite{LohreyZ18}, one only has to argue with the $\Gamma$-alphabet $\alp_{\Gamma}(t)$ instead
the alphabet $\alp(t)$ of traces.

Note that all statements in this section refer to the trace monoid $\dM(A,I)$ and not
to the corresponding graph product $\dG$. In particular, when we write a product
$t_1 t_2 \cdots t_n$ of traces $t_i \in \dM(A,I)$ no cancellation occurs between the $t_i$.
We will also consider the case that $E=\emptyset$ (and hence $I=\emptyset$), in which case 
$\dM(A,I) = A^*$.

Let $s,t \in \dM(A,I)$ be traces. We say that $s$ is a \emph{prefix} of $t$ if
there is a trace $r\in \dM(A,I)$ with $sr=t$. Moreover, we define $\rho(t)$ as
the number of prefixes of $t$.  We will use the following statement from
\cite{BeMaSa89}.

\begin{lemma} \label{lemma-prefixes}
Let $t \in \dM(A,I)$ be a trace of length $n$. Then $\rho(t) \le \mathcal{O}(n^\alpha) \le \mathcal{O}(n^{|\Gamma|})$, where $\alpha$
is the size of a largest clique of the independence alphabet $(\Gamma, E)$.
\end{lemma}

\begin{bem}
It is easy to see that $\rho(t)=n+1$ if $E=\emptyset$.
\end{bem}

\begin{lemma} \label{lemma-simplify-factorization-power}
Let $u  \in \dM(A,I) \setminus \{1\}$ be a connected trace and $m \in \mathbb{N}$, $m \geq 2$. 
Then, for all $x \in \mathbb{N}$ and traces $y_1, \ldots, y_m$ the following two statements are equivalent:
\begin{enumerate}[(i)]
\item $u^x = y_1 y_2 \cdots y_m$. 
\item  There exist traces $p_{i,j}$ $(1 \leq j < i \leq m)$, $s_i$ $(1 \leq i \leq m)$  and
numbers $x_i, c_j \in \mathbb{N}$ $(1 \leq i \leq m$, $1 \leq j \leq m-1)$
such that:
\begin{itemize}
\item $y_i = (\prod_{j=1}^{i-1} p_{i,j})  u^{x_i} s_i$ for all $1 \leq i \leq m$,
\item $p_{i,j} \mathrel{I} p_{k,\ell}$ if $j < \ell < k < i$ and  $p_{i,j} \mathrel{I} (u^{x_k} s_k)$ if $j <  k < i$,\footnote{Note 
that since $\alp(p_{i,j}) \subseteq \alp(u)$,
we must have $p_{i,j}=1$ or $x_k=0$ whenever $j <  k < i$.}
\item $s_m = 1$ and for all $1 \leq j < m$, $s_j \prod_{i=j+1}^m p_{i,j} = u^{c_j}$,
\item $c_j \leq |\Gamma|$ for all $1 \leq j \leq m-1$,
\item $x = \sum_{i=1}^m x_i + \sum_{i=1}^{m-1} c_i$.
\end{itemize}
\end{enumerate}
\end{lemma}

Note that this implies $\alp_\Gamma(p_{i,j}) \cup \alp_\Gamma(s_i) \subseteq \alp_\Gamma(u)$ for $1\le j<i\le m$.

The proof of Lemma~\ref{lemma-simplify-factorization-power} is the same as for \cite[Lemma~3.3]{LohreyZ18}, where the statement is shown for the case
of a finite independence alphabet $(A,I)$. In our situation the independency between traces only depends on
their $\Gamma$-alphabets. This allows to carry over the proof of \cite[Lemma~3.3]{LohreyZ18}
to our situation by replacing the alphabet $\alp(t)$ of a trace $t \in \dM(A,I)$ by 
$\alp_{\Gamma}(u)$.

\begin{bem} \label{remark-simplify-factorization-power}
In Section~\ref{sec-main}
we will apply Lemma~\ref{lemma-simplify-factorization-power} in order to replace an equation $u^x = y_1 y_2 \cdots y_m$
(where $x,y_1, \ldots, y_m$ are variables and $u$ is a concrete connected trace) by an equivalent disjunction. Note that the length of all factors $p_{i,j}$ and $s_i$ in Lemma~\ref{lemma-simplify-factorization-power}
is bounded by $|\Gamma| \cdot |u|$ and that $p_{i,j}$ and $s_i$ only contain symbols from $u$.
Hence, one can guess these traces as well 
as the numbers $c_j \leq |\Gamma|$  (the guess results in a disjunction). We can also 
guess which of the numbers $x_i$ are zero and which are greater than zero (let $K$ consists of those
$i$ such that $x_i>0$). After these guesses we can verify the independencies
$p_{i,j} \mathrel{I} p_{k,\ell}$ ($j < \ell < k < i$) and  $p_{i,j} \mathrel{I} (u^{x_k} s_k)$ ($j <  k < i$), and the identities
$s_m = 1$, $s_j \prod_{i=j+1}^m p_{i,j} = u^{c_j}$ ($1 \leq j < m$). If one of them does not hold, the specific guess does
not contribute to the disjunction. In this way,
we can replace the equation $u^x = y_1 y_2 \cdots y_m$ by a disjunction of formulas of the form
$$
\exists x_i > 0 \; (i \in K):
x = \sum_{i\in K} x_i + c \wedge \bigwedge_{i \in K} y_i = p_i  u^{x_i} s_i  \wedge \bigwedge_{i \in [1,m] \setminus K}  y_i = p_i s_i ,
$$
where $K \subseteq [1,m]$,
$c \leq |\Gamma| \cdot (m-1)$ and the $p_i, s_i$ are concrete traces of length 
at most  $|\Gamma| \cdot (m-1) \cdot |u|$. The number of disjuncts in the disjunction will not be important for our purpose.
\end{bem}

\begin{lemma} \label{lemma-connected-star}
Let $p,q,u,v,s,t \in \dM(A,I)$ with $u\neq 1$ and $v \neq 1$ connected.  Let $m = \max\{ \rho(p), \rho(q), \rho(s), \rho(t) \}$
and $n = \max \{ \rho(u), \rho(v) \}$.
Then the set 
$$L(p,u,s,q,v,t) = \{ (x,y) \in \mathbb{N} \times \mathbb{N} \mid p u^x s = q v^y t \}$$
is a union of 
$\mathcal{O}(m^8 \cdot n^{4 |\Gamma|})$ many
linear sets of the form $\{ (a+bz, c+dz) \mid z \in \mathbb{N} \}$ with
$a,b,c,d \le \mathcal{O}(m^8 \cdot n^{4 |\Gamma|})$.  In particular, $L(p,u,s,q,v,t)$ is semilinear.
If $|\Gamma|$ is a fixed constant, then a  semilinear representation for $L(p,u,s,q,v,t)$ can be computed
in polynomial time.
\end{lemma}
Again, the proof of Lemma~\ref{lemma-connected-star} is exactly the same as the proof of \cite[Lemma 3.8]{LohreyZ18}. 
One simply substitutes $|A|$ by $|\Gamma|$ and $\alp(x)$ by $\alp_{\Gamma}(x)$.

\begin{bem} \label{rem-2dim-words}
Let us consider again the case $E = I = \emptyset$ in Lemma~\ref{lemma-connected-star}.
Let $m = \max\{ |p|, |q|, |s|, |t|, |u|, |v| \}$. We can construct an automaton accepting $p u^* s$ of size at most $3m$
and similarly for $q v^* t$. Hence, we obtain an automaton of size $\mathcal{O}(m^2)$ accepting the language $L = p u^* s \cap q v^* t$.
We are only interested in the length of words from $L$. Let $\mathcal{A}$ be the automaton obtained
from the automaton for $L$ by replacing every transition label by the symbol $a$. The resulting automaton
$\mathcal{A}$ is defined over a unary alphabet. Let $P = \{ n \mid a^n \in L(\mathcal{A}) \}$.
By \cite[Theorem~1]{To09ipl}, the set $P$ can be written as a union 
$$
P = \bigcup_{i=1}^r \{  b_i + c_i \cdot z \mid z \in \mathbb{N} \}
$$
with $r \in \mathcal{O}(m^4)$ and $b_i, c_i \in \mathcal{O}(m^4)$. 
For every $1 \leq i \leq r$ and $z \in \mathbb{N}$ there must exist a pair $(x,y) \in \mathbb{N} \times \mathbb{N}$
such that 
$$
b_i + c_i \cdot z = |ps| + |u| \cdot x = |qt| + |v| \cdot y.
$$
In particular, $b_i \geq |ps|$, $b_i \geq |qt|$, $|u|$ divides $b_i-|ps|$ and $c_i$,
and $|v|$ divides $b_i-|qt|$ and $c_i$.
We get
\begin{multline*}
L(p,u,s,q,v,t) =  \bigcup_{i=1}^r \bigg\{ \bigg( \frac{b_i-|ps|}{|u|} + \frac{c_i}{|u|} \cdot z, \frac{b_i-|qt|}{|v|} + \frac{c_i}{|v|} \cdot z \bigg) \; \bigg| \bigg. \; z \in \mathbb{N}\bigg\}
\end{multline*}
and all numbers that appear on the right-hand side are bounded by $\mathcal{O}(m^4)$.
\end{bem}

\subsection{Irreducible powers in graph products}

In this section, we study powers $u^n$ for an irreducible trace $u \in \IRR(R)$.
We need the following definitions:
A trace $u \in \dM(A,I)$ is called \emph{cyclically reduced} if 
$u \in \IRR(R)$ and there do not exist $a \in A$ and $v \in \dM(A,I)$ such that $u = a v a^{-1}$.
A trace $t \in \dM(A,I)$ is called {\em well-behaved} if it is connected and $t^m \in \IRR(R)$ 
for every $m \geq 0$. 

\begin{lemma} \label{lemma-square}
Let $u \in \IRR(R)$. If $u^2 \in \IRR(R)$ then $u^m \in \IRR(R)$ for all $m \ge 0$.
\end{lemma}

\begin{proof}
Assume that $m\ge 3$ is the smallest number, such that $u^{m-1} \in \IRR(R)$ and $u^m \not\in \IRR(R)$. Hence we can write $u^m=xaby$ with $x,y \in \IRR(R)$ and $a,b \in A_i$ for some $i \in \Gamma$. Applying Levi's lemma, we get factorizations $x=x_1x_2 \cdots x_m$ and $y=y_1y_2 \cdots y_m$ and the following diagram:
\begin{center}
  \begin{tabular}{c"c|c|c|c|c|}\hline
  $y$  & $y_1$ & $y_2$  & \dots  & $y_{m-1}$ & $y_m$ \\ \hline
  $b$  &            &             & \dots  &                   & $b$ \\ \hline
  $a$  & $a$     &			 & \dots  & 				 & \\ \hline
  $x$  & $x_1$ & $x_2$  & \dots  & $x_{m-1}$ & $x_m$ \\ \thickhline
          & $u$     & $u$     & \dots  & $u$ 			 & $u$
  \end{tabular}
\end{center}
This is in fact the only possibility for the positions of the atoms $a$ and $b$: If $a$ and $b$ were in the same column
then $u$ would contain the factor $ab$ and hence $u \notin \IRR(u)$.
Also $a$ and $b$ are not independent, which means $b$ has to be top-right from $a$. If $a$ is not in the first column or $b$ is not in the last column, then $u^{m-1}$ is reducible, which contradicts the choice of $m$.
Hence, we have $u = x_1 a y_1 = x_m b y_m$ with $a \mathrel{I} x_m$, $y_1 \mathrel{I} x_m$ and $b \mathrel{I} y_1$.
We get $u^2 = x_1 a y_1 x_m b y_m = x_1 a x_m y_1 b y_m = x_1 x_m a y_1 b y_m = x_1 x_m a b y_1 y_m$.
Hence $u^2 \not\in \IRR(R)$, which is a contradiction.
\end{proof}

\begin{lemma}\label{well-behaved-iff}
A trace $u \in \dM(A,I)$ is well-behaved if and only if it has the following properties:
\begin{itemize}
\item $u$ is irreducible,
\item $u$ is not atomic,
\item $u$ is connected, and
\item one cannot write $u$ as $u = a v b$ such that $a,b \in A_i$ for some $i \in \Gamma$
(in particular, $u$ is cyclically reduced).
\end{itemize}
\end{lemma}

\begin{proof}
Clearly, if one the four conditions in the lemma is not satisfied, then $u$ is not well-behaved.
Now assume that the four conditions hold for $u$. By Lemma~\ref{lemma-square}, it suffices to show that
$u^2 \in \IRR(R)$. Assume that $u^2=xaby$ with $a,b \in A_i$. Applying Levi's lemma, and using $u \in \IRR(R)$ and 
$(a,b) \notin I$,  we obtain the following diagram:
\begin{center}
  \begin{tabular}{c"c|c|}\hline
  $y$  & $y_1$ & $y_2$  \\ \hline
  $b$  &            &  $b$      \\ \hline
  $a$  & $a$     &			\\ \hline
  $x$  & $x_1$ & $x_2$  \\ \thickhline
          & $u$     & $u$
  \end{tabular}
\end{center}
From Levi's lemma we also get $b\mathrel{I} y_1$ and $a\mathrel{I} x_2$. But $a$ and $b$ are in the same group, hence $a\mathrel{I} y_1$ and $b\mathrel{I} x_2$ also hold. The first property implies
$u=va$ with $v=x_1y_1$ and the seconds property
gives us $u=bw$ with $w=x_2y_2$. Since $u$ is not atomic, we have $v\neq 1 \neq w$. Now we apply Levi's lemma to $va=bw$, which yields one of the following diagrams:
\begin{center}
  \begin{tabular}{c"c|c|}\hline
  $w$  & $v=w$ &  \\ \hline
  $b$  &            &  $a=b$      \\ \thickhline
          & $v$     & $a$
  \end{tabular}
  \qquad
  \begin{tabular}{c"c|c|}\hline
  $w$  & $w^\prime$ &  $a$\\ \hline
  $b$  &     $b$     &       \\ \thickhline
          & $v$     & $a$
  \end{tabular}
\end{center}
From the left diagram we get $a \mathrel{I} v$. Hence $u = va$ is not connected, which is a contradiction.
From the right diagram we get $u = va = b w' a$ for some trace $w^\prime$, which is a contradiction to our last assumption. This finally proves $u^2 \in \IRR(R)$, hence $u$ is well-behaved.
\end{proof}

\begin{lemma}\label{power-presentation}
From a trace $u \in \dM(A,I)$ one can compute traces $s,t, v_1, \ldots, v_k \in \IRR(R)$,
such that the following hold:
\begin{itemize}
\item every $v_i$ is either atomic or well-behaved,
\item $u^m =_{\dG} s  v_1^m \cdots v_k^{m} t$ for all $m \geq 0$,
\item $\norm{s} + \norm{t} + \sum_{i=1}^k \norm{v_i} \le 3 \norm{u}$,
\item  $k \leq \alpha$, where $\alpha$ is the size of a largest clique in $(\Gamma,E)$.
\end{itemize}
\end{lemma}

\begin{proof}
Let $u \in \dM(A,I)$. As an initial processing, we can replace every $u$ by $\NF_R(u) \in \IRR(R)$.
So we can assume that $u$ is already irreducible.
In the next step, we compute irreducible traces $s, w, t$, such that 
$u^m=_{\dG} sw^m t$ for all $m \ge 0$ and $w$ cannot be 
written as $w=aw' b$ with $a,b \in A_j$ for some $j \in \Gamma$. For this, we will inductively construct 
irreducible traces $s_i, u_i, t_i$ (with $0 \le i \le \ell$ for some $\ell$)
such that $u^m=_{\dG}s_{i} u_{i}^m t_{i}$ for all $m \ge 0$. Moreover, if $0 \le i < \ell$ then $|u_{i}| > |u_{i+1}|$.
We start with 
$u_0=u$ and $s_0=t_0=1$.
Assume that after $i$ steps we already found irreducible traces 
$s_{i}, u_{i}, t_{i}$ with $u^m=_{\dG} s_{i} u_{i}^m t_{i}$
for all $m \ge 0$. If $u_{i}$ cannot be written in the form $au'b$ with $a,b \in A_j$ for some $j$, then we are done.
Otherwise assume that $u_{i}=av_{i}b$ for some $a, b \in A_j$. Let $c \in A_j \cup \{ 1 \}$ such that
$c = ba$ in the group $\dG_j$. So we get $u_{i}^m=_{\dG} a(v_{i}c)^m a^{-1}$ for all $m \ge 0$. This means $u^m=_{\dG} (s_{i}a) (v_{i}c)^m (a^{-1}t_{i})$. Hence, we can set $u_{i+1} =v_ic$, $s_{i+1} = \NF_R(s_ia)$ and $t_{i+1} =\NF_R(a^{-1}t_i)$. Note that $|u_{i+1}| = |u_i|-1$,
$\norm{u_{i+1}} \leq \norm{u_i}$, $\norm{s_{i+1}} \leq \norm{s_i} + \norm{a}$, and 
$\norm{t_{i+1}} \leq \norm{t_i} + \norm{a}$.
This process is terminating after at most $|u|$ steps. Note 
also that each $u_{i+1}$ is irreducible. When our algorithm is terminating after step $\ell$, we set $v=u_\ell$, $s=s_\ell$ and $t=t_\ell$.
We have 
\begin{equation} \label{norm-bound}
\norm{s}, \norm{t}, \norm{v} \leq \norm{u} .
\end{equation}
Finally, we split $v$ into its connected components, i.e., we write $v=v_1 \cdots v_{k}$, where every $v_j$ is connected
and $v_i \mathrel{I} v_j$ for $i \neq j$. We obtain for every $m \ge 0$ the identity
$u^m =_{\dG} s v_{1}^m \cdots v_k^{m} t$
as described in the statement of the lemma. If a $v_j$ is not atomic then it cannot be written 
as $v_j=bv'_jc$ with $b, c \in A_i$ (otherwise the above reduction process would continue).
Thus Lemma~\ref{well-behaved-iff} implies that the non-atomic $v_j$ are well-behaved. 
Finally, we have $\sum_{i=1}^k \norm{v_i} = \norm{v}  \le \norm{u}$ by \eqref{norm-bound}.
\end{proof}

\begin{bem} \label{remark-power-presentation}
If $E=\emptyset$ then we must have $k = 1$ in Lemma~\ref{power-presentation} since $\alpha=1$.
Hence, we obtain $s,t,v$, where $v$ is either atomic or well-behaved, such that
$u^m = s v^m t$ for every $m \geq 0$ and $\norm{s} + \norm{v} + \norm{t}  \le 3 \norm{u}$.
\end{bem}

\subsection{Reductions to the empty trace}

For the normal form of the product of two $R$-irreducible traces we
have the following lemma, which was shown in \cite{DiLo08IJAC}
(equation (21) in the proof of Lemma 22) using a slightly different
notation.

\begin{lemma}\label{normalform}
Let $u,v \in \dM(A,I)$ be $R$-irreducible. 
Then there exist strongly compatible independence cliques $C, D$ and
unique factorizations $u = p C s$, $v = s^{-1} D t$
such that $\NF_R(uv) = p (CD) t$.
\end{lemma}

In the following, we consider tuples over $\IRR(R)$ of arbitrary length.
We identify tuples that can be obtained from each other by inserting/deleting $1$'s at
arbitrary positions. Clearly, every tuple is equivalent to a possibly empty tuple over $\IRR(R) \setminus \{1\}$.

\begin{defi} \label{def-reduction}
We define a reduction relation on tuples over $\IRR(R)$ of arbitrary length.
Take $u_1, u_2, \ldots, u_m \in \IRR(R)$. Then we have
\begin{itemize}
\item $(u_1, u_2, \ldots, u_m) \to (u_1, \ldots, u_{i-1}, u_{i+1}, u_i, u_{i+2}, \ldots, u_m)$  if 
$u_i \mathrel{I} u_{i+1}$ (a {\em swapping step}),
\item  $(u_1, u_2, \ldots, u_m) \to (u_1, \ldots, u_{i-1}, u_{i+2}, \ldots, u_m)$  if
$u_i = u_{i+1}^{-1}$ in $\dM(A,I)$ (a {\em cancellation step}),
\item $(u_1, u_2, \ldots, u_m) \to (u_1, \ldots, u_{i-1}, a, u_{i+2}, \ldots, u_m)$ if
there exists $j \in \Gamma$ with $u_i, u_{i+1}, a \in A_j$, and $a = u_i u_{i+1}$ in $G_j$
(an {\em atom creation step of type $j$}).
\end{itemize}
Moreover, these are the only reduction steps. 
A concrete sequence of these rewrite steps leading to the empty tuple
is a {\em reduction} of $(u_1, u_2, \ldots, u_m)$. 
If such a sequence exists, the tuple is called {\em $1$-reducible}.
\end{defi}
A reduction of the tuple $(u_1, u_2, \ldots, u_m)$ can be seen as
a witness for the fact that $u_1 u_2 \cdots u_m =_{\dG} 1$. 
On the other hand, $u_1 u_2 \cdots u_m =_{\dG}1$ does not necessarily imply
that $(u_1, u_2, \ldots, u_m)$ has a reduction. For instance, the tuple $(a^{-1}, ab, b^{-1})$
has no reduction. But we can show that every sequence which multiplies to $1$ in $\dG$ can be refined
(by factorizing the elements of the sequence) such that the resulting refined sequence has a reduction.
We say that the tuple $(v_1, v_2, \ldots, v_n)$ is a {\em refinement} of the tuple 
$(u_1, u_2, \ldots, u_m)$ if there exists factorization $u_i = u_{i,1}  \cdots u_{i,k_i}$ in $\dM(A,I)$ such that
$(v_1, v_2, \ldots, v_n) = (u_{1,1}, \ldots, u_{1,k_1}, \; u_{2,1}, \ldots, u_{2,k_2}, \; \ldots, u_{m,1}, \ldots, u_{m,k_m})$.
In the following, if an independence clique $C$ appears in a tuple over $\IRR(R)$, we identify
this clique with the sequence $a_1, a_2, \ldots, a_n$ which is obtained by enumerating the elements of 
$C$ in an arbitrary way. For instance, $( [abcd]_{\mathrel{I}}, \{a,b,c\})$ stands for the tuple $([abcd]_{\mathrel{I}}, a,b,c)$.
Let us first prove the following lemma:

\begin{lemma} \label{lemma-aux-eps-red}
Assume that the tuple $(v_1, v_2, \ldots, v_n)$ is $1$-reducible with at most $m$ atom creations of each type. For all $1 \leq i \leq n$ 
let $v_i = p_i D_i t_i$ be a factorization in $\dM(A,I)$ where $D_i$ is an independence clique of $(A,I)$.
By refining $p_1, t_1, \ldots, p_n, t_n$ into a total of at most $4n + \sum_{i=1}^n |D_i|$ traces, we can obtain
a refinement of $(p_1, D_1, t_1, p_2, D_2, t_2, \ldots, p_n, D_n, t_n)$ which is $1$-reducible with at most $m$ atom creations of each type.
\end{lemma}

\begin{proof}
Basically, we would like to apply to  $(p_1, D_1, t_1, p_2, D_2, t_2, \ldots, p_n, D_n, t_n)$ 
the same reduction that reduces $(v_1, v_2, \ldots, v_n)$ to the empty tuple. If we do a swapping
step $v_i, v_j \to v_j, v_i$ then we can swap also the order of $p_i, D_i, t_i$ and $p_j, D_j, t_j$ in several
swapping steps. Also notice that if $v_i$ is an atom, then the subsequence $p_i, D_i, t_i$ is equivalent to the atom $v_i$.
The only remaining problem are cancellation steps. Assume that $v_i$ and $v_j$ cancel, i.e., $v_i = v_j^{-1}$.
The traces $t_i$ and $t_j$ do not necessarily cancel out, and similarly for $p_i$ and $p_j$ and the atoms in $D_i$ and $D_j$.
Hence, we have to further refine $p_i, t_i, p_j, t_j$ using Levi's lemma. Applied to the identity 
$p_i D_i t_i = t_j^{-1} D_j^{-1} p_j^{-1}$ it yields the following diagram:
\begin{equation}  \label{diagramm-aux}
  \begin{tabular}{c"c|c|c|}\hline 
  $p_j^{-1}$  & $x_{i,j}$   &  $N_{i,j}$ & $z_{i,j}$  \\ \hline 
  $D_j^{-1}$ &  $W_{i,j}$ & $C_{i,j}$ & $E_{i,j}$ \\ \hline
  $t_j^{-1}$   & $w_{i,j}$  & $S_{i,j}$ & $y_{i,j}$  \\ \thickhline 
                    & $p_i$       & $D_i$     & $t_i$
   \end{tabular}
\end{equation}
Hence, we get factorizations 
\begin{eqnarray}
p_i &=& w_{i,j}  W_{i,j}  x_{i,j}  \label{refine-pi} \\
t_i &=& y_{i,j}  E_{i,j} z_{i,j} \label{refine-ti} \\
p_j &=& z_{i,j}^{-1}  N_{i,j}^{-1}  x_{i,j}^{-1} \label{refine-pj} \\
t_j &=& y_{i,j}^{-1}  S_{i,j}^{-1}  w_{i,j}^{-1} \label{refine-tj} .
\end{eqnarray}
where $D_i = S_{i,j} \uplus C_{i,j} \uplus N_{i,j}$ and $D_j = E_{i,j}^{-1} \uplus C_{i,j}^{-1} \uplus W_{i,j}^{-1}$.
Using these facts and the independencies obtained from the diagram \eqref{diagramm-aux} shows that 
the tuple $$(w_{i,j}, W_{i,j}, x_{i,j}, D_i, y_{i,j}, E_{i,j}, z_{i,j}, z_{i,j}^{-1}, N_{i,j}^{-1}, x_{i,j}^{-1}, D_j, y_{i,j}^{-1}, S_{i,j}^{-1}, w_{i,j}^{-1})$$
is $1$-reducible.
Hence, by refining $p_i$, $t_i$, $p_j$, and $t_j$ according to the factorizations \eqref{refine-pi}, \eqref{refine-ti}, \eqref{refine-pj}, and \eqref{refine-tj}, respectively,
we obtain a $1$-reducible refinement of $(p_1, D_1, t_1, p_2, D_2, t_2, \ldots, p_n, D_n, t_n)$.
Note that $|W_{i,j} \cup E_{i,j}| \le |D_j|$ and $|N_{i,j} \cup S_{i,j}| \le |D_i|$. Hence, 
the $2n$ traces  $p_1, t_1, \ldots, p_n, t_n$ are refined into totally at most $4n + \sum_{i=1}^n |D_i|$ traces.
\end{proof}

As before, $\alpha$ denotes the size of a largest independence clique in $(A,I)$.
\begin{lemma} \label{lemma-reduction}
Let $m \geq 2$ and $u_1, u_2, \ldots, u_m \in \IRR(R)$. If $u_1 u_2 \cdots u_m = 1$ in
$\dG$, then there exists a $1$-reducible refinement of $(u_1, u_2, \ldots, u_m)$ that 
has length at most $(3\alpha+4)m^2 \leq 7 \alpha m^2$
and there is a reduction of that refinement with at most $m-2$ atom creations of each type $i \in \Gamma$.
\end{lemma} 

\begin{proof}
The proof of the lemma will be an induction on $m$. For this we first assume that $m$ is a power of 2. To make the induction work, we slightly strengthen the claim: We will show that there
exist factorizations of the $u_i$ with totally at most $f(m)=(\frac{3}{4}\alpha+1)m^2-(\frac{3}{2}\alpha+1)m$ factors such that the resulting
tuple is $1$-reducible and has a reduction with at most $(m-2)$ atom creations of each type $i \in \Gamma$.
This implies the lemma for the case that $m$ is a power of two.

The case $m=2$ is trivial (we must have $u_2 = u_1^{-1}$). Let $m = 2n \geq 4$. Then by Lemma~\ref{normalform} we can factorize 
$u_{2i-1}$ and $u_{2i}$ for $1\leq i \leq n$ as $u_{2i-1} = p_i C_{2i-1} s_i$ and $u_{2i} = s_i^{-1} C_{2i} t_i$ in $\dM(A,I)$ such that $C_{2i-1}$ and $C_{2i}$ are strongly compatible 
independence cliques and $v_i = p_i (C_{2i-1}C_{2i}) t_i$ is irreducible. Define the independence clique $D_i = C_{2i-1}C_{2i}$.
We have $v_1 v_2 \cdots v_n = 1$ in $\dG$. By induction, we obtain 
factorizations $p_i D_i t_i = v_i = v_{i,1}  \cdots v_{i,k_i}$ ($1\leq i \leq n$) such that the tuple
\begin{equation} \label{sequence-IH}
(v_{i,1}, \ldots, v_{i,k_i})_{1 \le i \leq n} 
\end{equation}
is $1$-reducible.  Moreover,
$$
\sum_{i=1}^{n} k_i \leq \left(\frac{3}{4}\alpha+1 \right)n^2 - \left(\frac{3}{2}\alpha +1\right)n
$$
and there exists a reduction of the tuple \eqref{sequence-IH} with at most $n-2$
atom creations of each type.
By applying Levi's lemma  to the trace identities $p_i D_i t_i = v_{i,1} v_{i,2} \cdots v_{i,k_i}$,
we obtain factorizations $v_{i,j} = x_{i,j} D_{i,j}  y_{i,j}$ in $\dM(A,I)$ such that
$D_i = \biguplus_{1 \leq j \leq k_i} D_{i,j}$, 
$p_i = x_{i,1} \cdots x_{i,k_i}$, $t_i = y_{i,1} \cdots u_{i,k_i}$, 
and the following independencies hold for $1 \leq j < \ell \leq k_i$:
$y_{i,j} \mathrel{I} x_{i,\ell}$, $y_{i,j} \mathrel{I} a$ for all $a \in D_{i,\ell}$, $a \mathrel{I} x_{i,\ell}$ for all $a \in D_{i,j}$.
Note that $D_{i,j}$ can be the empty set.

Let us now define for every $1 \leq i \leq n$ the tuples $\overline{u}_{2i-1}$ and $\overline{u}_{2i}$ as follows:
\begin{itemize}
\item $\overline{u}_{2i-1} = (x_{i,1}, \ldots, x_{i,k_i}, C_{2i-1}, s_i)$
\item $\overline{u}_{2i} = (s_i^{-1}, C_{2i}, y_{i,1}, \ldots, y_{i,k_i})$
\end{itemize}
Thus, the tuple $\overline{u}_i$ defines a factorization of the trace $u_i$ and the tuple $(\overline{u}_1, \overline{u}_2, \ldots, \overline{u}_{2n})$ is a refinement 
of $(u_1, \ldots, u_{2n})$ of length $2 f(n) + 2 n (\alpha + 1)$.
This tuple can be transformed using $n$ cancellation steps (cancelling $s_i$ and $s_i^{-1}$) and 
$n$ atom creations of each type into the sequence 
$$
(x_{i,1}, \ldots, x_{i,k_i}, D_i, y_{i,1}, \ldots, y_{i,k_i})_{1 \le i \le n} .
$$
Using swappings, we finally obtain the sequence
\begin{equation} \label{tuple-xDy}
(x_{i,1},D_{i,1},y_{i,1}, \ldots, x_{i,k_1},D_{i,k_1},y_{i,k_1})_{1 \le i \le n} .
\end{equation}
Recall that $v_{i,j} = x_{i,j} D_{i,j}  y_{i,j}$. Hence, the tuple \eqref{tuple-xDy} is a refinement of the $1$-reducible tuple \eqref{sequence-IH}.
We are therefore in the situation of Lemma~\ref{lemma-aux-eps-red}.
By further refining the totally at most $2 f(n)$ factors $x_{i,j}$ and $y_{i,j}$ of the traces $u_1, \ldots, u_{2n}$ we obtain a $1$-reducible tuple. 
The resulting refinement of $(u_1, \ldots, u_{2n})$ has length at most
\begin{eqnarray*}
 &&  4 \sum_{i=1}^{n} k_i  + \sum_{i=1}^n \sum_{j=1}^{k_i} |D_{i,j}|  +  2n + 2n\alpha \\
&\le & 4 \left(\frac{3}{4}\alpha+1 \right)n^2 - 4 \left(\frac{3}{2}\alpha+1\right)n + \sum_{i=1}^n  |D_i| +  2n + 2n\alpha \\
&\le&   (3\alpha+4) n^2 - (6\alpha+4)n + (3\alpha+2) n \\ 
&=&   (3\alpha+4) n^2 - (3\alpha+2) n \\
& = & \left(\frac{3}{4}\alpha+1\right) m^2 - \left(\frac{3}{2}\alpha+1\right) m
\end{eqnarray*}
($\sum_{i=1}^{n} k_i + \sum_{i=1}^n \sum_{j=1}^{k_i} |D_{i,j}|$ traces from the refinement of the traces $x_{i,j}$ and $y_{i,j}$ by Lemma~\ref{lemma-aux-eps-red},
$2n$ traces $s_i^{\pm 1}$, and $2n\alpha$ atoms from the independence cliques $C_i$).
Finally, the total number of atom creations of a certain type is $n + n-2 = 2n-2 = m-2$.

In the general case, where $m$ is not assumed to be a power of two, we can naturally extend the sequence to $u_1, u_2, \dots , u_{\ell}$ by possibly adding $u_i=1$ for $i>m$ to the smallest power of 2. Hence $\ell\leq 2m$. Substituting $2m$ for $m$ yields the desired bound. Note that by this process, the number of atom creations will not increase. This concludes the proof of the lemma.
\end{proof}
Since by this result we also get a $1$-reducible tuple with at most $\mathcal{O}(m^2)$ many elements for equations $u_1 u_2 \cdots u_m = 1$ over a graph group, this improves the result of \cite{LohreyZ18}.

\begin{bem} \label{rem-atom-creation}
The atom creations that appear in a concrete reduction can be collected into finitely many identities
of the form $a_1 a_2 \cdots a_k =_{\dG_i} b_1 b_2 \cdots b_\ell$ (or $a_1 a_2 \cdots a_k b_\ell^{-1} \cdots b_2^{-1} b_1^{-1} =_{\dG_i} 1$),
where $a_1, a_2, \ldots, a_k, b_1, b_2,\ldots,  b_\ell$ are atoms from the initial sequence that all belong to the same
group $\dG_i$. The new atoms
$a_1 a_2 \cdots a_k$ and $b_1 b_2 \cdots b_\ell$ are created by at most $m-2$ atom creations. Finally, the two resulting atoms
cancel out. Note that $k-1+\ell-1 \leq m-2$, i.e., $k+\ell \leq m$.
\end{bem}	
In case $E = I = \emptyset$ the quadratic dependence on $m$ in Lemma~\ref{lemma-reduction} can be avoided:

\begin{lemma}\label{lemma-reduction-free-product}
Let $m \geq 2$ and $u_1, u_2, \ldots, u_m \in \IRR(R)$. Moreover let $E=I=\emptyset$. If $u_1 u_2 \cdots u_m = 1$ in
the free product $\dG$, then there exists a $1$-reducible refinement of the tuple 
$(u_1, u_2, \ldots, u_m)$ that has length at most $7m-12$ and
there is a reduction of this refinement with at most $m-2$ atom creations.% of any type.
\end{lemma}

\begin{proof}
We prove the lemma by induction on $m$. 
The case $m=2$ is trivial (we must have $u_2 = u_1^{-1}$). If $m \geq 3$ then  for the normal form of $u_1 u_2$ there are two cases: either 
$u_1 u_2 \in \IRR(R)$ or $u_1=pas$ and $u_2=s^{-1}bt$
for atoms $a,b$ from the same group $\dG_i$ that do not cancel out.
We consider only the latter case. Let $c = ab$ in $\dG_i$, i.e., $c \in A_i$. 
By the induction hypothesis, the tuple $(pct, u_3, \dots , u_m)$ has a $1$-reducible refinement
\begin{equation} \label{eq-refinement-free}
(v_1, \dots , v_k, w_1, \dots , w_\ell)
\end{equation}
with $k+\ell \leq 7(m-1)-12$  and $pct=v_1 \cdots v_k$, where the latter is an identity between words from $A^*$. 
Moreover, there is a reduction of \eqref{eq-refinement-free} with at most $m-3$ atom creations.
Since $pct=v_1 \cdots v_k$,
one of the $v_j$ ($1\leq j\leq k$) must factorize as $v_j = v_{j,1} c v_{j,2}$ 
such that $p=v_1 \cdots v_{j-1}v_{j,1}$ and $t = v_{j,2} v_{j+1}\cdots v_k$, which implies
$u_1 = v_1 \cdots v_{j-1}v_{j,1} a s$ and $u_2 = s^{-1}b v_{j,2} v_{j+1}\cdots v_k$.
Therefore we have a $1$-reducible tuple of the form 
\begin{equation} \label{eq-refinement-free2}
(v_1, \dots , v_{j-1}, v_{j,1},a ,s , s^{-1}, b, v_{j,2}, v_{j+1}, \dots , v_k , \tilde{w}_1, \dots ,  \tilde{w}_\ell),
\end{equation}
where the sequence $\tilde{w}_i$ is $w_i$ unless $w_i$ cancels out with $v_j$ in our reduction
of \eqref{eq-refinement-free} (there can be only one such $i$), in which case $\tilde{w}_i$ is $(v_{j,2})^{-1}, c^{-1}, (v_{j,1})^{-1}$. 
It follows that the tuple \eqref{eq-refinement-free2} is a refinement of 
$(u_1, u_2, \ldots, u_m)$ with at most $7(m-1)-12+7=7m-12$ words, having a reduction with
at most $m-2$ atom creations.
\end{proof}

\subsection{Graph products preserve knapsack semilinearity}\label{sec-main}

In this section, we assume that every group $\dG_i$ ($i \in \Gamma$) is knapsack-semilinear.
Recall that we fixed the symmetric generating set $\Sigma_i$ for $\dG_i$, which yields the 
generating set $\Sigma = \bigcup_{i \in \Gamma} \Sigma_i$ for the graph product $\dG$.
In this section, we want to show that the graph product $\dG$ is knapsack-semilinear as well.
Moroever, we want to bound the function $\mathsf{E}_{\dG,\Sigma}$ in terms of the functions
$\mathsf{K}_{\dG_i,\Sigma_i}$. Let $\mathsf{K} : \mathbb{N} \times \mathbb{N} \to \mathbb{N}$ be the pointwise
maximum of the functions $\mathsf{K}_{\dG_i,\Sigma_i}(n,m)$. 
We will bound $\mathsf{E}_{\dG, \Sigma}$ in terms of $\mathsf{K}$.

Consider an exponent expression $\E = u_1^{x_1} v_1 u_2^{x_2} v_2 \cdots u_m^{x_m} v_m$, where 
$u_i, v_i$ are words over the generating set $\Sigma$.
Let $g_i$ (resp., $h_i$) be the element of $\dG$ represented by $u_i$ (resp., $v_i$).
We can assume that all $u_i$ and $v_i$ are geodesic
words in the graph product $\dG$.\footnote{Since the word problem for every $\dG_i$ is decidable, also the word
problem for $\dG$ is decidable~\cite{Gre90}, which implies that one can compute a geodesic word for a given
group element of $\dG$.} We will make this assumption throughout this section.
Moreover, we can identify each 
$u_i$ (resp., $v_i$) with the unique irreducible trace from $\IRR(R)$ that represents the group element $g_i$ (resp., $h_i$).
In addition, for each atom $a \in A$ (say $a \in A_i$) that occurs in one of the traces $u_1, u_2, \ldots, u_m, v_1, \ldots, v_m \in \IRR(R)$
a geodesic word $w_a \in \Sigma_i^*$ that evaluates to $a$ in the group $\dG_i$ is given.  This yields geodesic words for 
the group elements $g_1, \ldots, g_m, h_1, \ldots, h_m \in \dG$. The lengths of these words are 
$\norm{u_1}, \ldots, \norm{u_m}, \norm{v_1}, \ldots, \norm{v_m}$ and we have
$\norm{\E} = \norm{u_1} +  \cdots + \norm{u_m} + \norm{v_1} + \cdots + \norm{v_m}$.

We start with the following  preprocessing step.
\begin{lemma} \label{lemma-preproc}
Let $\E$ be an exponent expression over $\Sigma$. 
From $\E$ we can compute a knapsack expression $\E'$
with the following properties:
\begin{itemize}
\item $X_{\E} \subseteq X_{\E'}$,
\item $\norm{\E'} \le 3 \norm{\E}$,
\item $\deg(\E') \le \alpha \cdot \deg(\E)$,
\item every period of $\E'$ is either atomic or well-behaved, and
\item $\Sol_{\dG}(\E) = (K \cap  \Sol_{\dG}(\E')) \rest_{X_{\E}}$ for a semilinear set $K$ of magnitude one.
\end{itemize}
\end{lemma}

\begin{proof}
Let $u_1, \ldots, u_m \in \Sigma^*$ be the periods of $\E$.
We can view these words as traces $u_1, \ldots, u_m \in \dM(A,I)$ that are moreover irreducible.
We apply Lemma~\ref{power-presentation} to each power $u_i^{x}$ in $\E$ and obtain an equivalent 
exponent expression $\tilde\E$ of degree $n \le \alpha \cdot m$ and $\norm{\tilde\E} \leq 3\norm{\E}$.
%Let $v_1, \ldots, v_n$ be the periods of $\tilde\E$.
%Also note that  $\sum_{i=1}^{n} \norm{v_i} \le \sum_{i=1}^{m} \norm{u_i} \le \norm{\E}$ by Lemma~\ref{power-presentation}.
We have $X_{\tilde\E} = X_{\E}$ and  $\Sol_{\dG}(\E) = \Sol_{\dG}(\tilde\E)$.

We now rename in $\tilde\E$ the variables by fresh variables in such a way
that we obtain a knapsack expression $\E'$. Moreover, for every $x \in X_{\E}$ we keep exactly one 
occurrence of $x$ in $\tilde\E$ and do not rename this occurrence of $x$. This implies that there is a  
semilinear set $K \subseteq \mathbb{N}^{X_{e'}}$ of magnitude one such that
$\Sol_{\dG}(\E) = (K \cap \Sol_{\dG}(\E'))\rest_{X_{\E}}$. 
\end{proof}

In case $E = I = \emptyset$ and that $\E$ is a knapsack expression, we can simplify the statement of 
Lemma~\ref{lemma-preproc} as follows:

\begin{bem} \label{rem-preproc-free-product}
Assume that $E = I = \emptyset$ and that $\E$ is a knapsack expression as in Lemma~\ref{lemma-preproc}.
By Remark~\ref{remark-power-presentation} we can compute from $\E$ a 
knapsack expression $\E'$ over $\Sigma$ with the following properties:
\begin{itemize}
\item $\norm{\E'} \leq 3\norm{\E}$,
\item  $\deg(\E') \le \deg(\E)$,
\item every period of $\E'$ is either atomic or well-behaved, and
\item $\Sol_{\dG}(\E) = \Sol_{\dG}(\E')$.
\end{itemize} 
\end{bem}
We now come to the main technical result of Section~\ref{sec-part1}.
As before, we denote with $\alpha$ the size of a largest independence clique in $(\Gamma,E)$.

\begin{theorem} \label{thm-main-technical}
If each group $G_i$, $i\in\Gamma$, is knapsack-semilinear, then their graph product $G=G(\Gamma,E,(G_i)_{i\in\Gamma})$ is knapsack-semilinear as well.
Let $\mathsf{K} : \mathbb{N} \times \mathbb{N} \to \mathbb{N}$ be the pointwise
maximum of the functions $\mathsf{K}_{\dG_i,\Sigma_i}(n,m)$ for $i \in \Gamma$.
Then $\mathsf{E}_{\dG,\Sigma}(n,m) \leq \max\{\mathsf{K}_1 , \mathsf{K}_2\}$
with
\begin{align*}
\mathsf{K}_1 & \le \mathcal{O}\big( (\alpha m)^{\alpha m/2+3} \cdot  \mathsf{K}(6 \alpha m  n,\alpha m)^{\alpha m+3}\big),\\
\mathsf{K}_2 & \le (\alpha m)^{\mathcal{O}(\alpha^2 m)}  \cdot n^{\mathcal{O}(\alpha^2 |\Gamma| m)}.
\end{align*}
\end{theorem}

\begin{proof}
Consider an exponent expression 
$\E = u_1^{x_1} v_1 u_2^{x_2} v_2 \cdots u_m^{x_m} v_m$. 
Let us denote with $A(\E) = \alp(u_1 v_1 \cdots u_m v_m) \subseteq A$ the set of all atoms that appear in the 
traces $u_i,v_i$. Finally let $\mu(\E) = \max\{\norm{a} \mid a \in A(\E)\}$
and let $\lambda(\E)$ be the maximal length $|t|$ where $t$ is one of the traces $u_1, u_2, \ldots, u_m, v_1, \ldots, v_m$.
We clearly have $\mu(\E) \le \norm{\E}$ and $\lambda(\E) \le \norm{\E}$.

Let us first assume that $\E$ is a knapsack expression (i.e., $x_i \neq x_j$ for $i \neq j$) 
where every period $u_i$ is either an atom
or a well-behaved trace (see Lemma~\ref{lemma-preproc}).

In the following we describe an algorithm that computes a semilinear representation of $\Sol_{\dG}(\E)$
(for $\E$ satisfying the conditions from the previous paragraph). At the same time, we will compute the magnitude
of this semilinear representation. The algorithm transforms logical statements into equivalent logical statements
(we do not have to define the precise logical language; the meaning of the statements should be always clear).
Every statement contains the variables $x_1, \ldots, x_m$ from our knapsack expression and equivalence of two statements means
that for every valuation $\sigma : \{x_1, \ldots, x_m\} \to \N$ the two statements yield the same truth value.
We start with the statement $\E = 1$. In each step we transform the current
statement $\Phi$ into an equivalent disjunction $\bigvee_{i=1}^n \Phi_i$. We can therefore view the whole process
as a branching tree, where the nodes are labelled with statements. If a node is labelled with $\Phi$ and its children are labelled
with $\Phi_1, \ldots, \Phi_n$ then $\Phi$ is equivalent to $\bigvee_{i=1}^n \Phi_i$. The leaves of the tree are labelled with 
semilinear constraints of the form $(x_1, \ldots, x_m) \in L$ for semilinear sets $L$. Hence, the solution set
$\Sol_{\dG}(\E)$ is the union of all semilinear sets that label the leaves of the tree. A bound on the magnitude 
of these semilinear sets yields a bound on the magnitude 
of $\Sol_{\dG}(\E)$. Therefore, we can restrict our analysis to a single branch of the tree.
We can view this branch as a sequence of nondeterministic guesses. Some guesses lead
to dead branches because the corresponding statement is unsatisfiable. We will speak
of a bad guess in such a situation.

Let $N_a \subseteq [1,m]$ be the set of indices such that $u_i$ is atomic and 
let $N_{\overline{a}} = [1,m] \setminus N_a$ be the set of indices such that $u_i$
is not atomic (and hence a well-behaved trace). For better readability, we write $a_i$ for the atom $u_i$
in case $i \in N_a$.
Define $X_a = \{ x_i \mid i \in N_a\}$ and $X_{\overline{a}} = \{ x_i \mid i \in N_{\overline{a}} \}$.  For $i \in N_a$ let 
$\gamma(i) \in \Gamma$ be the index with $u_i \in A_{\gamma(i)}$.

\medskip
\noindent
{\em Step 1: Eliminating trivial powers.}
In a first step we guess a set $N_1 \subseteq N_a$ of indices with the meaning that for $i \in N_1$
the power $a_i^{x_i}$ evaluates to the identity element of the group $\dG_{\gamma(i)}$. To express this we 
continue with the formula
\begin{equation} \label{formula-Phi}
\Phi[N_1] = (e[N_1] = 1) \wedge \bigwedge_{i \in N_1} a_i^{x_i} =_{\dG_{\gamma(i)}} 1,
\end{equation}
where $e[N_1]$ is the knapsack expression obtained from $e$ by deleting all powers $a_i^{x_i}$ with $i \in N_1$.
Note that the above constraints do not exclude that a power $u_i^{x_i}$ with 
$i \in [1,m] \setminus N_1$ evaluates to the identity element.
This will not cause any trouble for the following arguments.
Clearly, the initial equation $e=1$ is equivalent to the formula
$\bigvee_{N_1 \subseteq N_a} \Phi[N_1]$.

In the following we transform every equation $e[N_1] = 1$ into a formula $\Psi[N_1]$ such that the following hold for every
valuation $\sigma : \{ x_i \mid i \in [1,m] \setminus N_1 \} \to \N$:
\begin{enumerate}[(1)]
\item if $\Psi[N_1]$ is true under $\sigma$ then $\sigma(e[N_1]) =_{\dG} 1$,
\item if $a_i^{\sigma(x_i)} \neq_{\dG_{\gamma(i)}} 1$ for all $i \in N_a \setminus N_1$ and $\sigma(e[N_1]) =_{\dG} 1$ then
$\Psi[N_1]$ is true under $\sigma$.
\end{enumerate}
This implies that $\bigvee_{N_1 \subseteq N_a} \Phi[N_1]$ (and hence $e=1$) is equivalent to the formula
$$
\bigvee_{N_1 \subseteq N_a} (\Psi[N_1] \wedge \bigwedge_{i \in N_1} a_i^{x_i} =_{\dG_{\gamma(i)}} 1).
$$
{\em Step 2: Applying Lemma~\ref{lemma-reduction}.}  
We construct the formula $\Psi[N_1]$ from the knapsack expression $e[N_1]$ using  Lemma~\ref{lemma-reduction}. More precisely, we construct 
$\Psi[N_1]$ by nondeterministically guessing the following data:
\begin{enumerate}[(i)]
\item factorizations $v_i = v_{i,1} \cdots v_{i,\ell_i}$ in $\dM(A,I)$ of all non-trivial traces $v_i$. Each factor $v_{i,j}$ must be nontrivial too.
\item ``symbolic factorizations'' $u_i^{x_i} = y_{i,1} \cdots y_{i,k_i}$ for all $i \in N_{\overline{a}}$. 
The numbers $k_i$ and $\ell_i$ must sum up to at most $28\alpha m^2$ (this number is obtained by replacing $m$ by $2m$
in Lemma~\ref{lemma-reduction}). The $y_{i,j}$ are existentially quantified variables that take values 
in $\IRR(R)$ and which will be eliminated later.
\item non-empty alphabets $A_{i,j} \subseteq \alp(u_i)$ for each symbolic factor $y_{i,j}$ ($i \in N_{\overline{a}}$, $1 \leq j \leq k_i$)
with the meaning that $A_{i,j}$ is the alphabet of atoms that appear in  $y_{i,j}$.
\item a reduction (according to Definition~\ref{def-reduction}) of the resulting refined factorization of 
$u_1^{x_1} v_1 u_2^{x_2} v_2 \cdots u_m^{x_m} v_m$ with at most $2m-2$ atom creations of each type $i \in \Gamma$.
\end{enumerate}
Note that every factor  $a_i^{x_i}$ with $i \in N_{a} \setminus N_1$ evaluates (for a given valuation) either to an atom
from $A_{\gamma(i)}$ or to the identity element. Hence, there is no need to further  factorize such a power $a_i^{x_i}$.
In our guessed reduction we treat $a_i^{x_i}$ as a symbolic atom (although it might happen that $a_i^{\sigma(x_i)} = 1$ 
 for a certain valuation $\sigma$; but this will not make the above statements (1) and (2) wrong). 

We can also guess $k_i = 0$ in (ii). In this case, we 
can replace $u_i^{x_i}$ in $e[N_1]$ by the empty trace and add the constraint $x_i = 0$
(note that a well-behaved trace $u_i \neq 1$ represents an element of the graph product $\dG$ 
without torsion). Hence, in the following
we can assume that the $k_i$ are not zero.
Some of the $y_{i,j}$ must be atoms since they take part in an 
atom creation in our guessed reduction. Such an $y_{i,j}$ is 
replaced by a nondeterministically guessed atom $a_{i,j}$ from the atoms in $u_i$. 

The guessed alphabetic constraints from (iii) must be consistent with the independencies from our guessed reduction in (iv).
This means that if for instance $y_{i,j}$ and $y_{k,\ell}$ are swapped in the reduction then we must have $A_{i,j} \times A_{j,k} \subseteq I$.
Here comes a subtle point: Recall that each power $a^{x_i}$ ($i \in N_a \setminus N_1$) evaluates for a given valuation either to an atom
from $A_{\gamma(i)}$ or to the identity element. When checking the consistency of the alphabetic constraints with the guessed reduction 
we make the (pessimistic) assumption that every $a^{x_i}$ evaluates to an atom from $A_{\gamma(i)}$. This is justified below.

For every specific guess in (i)--(iv)
we write down the existentially quantified conjunction of the following formulas:
\begin{itemize}
\item the equation $u_i^{x_i} = y_{i,1} \cdots y_{i,k_i}$ from (ii) (every trace-variable $y_{i,j}$ is existentially quantified),
\item all trace equations that result from cancellation steps in the guessed reduction,
\item all ``local'' identities that result from the atom creations in the guessed reduction,
\item all alphabetic constraints from (iii) and 
\item all constraints $x_i = 0$ in case we guessed $k_i = 0$ in (ii). 
\end{itemize}
The local identities in the third point involve the above atoms $a_{i,j}$ and the powers $a_i^{x_i}$ for 
$i \in N_a \setminus N_1$. According to Remark~\ref{rem-atom-creation} they are combined into
several knapsack expressions over the groups $\dG_i$.

The formula $\Psi[N_1]$ is the disjunction of the above existentially quantified conjunctions, taken over all possible guesses in (i)--(iv).
It is then clear that the above points (1) and (2) hold. Point (2) follows immediately from 
Lemma~\ref{lemma-reduction}. For point (1) note that each of the existentially quantified conjunctions in $\Psi[N_1]$ yields the identity $e[N_1]=1$, irrespective of 
whether a power $a_i^{x_i}$ is trivial or not.

So far, we have obtained a disjunction of existentially quantified conjunctions. Every conjunction involves the 
equations $u_i^{x_i} = y_{i,1} \cdots y_{i,k_i}$ from (ii), trace equations that result from cancellation steps (we will deal with them in step 4 below),
local knapsack expressions over the groups $\dG_i$, alphabetic constraints for the variables $y_{i,j}$ and constraints $x_i = 0$ (if $k_i = 0$).
In addition we have the identities $a_i^{x_i} =_{\dG_{\gamma(i)}} 1$ ($i \in N_1$) from \eqref{formula-Phi}. In the following we deal with a single
existentially quantified conjunction of this form.

\medskip
\noindent
{\em Step 3: Isolating the local knapsack instances for the groups $\dG_i$.}  
In our existentially quantified conjunction we have knapsack expressions $\E_1, \ldots, \E_q$ over the groups $\dG_j$ ($j \in \Gamma$).
These knapsack expressions involve the atoms $a_{i,j}$ and the symbolic expressions $a_i^{x_i}$ with $i \in N_a$.
Note that every identity $a_i^{x_i} =_{\dG_{\gamma(i)}} 1$ ($i \in N_1$) yields the knapsack expression 
$a_i^{x_i}$. Each of the expressions $\E_j$ is built from at most $2m$
atom powers $a_i^{x_i}$ and atoms $a_{i,j}$ (since for every $j \in \Gamma$ there are at most $2m-2$ atom creations
of type $j$) and its degree is at most $m$ (since there 
are at most $m$ atom powers $a_i^{x_i}$). All atoms $a_i$ and $a_{i,j}$ belong to 
$A(\E)$. This yields the bound $\norm{\E_j} \le 2m \mu(\E)$ for $1 \leq j \leq q$.
We can assume that each expression $\E_j$ contains at least one  atom power $u_i^{x_i}$
(identities between the explicit atoms $a_{i,j}$ can be directly verified; if they do not hold,
one gets a bad guess).
Moreover, note that every atom power $a_i^{x_i}$ with $i \in N_a$ occurs in exactly one $\E_j$. 
Assume that the knapsack expression $\E_j$ is defined over the group $H_j \in \{ \dG_i \mid i \in \Gamma\}$.
The solution sets $\Sol_j = \Sol_{H_j}(\E_j)$ of these expressions are semilinear by the assumption 
on the groups $\dG_i$. Each $\Sol_j$ has some dimension $d_j \le m$ (which is the number of symbolic
atoms in $\E_j$), where $\sum_{j=1}^q d_j = |N_a|$ and the magnitude of $\Sol_j$ is bounded by 
$\mathsf{K}(2m \mu(\E), m) \le \mathsf{K}(2m \norm{\E}, m)$. Finally, we can combine these sets $\Sol_j$ into the single semilinear set
$S_a = \bigoplus_{j=1}^q \Sol_j \subseteq \N^{X_a}$ of dimension $|N_a|$ and magnitude at most
$\mathsf{K}(2m \norm{\E},m)$. Recall that the sets $\Sol_j$ refer to pairwise 
disjoint sets of variables. 
For the variables $x_i \in X_a$ we now obtain the semilinear contraint $(x_i)_{i \in N_a} \in S_a$.

\medskip
\noindent
{\em Step 4: Reduction to two-dimensional knapsack instances.}
Let us now deal with the cancellation steps from our guessed reduction.
From these reduction steps we will produce two-dimensional knapsack instances
on pairwise disjoint variable sets.

If two explicit factors $v_{i,j}$ and 
$v_{k,\ell}$ (from (i) in step 2) cancel out in the reduction, we must have $v_{k,\ell} = v_{i,j}^{-1}$; otherwise our previous guess
was bad. If a symbolic factor $y_{i,j}$ and an explicit factor $v_{k,\ell}$ cancel out, then we can replace
$y_{i,j}$ by $v_{k,\ell}^{-1}$. Before doing this, we check whether $\alp(v_{k,\ell}^{-1}) = A_{i,j}$ and if this 
condition does not hold, then we obtain again a bad guess. Let $S$ be the set of pairs $(i,j)$ such that
the symbolic factor $y_{i,j}$ still exists after this step. On this set $S$ 
there must exist a matching $M \subseteq \{ (i,j,k,\ell) \mid (i,j), (k,\ell) \in S \}$ such that
$y_{i,j}$ and $y_{k,\ell}$ cancel out in our reduction if and only if $(i,j,k,\ell) \in M$.
We have $(i,j,k,\ell) \in M$ if and only if $(k,\ell,i,j) \in M$.

Let us write the new symbolic factorization of $u_i^{x_i}$ as 
$u_i^{x_i} = \tilde{y}_{i,1} \cdots \tilde{y}_{i,k_i}$, where every
$\tilde{y}_{i,j}$ is either the original symbolic factor $y_{i,j}$ (in case $(i,j) \in S$) or a concrete
trace $v_{k,\ell}^{-1}$ (in case $y_{i,j}$ and $v_{k,\ell}$ cancel out in our reduction) or an atom $a_{i,j} \in \alp(u_i)$ (that was guessed in step 2).
It remains to describe the set of all tuples $(x_1,\ldots, x_m)$ that satisfy a statement of the following form:  there exist
traces $y_{i,j}$ ($(i,j) \in S$) such that
 the following hold:
\begin{enumerate}[(a)]
\item $u_i^{x_i} = \tilde{y}_{i,1} \cdots \tilde{y}_{i,k_i}$ in $\dM(A,I)$ for all $i \in N_{\overline{a}}$
\item $\alp(y_{i,j}) = A_{i,j}$ for all $(i,j) \in S$,
\item $y_{i,j} = y_{k,\ell}^{-1}$ in $\dM(A,I)$ for all $(i,j,k,\ell) \in M$
\item $(x_i)_{i \in N_a} \in S_a$
\end{enumerate}
In the next step, we eliminate the  trace equations $u_i^{x_i} = \tilde{y}_{i,1} \cdots \tilde{y}_{i,k_i}$  
($i \in N_{\overline{a}}$). We apply to each of these trace equations Lemma~\ref{lemma-simplify-factorization-power}
(or Remark~\ref{remark-simplify-factorization-power}). For every $i \in N_{\overline{a}}$ we guess a subset
$K_i \subseteq [1,k_i]$, an integer $0 \le c_i \leq  |\Gamma| \cdot (k_i-1)$ and traces $p_{i,j}, s_{i,j}$ 
with $\alp(p_{i,j}) \subseteq \alp(u_i) \supseteq \alp(s_{i,j})$ and 
$|p_{i,j}|, |s_{i,j}| \leq |\Gamma| \cdot (k_i-1) \cdot |u_i| \leq |\Gamma| \cdot (k_i-1) \cdot  \lambda(\E)$,
and replace $u_i^{x_i} = \tilde{y}_{i,1} \cdots \tilde{y}_{i,k_i}$ by the following statement:
there exist integers $x_{i,j} > 0$ ($j \in K_i$)
 such that
 \begin{itemize}
 \item $x_i = c_i +  \sum_{j \in K_i} x_{i,j}$,
 \item $\tilde{y}_{i,j} = p_{i,j}  u_i^{x_{i,j}} s_{i,j}$ for all $j \in K_i$,
 \item $\tilde{y}_{i,j} = p_{i,j} s_{i,j}$ for all $j \in  [1,k_i] \setminus K_i$.
 \end{itemize}
 At this point we can check whether the alphabetic constraints $\alp(y_{i,j}) = A_{i,j}$ for $(i,j) \in S$
 hold (note that an equation $y_{i,j} = p_{i,j} s_{i,j}$ or $y_{i,j} = p_{i,j} u_i^{x_{i,j}} s_{i,j}$ with $x_{i,j} > 0$
 determines the alphabet of $y_{i,j}$).
 Equations $\tilde{y}_{i,j} = p_{i,j} s_{i,j}$, where $\tilde{y}_{i,j}$ is an explicit trace
 can be checked and possibly lead to a bad guess.
 From an equation $\tilde{y}_{i,j} = p_{i,j}  u_i^{x_{i,j}} s_{i,j}$, where $\tilde{y}_{i,j}$ is an explicit trace,
 we can determine a unique solution for $x_{i,j}>0$ (if it exists) and substitute this value into the equation 
 $x_i = c_i +  \sum_{j \in K_i} x_{i,j}$.  Note that we must 
 have $x_{i,j} \le |\tilde{y}_{i,j}| \leq \lambda(\E)$, since $\tilde{y}_{i,j}$ is an atom or a factor of a trace $v_k^{-1}$.
 Similarly, an equation $y_{i,j} = p_{i,j} s_{i,j}$ with $(i,j) \in S$ allows us to replace  the symbolic factor 
 $y_{i,j}$ by the concrete trace $p_{i,j} s_{i,j}$ and the unique symbolic factor $y_{k,\ell}$ with $(i,j,k,\ell) \in M$
 by the concrete trace $s_{i,j}^{-1} p_{i,j}^{-1}$. If we have an equation 
 $y_{k,\ell} = p_{k,\ell}  s_{k,\ell}$ then we check whether $s_{i,j}^{-1} p_{i,j}^{-1} = p_{k,\ell}  s_{k,\ell}$ holds. Otherwise
 we have an equation 
  $y_{k,\ell} = p_{k,\ell}  u_k^{x_{k,\ell}} s_{k,\ell}$, and we can compute the unique non-zero solution 
  for $x_{k,\ell}$ (if it exists). Note that $x_{k,\ell} \le |s_{i,j}^{-1} p_{i,j}^{-1}| \le 2 |\Gamma| \cdot (k_i-1) \cdot  \lambda(\E) \in \mathcal{O}(|\Gamma|^2 \cdot m^2 \cdot \lambda(\E))$.
We then replace $x_{k,\ell}$ in the equation $x_k = c_k +  \sum_{\ell \in K_k} x_{k,\ell}$
by this unique solution. 

By the above procedure, our statement  (a)--(d) (with existentially quantified traces $y_{i,j}$) is 
transformed nondeterministically into a statement of the following form:
there exist integers $x_{i,j} > 0$ ($i \in N_{\overline{a}}$, $j \in K'_i$) such that 
the following hold:
\begin{enumerate}[(a)]
\item $x_i = c'_i +  \sum_{j \in K'_i} x_{i,j}$ for $i \in N_{\overline{a}}$,
\item $p_{i,j} u_i^{x_{i,j}} s_{i,j} = s_{k,\ell}^{-1} (u_k^{-1})^{x_{k,\ell}} p_{k,\ell}^{-1}$ in $\dM(A,I)$ for all $(i,j,k,\ell) \in M'$,
\item $(x_i)_{i \in N_a} \in S_a$.
\end{enumerate}
Here, $K'_i \subseteq K_i \subseteq [1,k_i]$ is a set of size at most $k_i \leq 28\alpha m^2$, 
$M' \subseteq M$ is a new matching relation (with $(i,j,k,\ell) \in M'$ if and only if $(k,\ell,i,j) \in M'$), and
$c'_i \le |\Gamma| \cdot (k_i-1) + k_i \cdot  \mathcal{O}(|\Gamma|^2 \cdot m^2 \cdot \lambda(\E)) \le 
\mathcal{O}(|\Gamma|^3 \cdot m^4\cdot \lambda(\E))$.
%and the $p_{i,j}, s_{i,j}$ are concrete traces of length
%at most  $|\Gamma| \cdot (\ell_i -1) \cdot |u_i| \leq 2^{2m+1}(|\Gamma| +1)^2 \cdot \lambda$. 

\medskip
\noindent
{\em Step 5: Elimination of two-dimensional knapsack instances.}
The remaining knapsack equations $p_{i,j} u_i^{x_{i,j}} s_{i,j} = s_{k,\ell}^{-1} (u_k^{-1})^{x_{k,\ell}} p_{k,\ell}^{-1}$ in (b)
are two-dimensional and can be eliminated with 
Lemma~\ref{lemma-connected-star}. By this lemma, every trace equation 
$$p_{i,j}  u_i^{x_{i,j}} s_{i,j}  = s_{k,\ell}^{-1} (u^{-1}_k)^{x_{k,\ell}} p_{k,\ell}^{-1}$$
(recall that all $u_i$ are connected, which is assumed in Lemma~\ref{lemma-connected-star})
can be nondeterministically replaced by a semilinear constraint
$$
(x_{i,j}, x_{k,\ell}) \in \{ (a_{i,j,k,\ell} + b_{i,j,k,\ell} \cdot z, a_{k,\ell,i,j} + b_{k,\ell,i,j} \cdot z) \mid z \in \mathbb{N} \}.
$$
For the numbers $a_{i,j,k,\ell}, b_{i,j,k,\ell}, a_{k,\ell,i,j}, b_{k,\ell,i,j}$ we obtain the bound
$$
a_{i,j,k,\ell}, b_{i,j,k,\ell}, a_{k,\ell,i,j}, b_{k,\ell,i,j} \in \mathcal{O}(\mu^8 \cdot \nu^{4 |\Gamma|}),
$$
where, by Lemma~\ref{lemma-prefixes},
\begin{equation} \label{eq-mu}
\mu = \max\{ \rho(p_{i,j}), \rho(p_{k,\ell}), \rho(s_{i,j}), \rho(s_{k,\ell})\} \le  \mathcal{O}(|\Gamma|^{2\alpha} \cdot 28^\alpha \cdot m^{2 \alpha } \cdot \lambda(\E)^\alpha)
\end{equation}
and 
\begin{equation}  \label{eq-nu}
\nu = \max\{ \rho(u_i), \rho(u_k) \} \le \mathcal{O}(\lambda(\E)^\alpha).
\end{equation}
Note that $\rho(t) = \rho(t^{-1})$ for every trace $t$. 
Moreover, note that we have the constraints $x_{i,j}, x_{k,\ell} > 0$.
Hence, if our nondeterministic guess yields $a_{i,j,k,\ell} = 0$ or $a_{k,\ell,i,j} = 0$ then
we make the replacement $a_{i,j,k,\ell} = a_{i,j,k,\ell} + b_{i,j,k,\ell}$ and 
$a_{k,\ell,i,j} = a_{k,\ell,i,j} + b_{k,\ell,i,j}$. If after this replacement we still have
$a_{i,j,k,\ell} = 0$ or $a_{k,\ell,i,j} = 0$ then our guess was bad.

At this point, we have obtained a statement of the following form:
there exist $z_{i,j,k,\ell} \in \N$ (for $(i,j,k,\ell) \in M'$) with $z_{i,j,k,\ell} = z_{k,\ell,i,j}$ and such that
\begin{enumerate}[(a)]
\item $x_i = c'_i+ \sum_{(i,j,k,\ell) \in M'} (a_{i,j,k,\ell} + b_{i,j,k,\ell} \cdot z_{i,j,k,\ell})$ for $i \in N_{\overline{a}}$, and
\item $(x_i)_{i \in N_a} \in S_a$.
\end{enumerate}
Note that the sum in (a)  contains $|K'_i| \leq 28m^2\alpha$ many summands (since for every
$j \in K'_i$ there is a unique pair $(k,\ell)$ with $(i,j,k,\ell) \in M'$).
Hence, (a) can be written as $x_i = c''_i+ \sum_{(i,j,k,\ell) \in M'} b_{i,j,k,\ell} \cdot z_{i,j,k,\ell}$ with
\begin{eqnarray*}
c''_i & = & c'_i + \sum_{(i,j,k,\ell) \in M'} a_{i,j,k,\ell} \\
& \le &  \mathcal{O}( |\Gamma|^3 \cdot m^4\cdot \lambda(\E)) +  \mathcal{O}(\alpha \cdot m^2 \cdot \mu^8 \cdot \nu^{4 |\Gamma|}) \\
& \le & \mathcal{O}(|\Gamma|^{16 \alpha +1} \cdot 28^{8 \alpha} \cdot m^{16\alpha +2} \cdot \lambda(\E)^{8 \alpha + 4 \alpha |\Gamma|}) \\
& \le & \mathcal{O}(|\Gamma|^{16 \alpha +1} \cdot 28^{8 \alpha} \cdot m^{16\alpha +2} \cdot \norm{\E}^{8 \alpha + 4 \alpha |\Gamma|})
\end{eqnarray*}
(since $\lambda(\E) \leq \norm{\E}$).
The bound in the last line is also an upper bound for the numbers $b_{i,j,k,\ell}$. 
Hence, we have obtained a semilinear representation for $\Sol_{\dG}(\E)$ whose magnitude is bounded by
$\max\{ \mathsf{K}_1, \mathsf{K}_2 \}$, where 
$$
\mathsf{K}_1  \le   \mathsf{K}(2 m \norm{\E},m)
$$
(this is our upper bound for the magnitude of the semilinear set $S_a$) and
$$
\mathsf{K}_2  \le \mathcal{O} \big(|\Gamma|^{16 \alpha +1} \cdot 28^{8 \alpha} \cdot m^{16\alpha +2} \cdot \norm{\E}^{8 \alpha + 4 \alpha |\Gamma|}\big).
$$
{\em Step 6: Integration of the preprocessing step.}
Recall that so far we only considered the case where $\E$ is a knapsack expression having the form of the $\E'$ in 
Lemma~\ref{lemma-preproc}.
Let us now consider an arbitrary exponent expression $\E$ of degree $m$.
By Lemma~\ref{lemma-preproc} we have
$\Sol_{\dG}(\E) = (K \cap  \Sol_{\dG}(\E')) \rest_{X_{\E}}$ 
where $K$ is semilinear of magnitude one and 
$\E'$ has degree at most $\alpha \cdot m$ and satisfies 
$\norm{\E'} \leq 3 \norm{\E}$. We can apply the upper bound shown so far to 
$\E'$. Hence, the magnitude of $\Sol_{\dG}(\E')$ is bounded by
$\max\{ \mathsf{K}'_1, \mathsf{K}'_2 \}$, where 
$$
\mathsf{K}'_1  \le  \mathsf{K}(6 \alpha m \norm{\E},\alpha m)
$$
and
$$
\mathsf{K}'_2   \le \mathcal{O}\big(|\Gamma|^{16 \alpha +1} \cdot 28^{8 \alpha} \cdot (\alpha m)^{16\alpha +2} \cdot (3 \norm{\E})^{8 \alpha + 4 \alpha |\Gamma|}\big).
$$
It remains to analyze the influence of intersecting with $K$. For this, we can apply 
Proposition~\ref{intersection-semilinear-sets}, which yields for the magnitude the upper bound 
$\max\{ \mathsf{K}_1, \mathsf{K}_2 \}$, where 
\begin{eqnarray*}
\mathsf{K}_1 &  \le &  \mathcal{O}\big( (\alpha m)^{\alpha m/2+3} \cdot  \mathsf{K}(6 \alpha m  \norm{\E},\alpha m)^{\alpha m+3}\big) 
\end{eqnarray*}
and
\begin{eqnarray*}
\mathsf{K}_2 & \le & \mathcal{O}\big( (\alpha m)^{\alpha m/2+3} \cdot \mathcal{O}\big(|\Gamma|^{32 \alpha +3} \cdot 28^{8 \alpha} \cdot (\alpha m)^{16\alpha +2} \cdot (3 \norm{\E})^{8 \alpha + 4 \alpha |\Gamma|}\big)^{\alpha m+3} \big) \\
& \le & (\alpha m)^{\mathcal{O}(\alpha^2 m)}  \cdot \norm{\E}^{\mathcal{O}(\alpha^2 |\Gamma| m)} .
\end{eqnarray*}
This concludes the proof of the theorem.
\end{proof}

\begin{bem} \label{rem:solution-set-GP}
Assume that $G$ is a fixed graph product (hence, $|\Gamma|$ is a constant).
Consider again the case that $\E$ is a knapsack expression (i.e., $x_i \neq x_j$ for $i \neq j$) 
where every period $u_i$ is either an atom or a well-behaved trace. Let $m = \deg(\E)$.
In the above proof, we show that
the set of solutions $\Sol_{\dG}(\E)$ can be written as a finite union 
$$
\Sol_{\dG}(\E) = \bigcup_{i=1}^p \bigoplus_{j=1}^{q_i} \Sol_{H_{i,j}}(\E_{i,j}) \oplus L_i
$$
such that the following hold for every $1 \leq i \leq p$:
\begin{itemize}
\item every $H_{i,j}$ is one of the groups $G_k$ and $\E_{i,j}$ is a knapsack expression over the group $H_{i,j}$.
The variable sets $X_{\E_{i,j}}$ ($1 \leq j \leq q_i$) form a partition of the set $X_a$ (the variables corresponding to atomic periods).
\item Every $\E_{i,j}$ is a knapsack expression of size at most $2m \norm{\E}$ and degree at most $m$ (see step 3 in the above proof).
\item   The set $L_i$ is semilinear of magnitude 
$\mathcal{O} \big(|\Gamma|^{16 \alpha +1} 28^{8 \alpha}  m^{16\alpha +2} \norm{\E}^{8 \alpha + 4 \alpha |\Gamma|}\big) = 
\mathcal{O} \big(m^{16\alpha +2}  \norm{\E}^{8 \alpha + 4 \alpha |\Gamma|}\big)$
(see step 5 in the above proof).
\end{itemize}
Here, the indices $i \in [1,p]$ correspond to the guessed data in the above prove.
Moreover, given $i \in [1,p]$ (i.e., a specific guess), one can compute the knapsack expressions $\E_{i,j}$ ($1 \leq j \leq q_i$)
and a semilinear representation of $L_i$ in polynomial time. This yields a nondeterministic reduction of the knapsack problem
for the graph product $G$ to the knapsack problems for the groups $G_i$ ($i \in \Gamma$), assuming the input expression $\E$
satisfies the above restriction. Recall that in general, direct products do not preserve decidability of the knapsack problem.
\end{bem}

The reader might wonder, whether we can obtain a bound for the function $\mathsf{K}_{G,\Sigma}$ in terms 
of the function $\mathsf{K}_{G_i,\Sigma_i}$, which is better than the corresponding bound for 
$\mathsf{E}_{G,\Sigma}$ from Theorem~\ref{thm-main-technical}. This is actually not the case (at least with our proof
technique): a power of the form $u^x$ where $u = u_1 u_2 \in \dM(A,I)$ with $u_1 \mathrel{I} u_2$ is equivalent to $u_1^x u_2^x$.
Hence, powers $u^x$ with $u$ a non-connected trace naturally lead to a duplication
of the variable $x$ (and hence to an exponent expression which is no longer a knapsack expression). 
This is the reason why we bounded the (in general faster growing) function $\mathsf{E}_{G,\Sigma}$ in terms
of the functions $\mathsf{K}_{G_i,\Sigma_i}$ in Theorem~\ref{thm-main-technical}.

An application of Theorem~\ref{thm-main-technical} is the following:

\begin{theorem}
Let $G$ be a graph product of hyperbolic groups. Then solvability of exponent equations over $G$ belongs to {\bf NP}.
\end{theorem}

\begin{proof}
For a hyperbolic group $H$ (with an arbitrary generating set $\Sigma'$) it was shown in \cite{LOHREY2019} that the function $\mathsf{K}_{H,\Sigma'}(n) = \mathsf{K}_{H,\Sigma'}(n,n)$ 
is polynomially bounded. Theorem~\ref{thm-main-technical} yields an exponential bound for the function  $\mathsf{E}_{G,\Sigma}(n) = \mathsf{E}_{G,\Sigma}(n,n)$
(note that $|\Gamma|$ and $\alpha$ are constants since we consider a fixed graph product $G$).
A nondeterministic polynomial time Turing machine can therefore guess the binary encodings of numbers $\sigma(x) \leq \mathsf{K}_{G,\Sigma}(\norm{\E})$ for 
each variable $x$ of the input exponent expression $e$.
Checking whether $\sigma$ is a $G$-solution of $e$ is an instance of the compressed word problem for $G$.
By the main result of \cite{HoltLS19} the compressed word problem for a hyperbolic group can be solved in polynomial time and 
by \cite{HauLohHau13} the compressed word problem for a graph product of groups $G_i$ ($i \in \Gamma$) can be solved in polynomial 
time if for every $i \in \Gamma$ the compressed word problem for $G_i$ can be solved in polynomial time. Hence, we can check
in polynomial time if  $\sigma$ is a $G$-solution of $e$.
\end{proof}
Let us now consider the special case where the graph product is a free product of two groups $\dG_1$ and $\dG_2$.

\begin{theorem}\label{thm-cor-technical}
If the groups $\dG_1$ and $\dG_2$ are knapsack-semilinear, then $\dG_1*\dG_2$ is knapsack-semilinear as well.
Let $\mathsf{K}(n,m)$ be the pointwise maximum of the functions
$\mathsf{K}_{\dG_1,\Sigma_1}$ and $\mathsf{K}_{\dG_2,\Sigma_2}$.
 Then for $\dG = \dG_1 * \dG_2$ we have $\mathsf{K}_{\dG,\Sigma}(n,m) \le \max\{\mathsf{K}_1, \mathsf{K}_2\}$ with 
\begin{equation*}
\mathsf{K}_1 = \mathsf{K}(6 m n,m)  \text{ and } 
\mathsf{K}_2 \le \mathcal{O}(m n^4) .
\end{equation*}
\end{theorem}

\begin{proof}
The proof is similar to the one from Theorem~\ref{thm-main-technical}. 
We first consider the case where every period $u_i$ is either an atom
or a well-behaved word (see Remark~\ref{rem-preproc-free-product}).

Let us go throw the six steps from the proof of Theorem~\ref{thm-main-technical}:

\medskip
\noindent
{\em Step 1.} This step is carried out in the same way as in 
the proof of Theorem~\ref{thm-main-technical}. 

\medskip
\noindent
{\em Step 2.} Here we can use Lemma~\ref{lemma-reduction-free-product} instead of Lemma~\ref{lemma-reduction},
which yields the upper bound of $14m$ on the number of factors in our refinement of 
$u_1^{x_1} v_1 u_2^{x_2} v_2 \cdots u_m^{x_m} v_m$ (where powers $u_i^{x_i}$ with $i \in N_1$ have been replaced
by single atoms). The number of atom creations (of any type) is at most $2m-2$.  We do not have to guess the atom
sets $A_{i,j} \subseteq \alp(u_{i,j})$ since there are no swapping steps in 
Lemma~\ref{lemma-reduction-free-product}.

\medskip
\noindent
{\em Step 3.} 
This step is copied from the  proof of Theorem~\ref{thm-main-technical}. 
We obtain for the variables $x_i$ with $i \in N_a$ the semilinear constraint
$(x_i)_{i \in N_a} \in S_a$ where $S_a$ is of magnitude at most
$\mathsf{K}(2m \norm{\E},m)$.

\medskip
\noindent
{\em Step 4.} Also this step is analogous to the proof of Theorem~\ref{thm-main-technical}. Recall that
we have the better bound $14 m$ on the number of factors in our refinement of 
$u_1^{x_1} v_1 u_2^{x_2} v_2 \cdots u_m^{x_m} v_m$.  Eliminating the equations $u_i^{x_i} = \tilde{y}_{i,1} \cdots \tilde{y}_{i,k_i}$  
($i \in N_{\overline{a}}$), which are interpreted in $A^*$, is much easier due to the absence of commutation.
For every $i \in N_{\overline{a}}$ we obtain a disjunction of statements of the following form:
there exist integers $x_{i,j} \geq 0$ ($1 \leq j \leq k_i$) such that
 \begin{itemize}
 \item $x_i = c_i +  \sum_{j=1}^{k_i} x_{i,j}$,
 \item $\tilde{y}_{i,j} = p_{i,j}  u_i^{x_{i,j}} s_{i,j}$ for all $1 \le j \le k_i$.
 \end{itemize}
 Here, every $p_{i,j}$ is a suffix of $u_i$, every $s_{i,j}$ is a prefix of $u_i$
 and $c_i \leq k_i \leq 14 m$. Basically, $c_i$ is the number of factors $u_i$ that
 are split non-trivially in the factorization $u_i^{x_i} = \tilde{y}_{i,1} \cdots \tilde{y}_{i,k_i}$.
We can then carry out the same simplifications that  we did in 
the proof of Theorem~\ref{thm-main-technical}. If $\tilde{y}_{i,j}$ is an explicit word $v_{k,\ell}^{-1}$ then
we determine the unique solution $x_{i,j}$ (if it exists) of $\tilde{y}_{i,j} = p_{i,j}  u_i^{x_{i,j}} s_{i,j}$ and
replace $x_{i,j}$ by that number, which is at most $\lambda(\E)$.
We arrive at a statement of the following form:
there exist integers $x_{i,j} \geq 0$ ($i \in N_{\overline{a}}$, $j \in K_i$) such that 
the following hold:
\begin{enumerate}[(a)]
\item $x_i = c'_i +  \sum_{j \in K_i} x_{i,j}$ for $i \in N_{\overline{a}}$,
\item $p_{i,j} u_i^{x_{i,j}} s_{i,j} = s_{k,\ell}^{-1} (u_k^{-1})^{x_{k,\ell}} p_{k,\ell}^{-1}$ in $A^*$ for all $(i,j,k,\ell) \in M$,
\item $(x_i)_{i \in N} \in S_a$.
\end{enumerate}
Here, $K_i \subseteq [1,k_i]$ is a set of size at most $k_i \leq 14m$, 
$M$ is a matching relation (with $(i,j,k,\ell) \in M$ if and only if $(k,\ell,i,j) \in M$),
and $c'_i \le 14m + k_i \cdot \lambda(\E) \leq \mathcal{O}(m \cdot \lambda(\E))$.
The words $p_{i,j}$ and $s_{i,j}$ have length at most $\lambda(\E)$.

\medskip
\noindent
{\em Step 5.} The remaining two-dimensional knapsack equations from point (b)
are eliminated with 
Remark~\ref{rem-2dim-words}. Every equation 
$$p_{i,j}  u_i^{x_{i,j}} s_{i,j}  = s_{k,\ell}^{-1} (u^{-1}_k)^{x_{k,\ell}} p_{k,\ell}^{-1}$$
can be nondeterministically replaced by a semilinear constraint
$$
(x_{i,j}, x_{k,\ell}) \in \{ (a_{i,j,k,\ell} + b_{i,j,k,\ell} \cdot z, a_{k,\ell,i,j} + b_{k,\ell,i,j} \cdot z) \mid z \in \mathbb{N} \}.
$$
where the numbers $a_{i,j,k,\ell}, b_{i,j,k,\ell}, a_{k,\ell,i,j}, b_{k,\ell,i,j}$ are bounded by 
$\mathcal{O}(\lambda(\E)^4)$.

At this point, we have obtained a statement of the following form:
there exist $z_{i,j,k,\ell} \in \N$ (for $(i,j,k,\ell) \in M$) with $z_{i,j,k,\ell} = z_{k,\ell,i,j}$ and such that
\begin{enumerate}[(a)]
\item $x_i = c'_i+ \sum_{(i,j,k,\ell) \in M} (a_{i,j,k,\ell} + b_{i,j,k,\ell} \cdot z_{i,j,k,\ell})$ for $i \in N_{\overline{a}}$, and
\item $(x_i)_{i \in N_a} \in S_a$.
\end{enumerate}
The sum in (a)  contains $|K_i| \leq 14m$ many summands.
Hence, (a) can be written as $x_i = c''_i+ \sum_{(i,j,k,\ell) \in M'} b_{i,j,k,\ell} \cdot z_{i,j,k,\ell}$ with
\begin{eqnarray*}
c''_i & = & c'_i + \sum_{(i,j,k,\ell) \in M} a_{i,j,k,\ell} \\
& \le &  \mathcal{O}(m \cdot \lambda(\E)) + 14 m \cdot  \mathcal{O}(\lambda(\E)^4) \\
&  = & \mathcal{O}(m \cdot \lambda(\E)^4) .
\end{eqnarray*}
We therefore obtained a semilinear representation for $\Sol_{\dG}(\E)$ whose magnitude is bounded by
$\max\{ \mathsf{K}_1, \mathsf{K}_2 \}$, where 
\begin{equation*}
\mathsf{K}_1 = \mathsf{K}(2m \norm{\E},m) \text{ and } 
\mathsf{K}_2 \le \mathcal{O}(m \norm{\E}^4) .
\end{equation*}
\noindent
{\em Step 6.} For the preprocessing we apply Remark~\ref{rem-preproc-free-product}. Hence, we 
just have to replace $\norm{\E}$ by $3 \norm{\E}$ in the above bounds, which yields
the statement of the theorem.
\end{proof}

\begin{bem}
By Theorem~\ref{thm-cor-technical}, $\mathsf{K}_{\dG,\Sigma}$ is polynomially bounded
if $\mathsf{K}_{\dG_1,\Sigma}$ and $\mathsf{K}_{\dG_2,\Sigma}$ are polynomially bounded.
This was also shown in \cite{LohreyZ18}.
\end{bem}

\begin{bem}
Analogously to Remark~\ref{rem:solution-set-GP}, the above proof 
shows that the set of solutions $\Sol_{\dG}(\E)$ for $G = G_1*G_2$ can be written as a finite union 
$$
\Sol_{\dG}(\E) = \bigcup_{i=1}^p \bigoplus_{j=1}^{q_i} \Sol_{H_{i,j}}(\E_{i,j}) \oplus L_i
$$
such that the following hold for every $1 \leq i \leq p$:
\begin{itemize}
\item every $H_{i,j}$ is either $G_1$ or $G_2$ and $\E_{i,j}$ is a knapsack expression over the group $H_{i,j}$.
The variable sets $X_{\E_{i,j}}$ ($1 \leq j \leq q_i$) form a partition of the set $X_a$ (the variables corresponding to atomic periods).
\item Every $\E_{i,j}$ is a knapsack expression of size at most $6 m \norm{\E}$ and degree at most $m$.
\item   The set $L_i$ is semilinear of magnitude $\mathcal{O}(m n^4)$.
\end{itemize}
Here, the indices $i \in [1,p]$ correspond to the guessed data in the above prove.
Moreover, given $i \in [1,p]$, one can compute the knapsack expressions $\E_{i,j}$ ($1 \leq j \leq q_i$)
and a semilinear representation of $L_i$ in polynomial time.
\end{bem}
The above remark immediately yields the following complexity transfer result. A language $A$ is 
nondeterministically polynomial time reducible to a language $B$ if there exists a nondeterministic polynomial time 
Turing-machine $M$ that outputs on each computation path after termination a word over the alphabet of the language $B$ and such that
$x \in A$ if and only if on input $x$, the machine $M$ has at least one computation path on which it outputs a word from $B$.

\begin{theorem} \label{thm:free-product-NP-reduction}
The knapsack problem for $G_1*G_2$ is nondeterministically polynomial time reducible to the knapsack problems for $G_1$ and $G_2$.
\end{theorem}
The consequence of Theorem~\ref{thm:free-product-NP-reduction} that solvability
of the knapsack problem in $\mathsf{NP}$ is passed on from $G_i$ for $i=1,2$ to
the free product $G_1*G_2$ was shown using different methods in the extended
abstract~\cite{LohreyZetzsche2016a} (see also the comment after
Theorem~\ref{thm:amalgamated-product-NP-reduction}).

\section{Part 2: Knapsack in HNN-extensions and amalgamated products} \label{sec-HNN+amalgamated}

In this section we deal with two constructions that are of
fundamental importance in combinatorial group theory~\cite{LySch77}, namely
HNN-extensions and amalgamated products. In their general form, HNN-extensions
have been used to construct groups with an undecidable word problem, which
means they may destroy desirable algorithmic properties. We consider the
special case of finite associated (resp. identified) subgroups, for which these
constructions already play a prominent role, for example, in Stallings'
decomposition of groups with infinitely many ends~\cite{Stal71} or the
construction of virtually free groups~\cite{DiDu89}.  Moreover, these
constructions are known to preserve a wide range of important structural and
algorithmic properties~\cite{AllGre73,Bez98,HauLo11,KaSiSt06,KaWeMy05,KaSo70,KaSo71,LohSen06icalp,LohSen08,MeRa04}.

\subsection{HNN-extensions preserve knapsack semilinearity}

Suppose $G=\langle \Sigma\mid R\rangle$ is a finitely generated group 
with the  finite symmetric generating set $\Sigma = \Omega \cup \Omega^{-1}$ and the set of relators  $R \subseteq \Sigma^*$.
Fix two isomorphic subgroups $A$ and $B$ of $G$ together with an isomorphism $\varphi\colon A\to
B$. Let $t \notin \Sigma$ be a fresh generator. Then the corresponding \emph{HNN-extension} is the group 
$$H=\langle \Omega\cup \{t\} \mid R\cup \{t^{-1}a^{-1}t\varphi(a) \mid a\in A\}\rangle$$ (formally, we identify here every element $c \in A \cup B$
with a word over $\Sigma$ that evaluates to $c$). This group is usually denoted by 
\begin{equation} \label{eq-HNN}
H=\langle G, t \mid t^{-1}at=\varphi(a)~(a\in A)\rangle .
\end{equation}
Intuitively, $H$ is obtained from $G$ by adding a new element $t$ such that
conjugating elements of $A$ with $t$ applies the isomorphism $\varphi$.  Here,
$t$ is called the \emph{stable letter} and the groups $A$ and $B$ are the
\emph{associated subgroups}.   A basic fact about HNN-extensions is that
the group $G$ embeds naturally into $H$~\cite{HiNeNe49}.

Here, we only consider the case that $A$ and $B$ are finite groups, so that we
may assume that $(A \cup B) \setminus \{1\}$ is contained in the finite generating set $\Sigma$.
To exploit the symmetry of the
situation, we use the notation $A(+1)=A$ and $A(-1)=B$. Then, we have
$\varphi^{\alpha}\colon A(\alpha)\to A(-\alpha)$ for $\alpha\in\{+1,-1\}$.
We will make use of the (possibly infinite) alphabet $\Gamma=G\backslash \{1\}$.
By $h\colon (\Gamma \cup\{t,t^{-1}\})^*\to H$, we denote the canonical
morphism that maps each word to the element of $H$ it represents.

A word $u\in (\Gamma \cup \{t,t^{-1}\})^*$ is called \emph{Britton-reduced} if
it does not contain a factor of the form $cd$ with $c,d\in \Gamma$
or a factor $t^{-\alpha} a t^\alpha$ with $\alpha\in\{-1,1\}$ and
$a\in A(\alpha)$. A factor of the form $t^{-\alpha} a t^\alpha$ with $\alpha\in\{-1,1\}$ and
$a\in A(\alpha)$ is also called a {\em pin}. 
Note that the equation
$t^{-\alpha} a t^\alpha=\varphi^\alpha(a)$ allows us to replace a pin
$t^{-\alpha}at^{\alpha}$ by $\varphi^{\alpha}(a)\in A(-\alpha)$.  Since this
decreases the number of $t$'s in the word, we can reduce every word to
an equivalent Britton-reduced word. We denote the set of all Britton-reduced words in the HNN-extension \eqref{eq-HNN} by $\BR(H)$.

In this section, let $\gamma$ be the cardinality of $A$.
A word $w\in \BR(H) \setminus \Gamma$ is called {\em well-behaved}, if $w^m$ is Britton-reduced for every $m \geq 0$.
Note that $w$ is well-behaved if and only if $w$ and $w^2$ are Britton-reduced.
Elements of $\Gamma$ are also called {\em atomic}.

The \emph{length} of a word $w \in (\Gamma \cup \{t,t^{-1}\})^*$ is defined as usual and denoted by $|w|$. 
For a word $w = a_1 a_2 \cdots a_k$ with $a_i \in \Gamma \cup \{t,t^{-1}\}$ we define the 
the \emph{representation length} of $w$ as $\norm{w} = \sum_{i=1}^k n_i$, where $n_i = 1$ if $a_i \in \{t,t^{-1}\}$ 
and $n_i$ is the geodesic length of $a_i$ in the group $G$ if $a_i \in \Gamma$.
Note that $\norm{a} = 1$ for $a \in (A \cup B) \setminus \{1\}$.

The following lemma provides a necessary and sufficient condition for equality of Britton-reduced words in an HNN-extension (cf. Lemma~2.2 of \cite{HauLo11}):

\begin{lemma}\label{boat}
Let $u=g_0 t^{\delta_1}g_1 \cdots t^{\delta_k}g_k$ and
$v=h_0 t^{\varepsilon_1}h_1 \cdots t^{\varepsilon_\ell}h_\ell$ be Britton-reduced words with $g_0,\dots , g_k, h_0, \dots , h_\ell \in G$ and $\delta_1,\dots \delta_k,\varepsilon_1, \dots , \varepsilon_\ell\in \{ 1, -1 \}$. Then $u=v$ in the HNN-extension $H$ of $G$ if and only if the following hold:
\begin{itemize}
\item $k=\ell$ and $\delta_i=\varepsilon_i$ for $1\leq i \leq k$
\item there exist $c_1, \dots, c_{2m} \in A\cup B$ such that:
\begin{itemize}
\item $g_ic_{2i+1}=c_{2i}h_i$ in $G$ for $0\leq i \leq k$ (here we set $c_0=c_{2k+1}=1$)
\item $c_{2i-1}\in A(\delta_i)$ and $c_{2i}=\varphi^{\delta_i}(c_{2i-1})\in A(-\delta_i)$ for $1\leq i \leq k$.
\end{itemize}
\end{itemize}
\end{lemma}

The second condition of the lemma can be visualized by the diagram from Figure~\ref{fig-van-kampen} (also called a van Kampen diagram, see \cite{LySch77} for more details),
where $k=\ell=4$. Light-shaded (resp. dark-shaded) faces represent relations in $G$
(resp. relations of the form
$ct^\delta=t^\delta \varphi^\delta (c)$ with $c \in A(\delta)$).
The elements $c_1, \dots , c_{2k}$ in such a diagram are also called \emph{connecting elements}.

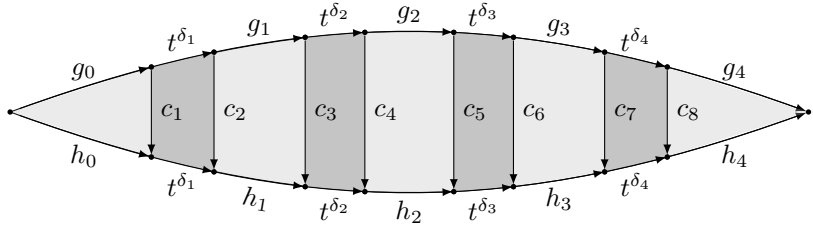
\begin{figure}[t]
\begin{tikzpicture}[x=0.75cm,y=0.75cm]
 \node (A) at (0,0) {};
 \node (R) at (14,0) {};
 
 \tikzset{>=latex}
\draw[->,shorten >= -3pt, shorten <= -3pt, nodes={circle,fill=black,inner sep=0pt,minimum size=2pt}] (A) to[bend right=20]
  node[pos=0.16] (C) {}
  node[pos=0.24] (E) {}
  node[pos=0.36] (G) {}
  node[pos=0.44] (I) {}
  node[pos=0.56] (K) {}
  node[pos=0.64] (M) {}
  node[pos=0.76] (O) {}
  node[pos=0.84] (Q) {}
  (R);
  
\draw[->,shorten >= -3pt, shorten <= -3pt, nodes={circle,fill=black,inner sep=0pt,minimum size=2pt}] (A) to[bend left=20]
  node[pos=0.16] (B) {}
  node[pos=0.24] (D) {}
  node[pos=0.36] (F) {}
  node[pos=0.44] (H) {}
  node[pos=0.56] (J) {}
  node[pos=0.64] (L) {}
  node[pos=0.76] (N) {}
  node[pos=0.84] (P) {}
  (R);
  
 \path[fill=gray!50,opacity=.3] (A.center) to (B.center) to (C.center) to (A.center);
 \path[fill=gray!50,opacity=.9] (B.center) to (C.center) to (E.center) to (D.center) to (B.center);
 \path[fill=gray!50,opacity=.3] (E.center) to (D.center) to (F.center) to (G.center) to (E.center);
 \path[fill=gray!50,opacity=.9] (F.center) to (G.center) to (I.center) to (H.center) to (F.center);
 \path[fill=gray!50,opacity=.3] (I.center) to (H.center) to (J.center) to (K.center) to (I.center);
 \path[fill=gray!50,opacity=.9] (J.center) to (K.center) to (M.center) to (L.center) to (J.center);
 \path[fill=gray!50,opacity=.3] (M.center) to (L.center) to (N.center) to (O.center) to (M.center);
 \path[fill=gray!50,opacity=.9] (N.center) to (O.center) to (Q.center) to (P.center) to (N.center);
 \path[fill=gray!50,opacity=.3] (Q.center) to (P.center) to (R.center) to (Q.center);
	 
\begin{scope}[decoration={
    markings,
    mark=at position 0.171 with {\arrow{>}}}
    ] 
    \path[postaction={decorate}] (A) to[bend left=20] (R);
    \path[postaction={decorate}] (A) to[bend right=20] (R);
\end{scope}

\begin{scope}[decoration={
    markings,
    mark=at position 0.253 with {\arrow{>}}}
    ] 
    \path[postaction={decorate}] (A) to[bend left=20] (R);
    \path[postaction={decorate}] (A) to[bend right=20] (R);
\end{scope}

\begin{scope}[decoration={
    markings,
    mark=at position 0.368 with {\arrow{>}}}
    ] 
    \path[postaction={decorate}] (A) to[bend left=20] (R);
    \path[postaction={decorate}] (A) to[bend right=20] (R);
\end{scope}

\begin{scope}[decoration={
    markings,
    mark=at position 0.443 with {\arrow{>}}}
    ] 
    \path[postaction={decorate}] (A) to[bend left=20] (R);
    \path[postaction={decorate}] (A) to[bend right=20] (R);
\end{scope}

\begin{scope}[decoration={
    markings,
    mark=at position 0.555 with {\arrow{>}}}
    ] 
    \path[postaction={decorate}] (A) to[bend left=20] (R);
    \path[postaction={decorate}] (A) to[bend right=20] (R);
\end{scope}

\begin{scope}[decoration={
    markings,
    mark=at position 0.63 with {\arrow{>}}}
    ] 
    \path[postaction={decorate}] (A) to[bend left=20] (R);
    \path[postaction={decorate}] (A) to[bend right=20] (R);
\end{scope}

\begin{scope}[decoration={
    markings,
    mark=at position 0.746 with {\arrow{>}}}
    ] 
    \path[postaction={decorate}] (A) to[bend left=20] (R);
    \path[postaction={decorate}] (A) to[bend right=20] (R);
\end{scope}

\begin{scope}[decoration={
    markings,
    mark=at position 0.826 with {\arrow{>}}}
    ] 
    \draw[postaction={decorate}] (A) to[bend left=20] (R);
    \draw[postaction={decorate}] (A) to[bend right=20] (R);
\end{scope}

 \path[line width=0pt] (A) -- (B) node[midway,above] {$g_0$};
 \path[line width=0pt] (A) -- (C) node[midway,below] {$h_0$};
 
 \path[line width=0pt] (B) -- (D) node[midway,above] {$t^{\delta_1}$};
 \path[line width=0pt] (C) -- (E) node[midway,below] {$t^{\delta_1}$};
 \draw[->, shorten >= -0.5pt] (B) -- (C) node[midway,right] {$c_1$};
 
 \path[line width=0pt] (D) -- (F) node[midway,above] {$g_1$};
 \path[line width=0pt] (E) -- (G) node[midway,sloped,below] {$h_1$};
 \draw[->, shorten >= -0.5pt] (D) -- (E) node[midway,right] {$c_2$};
 
 \path[line width=0pt] (F) -- (H) node[midway,above] {$t^{\delta_2}$};
 \path[line width=0pt] (G) -- (I) node[midway,below] {$t^{\delta_2}$};
 \draw[->, shorten >= -0.5pt] (F) -- (G) node[midway,right] {$c_3$};
 
 \path[line width=0pt] (H) -- (J) node[midway,above] {$g_2$};
 \path[line width=0pt] (I) -- (K) node[midway,below] {$h_2$};
 \draw[->, shorten >= -0.5pt] (H) -- (I) node[midway,right] {$c_4$};
 
 \path[line width=0pt] (J) -- (L) node[midway,above] {$t^{\delta_3}$};
 \path[line width=0pt] (K) -- (M) node[midway,below] {$t^{\delta_3}$};
 \draw[->, shorten >= -0.5pt] (J) -- (K) node[midway,right] {$c_5$};
 
 \path[line width=0pt] (L) -- (N) node[midway,above] {$g_3$};
 \path[line width=0pt] (M) -- (O) node[midway,below] {$h_3$};
 \draw[->, shorten >= -0.5pt] (L) -- (M) node[midway,right] {$c_6$};
 
 \path[line width=0pt] (N) -- (P) node[midway,above] {$t^{\delta_4}$};
 \path[line width=0pt] (O) -- (Q) node[midway,below] {$t^{\delta_4}$};
 \draw[->, shorten >= -0.5pt] (N) -- (O) node[midway,right] {$c_7$};
 
 \path[line width=0pt] (P) -- (R) node[midway,above] {$g_4$};
 \path[line width=0pt] (Q) -- (R) node[midway,below] {$h_4$};
 \draw[->, shorten >= -0.5pt] (P) -- (Q) node[midway,right] {$c_8$};
 
 \fill (A) circle (1pt);
 \fill (B) circle (1pt);
 \fill (C) circle (1pt);
 \fill (D) circle (1pt);
 \fill (E) circle (1pt);
 \fill (F) circle (1pt);
 \fill (G) circle (1pt);
 \fill (H) circle (1pt);
 \fill (I) circle (1pt);
 \fill (J) circle (1pt);
 \fill (K) circle (1pt);
 \fill (L) circle (1pt);
 \fill (M) circle (1pt);
 \fill (N) circle (1pt);
 \fill (O) circle (1pt);
 \fill (P) circle (1pt);
 \fill (Q) circle (1pt);
 \fill (R) circle (1pt);
\end{tikzpicture}
\caption{\label{fig-van-kampen}A van Kampen diagram witnessing the equality of two Britton-reduced words in the HNN-extension $H$.}
\end{figure}

For our purposes, we will need the following lemma (cf. Lemma 2.3 of \cite{HauLo11}), which allows us to transform an arbitrary string over the generating set of an HNN-extension into a reduced one:

\begin{lemma}\label{red uv}
Assume that $u=g_0 t^{\delta_1}g_1 \cdots t^{\delta_k}g_k$ and
$v=h_0 t^{\varepsilon_1}h_1 \cdots t^{\varepsilon_\ell} h_\ell$ are Britton-reduced words ($g_i, h_j \in G$). Let $m(u,v)$ be the largest number $m\geq 0$ such that
\begin{enumerate}[(a)]
\item $A(\delta_{k-m+1})=A(-\varepsilon_m)$ (we set $A(\delta_{k+1})=A(-\varepsilon_0)=1$) and
\item there is  $c \in A(-\varepsilon_m)$ such that
$$t^{\delta_{k-m+1}} g_{k-m+1} \cdots t^{\delta_k} g_k h_0 t^{\varepsilon_1} \cdots h_{m-1} t^{\varepsilon_m} =_H c$$ 
(for $m=0$ this condition is satisfied with $c=1$).
\end{enumerate}
Moreover, let $c(u,v) \in A(-\varepsilon_m)$ be the element $c$ in (b) (for $m=m(u,v)$). Then
$$
g_0 t^{\delta_1}g_1 \cdots t^{\delta_{k-m(u,v)}} \gamma(u,v) t^{\varepsilon_{m(u,v)+1}}h_{m(u,v)+1} \cdots t^{\varepsilon_\ell}h_\ell
$$
is a Britton-reduced word equal to $uv$ in $H$, where $\gamma(u,v) \in  G$ such that
$\gamma(u,v) =_G g_{k-m(u,v)} c(u,v) h_{m(u,v)}$.
\end{lemma}

The above lemma is visualized in Figure~\ref{fig-peak-elimination}.

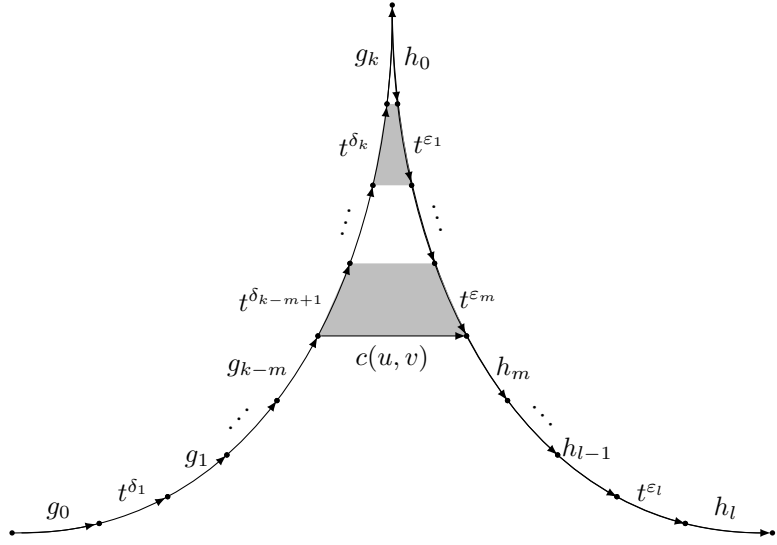
\begin{figure}[t]
\begin{center}
\begin{tikzpicture}[x=0.5cm,y=0.5cm]
 \node (A) at (-10,0) {};
 \node (J) at (0,14) {};
 \node (S) at (10,0) {};
 
  \tikzset{>=latex}
\path[->,shorten >= -3pt, shorten <= -3pt, nodes={circle,fill=black,inner sep=0pt,minimum size=2pt}] (A) to[out=0,in=270]
  node[pos=0.11] (B) {}
  node[pos=0.22] (C) {}
  node[pos=0.33] (D) {}
  node[pos=0.44] (E) {}
  node[pos=0.55] (F) {}
  node[pos=0.66] (G) {}
  node[pos=0.77] (H) {}
  node[pos=0.88] (I) {}
  (J);
  
\draw[->,shorten >= -3pt, shorten <= -3pt, nodes={circle,fill=black,inner sep=0pt,minimum size=2pt}] (J) to[out=270,in=180]
  node[pos=0.12] (K) {}
  node[pos=0.23] (L) {}
  node[pos=0.34] (M) {}
  node[pos=0.45] (N) {}
  node[pos=0.56] (O) {}
  node[pos=0.67] (P) {}
  node[pos=0.78] (Q) {}
  node[pos=0.89] (R) {}
  (S);
  
 \path[fill=gray!50,opacity=1.0]
 (F.center) to (G.center) to (M.center) to (N.center)
 (H.center) to (I.center) to (K.center) to (L.center);
 
\draw[->,shorten <= -3pt] (A) to[bend right=6] (B);

\begin{scope}[decoration={
    markings,
    mark=at position 0.211 with {\arrow{>}}}
    ] 
    \path[postaction={decorate},shorten >= -3pt, shorten <= -3pt] (A) to[out=0,in=270] (J);
\end{scope}

\begin{scope}[decoration={
    markings,
    mark=at position 0.315 with {\arrow{>}}}
    ] 
    \path[postaction={decorate},shorten >= -3pt, shorten <= -3pt] (A) to[out=0,in=270] (J);
\end{scope}

\begin{scope}[decoration={
    markings,
    mark=at position 0.42 with {\arrow{>}}}
    ] 
    \path[postaction={decorate},shorten >= -3pt, shorten <= -3pt] (A) to[out=0,in=270] (J);
\end{scope}

\begin{scope}[decoration={
    markings,
    mark=at position 0.527 with {\arrow{>}}}
    ] 
    \path[postaction={decorate},shorten >= -3pt, shorten <= -3pt] (A) to[out=0,in=270] (J);
\end{scope}

\begin{scope}[decoration={
    markings,
    mark=at position 0.64 with {\arrow{>}}}
    ] 
    \path[postaction={decorate},shorten >= -3pt, shorten <= -3pt] (A) to[out=0,in=270] (J);
\end{scope}

\begin{scope}[decoration={
    markings,
    mark=at position 0.755 with {\arrow{>}}}
    ] 
    \path[postaction={decorate},shorten >= -3pt, shorten <= -3pt] (A) to[out=0,in=270] (J);
\end{scope}

\begin{scope}[decoration={
    markings,
    mark=at position 0.87 with {\arrow{>}}}
    ] 
    \draw[postaction={decorate}] (A) to[out=0,in=270] (J);
\end{scope}

\draw[->,shorten >= -3pt, shorten <= -3pt] (I) to[bend right=3.5] (J);

\begin{scope}[decoration={
    markings,
    mark=at position 0.125 with {\arrow{>}}}
    ] 
    \path[postaction={decorate},shorten >= -3pt, shorten <= -3pt] (J) to[out=270,in=180] (S);
\end{scope}

\begin{scope}[decoration={
    markings,
    mark=at position 0.24 with {\arrow{>}}}
    ] 
    \path[postaction={decorate},shorten >= -3pt, shorten <= -3pt] (J) to[out=270,in=180] (S);
\end{scope}

\begin{scope}[decoration={
    markings,
    mark=at position 0.358 with {\arrow{>}}}
    ] 
    \path[postaction={decorate},shorten >= -3pt, shorten <= -3pt] (J) to[out=270,in=180] (S);
\end{scope}

\begin{scope}[decoration={
    markings,
    mark=at position 0.469 with {\arrow{>}}}
    ] 
    \path[postaction={decorate},shorten >= -3pt, shorten <= -3pt] (J) to[out=270,in=180] (S);
\end{scope}

\begin{scope}[decoration={
    markings,
    mark=at position 0.579 with {\arrow{>}}}
    ] 
    \path[postaction={decorate},shorten >= -3pt, shorten <= -3pt] (J) to[out=270,in=180] (S);
\end{scope}

\begin{scope}[decoration={
    markings,
    mark=at position 0.684 with {\arrow{>}}}
    ] 
    \path[postaction={decorate},shorten >= -3pt, shorten <= -3pt] (J) to[out=270,in=180] (S);
\end{scope}

\begin{scope}[decoration={
    markings,
    mark=at position 0.787 with {\arrow{>}}}
    ] 
    \path[postaction={decorate},shorten >= -3pt, shorten <= -3pt] (J) to[out=270,in=180] (S);
\end{scope}

\begin{scope}[decoration={
    markings,
    mark=at position 0.89 with {\arrow{>}}}
    ] 
    \draw[postaction={decorate}] (J) to[out=270,in=180] (S);
\end{scope}

\tikzset{>=latex}
 \path[line width=0pt] (A) -- (B) node[midway,above] {$g_0$};
 \path[line width=0pt] (B) -- (C) node[midway,above] {$t^{\delta_1}$};
 \path[line width=0pt] (C) -- (D) node[midway,above] {$g_1$};
 \path[line width=0pt] (D) -- (E) node[midway,sloped,above] {$\cdots$};
 \path[line width=0pt] (E) -- (F) node[midway,left] {$g_{k-m}$};
 \path[line width=0pt] (F) -- (G) node[midway,left] {$t^{\delta_{k-m+1}}$};
 \path[line width=0pt] (G) -- (H) node[midway,sloped,above] {$\cdots$};
 \path[line width=0pt] (H) -- (I) node[midway,left] {$t^{\delta_k}$};
 \path[line width=0pt] (I) -- (J) node[midway,left] {$g_k$};
 
 \path[line width=0pt] (J) -- (K) node[midway,right] {$h_0$};
 \path[line width=0pt] (K) -- (L) node[midway,right] {$t^{\varepsilon_1}$};
 \path[line width=0pt] (L) -- (M) node[midway,sloped,above] {$\cdots$};
 \path[line width=0pt] (M) -- (N) node[midway,right] {$t^{\varepsilon_m}$};
 \path[line width=0pt] (N) -- (O) node[midway,right] {$h_m$};
 \path[line width=0pt] (O) -- (P) node[midway,sloped,above] {$\cdots$};
 \path[line width=0pt] (P) -- (Q) node[midway,above=2pt] {$h_{\ell-1}$};
 \path[line width=0pt] (Q) -- (R) node[midway,above] {$t^{\varepsilon_\ell}$};
 \path[line width=0pt] (R) -- (S) node[midway,above] {$h_\ell$};
 
 \draw[->, shorten >= -1pt] (F) -- (N) node[midway,below] {$c(u,v)$};
 
 \fill (A) circle (1pt);
 \fill (B) circle (1pt);
 \fill (C) circle (1pt);
 \fill (D) circle (1pt);
 \fill (E) circle (1pt);
 \fill (F) circle (1pt);
 \fill (G) circle (1pt);
 \fill (H) circle (1pt);
 \fill (I) circle (1pt);
 
 \fill (J) circle (1pt);
 
 \fill (K) circle (1pt);
 \fill (L) circle (1pt);
 \fill (M) circle (1pt);
 \fill (N) circle (1pt);
 \fill (O) circle (1pt);
 \fill (P) circle (1pt);
 \fill (Q) circle (1pt);
 \fill (R) circle (1pt);
 \fill (S) circle (1pt);
\end{tikzpicture}
\end{center}
\caption{\label{fig-peak-elimination} The situation from Lemma~\ref{red uv}.}
\end{figure}

\begin{lemma}\label{Britton-reduced-powers}
From a given word $u\in \BR(H)$ we can compute words $s,p,v \in \BR(H)$ such that $u^m =_H s v^m p$ for every $m\geq 0$ 
and either $v \in G$ or $v$ is well-behaved and starts with $t^{\pm 1}$.
Moreover, $\norm{s}+\norm{p}+\norm{v} \leq 3\norm{u}$.
\end{lemma}

\begin{proof}
Let $u\in \BR(H)$. Assume that $u$ is not atomic; otherwise we are done. 
Let us now consider the word $u^2$. If $u^2$ is not Britton-reduced, we can do the following: using Britton-reduction 
we compute a factorization $u=xyz$ such that $zx =_H g \in G$ and hence $u^2 =_H xygyz$, where
moreover $x$ and $y$ are chosen such that the sum $|x|+|z|$ is maximal.
Note that either $y=1$ or $y$ begins and ends with $t$ or $t^{-1}$ (otherwise we could make $x$ or $z$ longer).
Moreover, $\norm{g} \leq \norm{u}$.
We obtain the equality $u^m = (xyz)^m =_H x (yg)^m g^{-1} z$ for every $m \geq 0$.
We set  $s = x$, $v = y g$, and $p = g^{-1} z$. If $y=1$, we have $v = g \in G$. Now assume that
$y$ begins and ends with $t$ or $t^{-1}$. Since $y$ as a factor of $u$ must be Britton-reduced,
also $v = yg$ is Britton-reduced. Moreover, $v^2 = y g y g$ must be Britton-reduced, otherwise
we could extend the length of $x$ and $z$ in the factorization $u=xyz$. 
Finally, we have $\norm{s}+\norm{p}+\norm{v} = \norm{x}+\norm{y}+\norm{z} + 2\norm{g} 
\leq 3 \norm{u}$.
\end{proof}

%%%
\iffalse
\begin{proof}
Let $u\in \BR(H)$. Assume that $u$ is not atomic; otherwise we are done. 
Let us now consider the word $u^2$. If $u^2$ is not Britton-reduced, we can do the following: With Lemma~\ref{red uv} it is easy to compute a factorization $u=xyz$, such that $zx =_H c\in A \cup B$ and hence $u^2 =_H xycyz$, where
moreover $x$ and $y$ are chosen such that their lengths $|x|$ and $|z|$ are maximal.
%$xycyz$ becomes Britton-reduced after multiplying successive symbols from $\Gamma$.
We obtain the equality $u^m = (xyz)^m =_H x (yc)^m c^{-1} z$ for every $m \geq 0$.
If $y \in G$ then we have $yc \in G$ and we can set $s = x$, $v = yc$, and $p = c^{-1} z$. Otherwise, assume that $y$ contains an occurrence of $t$ or $t^{-1}$
and let us write $y = g y' g'$ where $y'$ starts and ends with $t$ or $t^{-1}$ and $g, g' \in G$.  
We have $$u^m =_H x (yc)^m c^{-1} z = x (g y' g' c)^m c^{-1} z =_H xg (y' g'cg)^m (cg)^{-1} z.$$
We now set $s = xg$, $v = y' (g'cg)$, and $p = (cg)^{-1} z$. 
It is easy to observe that $\norm{s}$, $\norm{v}$, and $\norm{p}$ are bounded by $\norm{u}$.
Finally, observe that $v$ and $v^2$ are Britton-reduced (which means that $v$ is well-behaved):
Since $y'$ as a factor $u$ must be Britton-reduced, also $v$ is Britton-reduced. Moreover, 
we have 
$$v^2 = y' (g'cg) y' (g'cg).$$
Hence, if $v^2$ would be not Britton-reduced then $y' (g'cg) y'$ would be not Britton-reduced.
\end{proof}
\fi
%%%

\begin{lemma} \label{2dim}
Let $u,v \in \BR(H) \backslash G$ be well-behaved, both starting with $t^{\pm 1}$, $a,b \in A \cup B$, $u'$ (resp., $v'$) be a proper suffix of $u$ (resp., $v$) and
$u''$ (resp., $v''$) be a proper prefix of $u$ (resp., $v$).
Let $\mu = \max\{ |u|, |v|\}$.
Then the set 
$$L(a,u',u,u'',v',v,v'',b) = \{ (x,y) \in \mathbb{N} \times \mathbb{N} \mid a u' u^x u'' =_H v' v^y v'' b \}$$
is semilinear.
Moreover, one can compute in polynomial time a semilinear representation whose magnitude is bounded by $\mathcal{O}(\gamma^{2} \mu^{4})$. 
\end{lemma}

\begin{proof}
	The proof is inspired by the proof of \cite[Lemma 8.3]{LOHREY2019} for hyperbolic groups.
The assumptions in the lemma imply that for all $x,y \in \mathbb{N}$ the words
$a u' u^x u''$ and $v' v^y v'' b$ are Britton-reduced (possibly after multiplying $a$ and $b$ with neighboring symbols from $\Gamma$.
For a word $w \in (\Gamma \cup \{t,t^{-1}\})^*$ we define $|w|_{t^{\pm 1}} = |w|_{t} + |w|_{t^{-1}}$ (the $t^{\pm 1}$-length of $w$).
Clearly, for Britton-reduced words $w,w'$ with $w =_H w'$ we have $|w|_{t^{\pm 1}} = |w'|_{t^{\pm 1}}$.

We will first construct an automaton $\mathcal{A}$ over the unary alphabet $\{ \# \}$, where $\#$ is a fresh letter, such that
$$
L(\mathcal{A}) = \{ \#^\ell \mid \exists x,y \in \mathbb{N} : a u' u^x u'' =_H v' v^y v'' b, \ell = |a u' u^x u''|_{t^{\pm 1}} \}.
$$
Moreover, the number of states of $\mathcal{A}$ is $\mathcal{O}(\mu^2 \cdot  \gamma)$. 
Roughly speaking, the automaton $\mathcal{A}$ verifies from left to right the existence of a van Kampen diagram of the form shown in Figure~\ref{fig-van-kampen}.
Thereby it stores the current connecting element (an element from $A \cup B$).
By the assumptions on $u$ and $v$ we can write both words as
$u = u_0 u_{1} \cdots u_{m-1}$, $v = v_{0} v_{1} \cdots v_{n-1}$ where $n$ and $m$ are even, $u_{i} \in \{t,t^{-1}\}$ for $i$ even and $u_{i} \in G$ for $i$ odd, and analogously for $v$.
Let us write $u' = u_{p} u_{p+1} \cdots u_{m-1}$, $u'' = u_{0} u_{1} \cdots u_{q-1}$,  $v' = v_{r} v_{r+1} \cdots v_{n-1}$, $v'' = v_{0} v_{1} \cdots v_{s-1}$.
We set $p = m$ if $u'$ is empty and $q = 0$ if $u''$ is empty and similarly for $r$ and $s$.
We will first consider the case that $p \equiv r \mod 2$ and $q \equiv s \mod 2$; other cases are just briefly sketched at the end of the proof.

The state set of $\mathcal{A}$ is
$$
Q = \{ (c,i,j) \mid c \in A \cup B,  0 \leq i < m, 0 \leq j < n, i \equiv j \bmod 2 \}.
$$
The initial state is $(a,p \bmod m, r \bmod n)$ and the only final state is $(b, q \mod m, s \mod n)$. Finally, $\mathcal{A}$ contains the following transitions for
 $c_1, c_2 \in A \cup B$ such that $c_1 u_{i} =_H v_{j} c_2$ (in case $i$ and $j$ are odd, this must be an identity in $G$ since $u_i, v_j \in G$):
\begin{itemize}
\item $(c_1,i,j)  \xrightarrow{\#} (c_2, i+1 \mod m, j+1 \mod n)$ if $i$ and $j$ are even,
\item $(c_1,i,j)  \xrightarrow{1} (c_2, i+1 \mod m, j+1 \mod n)$ if $i$ and $j$ are odd.
\end{itemize}
The number of states of $\mathcal{A}$ is $\mathcal{O}(\gamma \cdot \mu^2)$. 
 If $p \equiv r \mod 2$ does not hold, then we have to introduce a fresh initial state $q_0$.
 Assume that for instance $p$ is odd and $r$ is even. Thus $u_p$ belongs to $G$ whereas $v_r$ is $t$ or $t^{-1}$.
 Then we add all transitions $q_0 \xrightarrow{1} (c,p+1 \mod m, r\mod n)$ for every $c \in A \cup B$ with $a u_{p} =_G c$.
 If $q \equiv s \mod 2$ does not hold, then we have to add a fresh final state $q_f$.
 
 The rest of the argument is the same as in Remark~\ref{rem-2dim-words}, we only have to replace
 the length of words by the $t^{\pm 1}$-length. We obtain a semilinear representation of $L(a,u',u,u'',v',v,v'',b)$
 of magnitude $\mathcal{O}(\gamma^2 \cdot \mu^4)$. 
\end{proof}

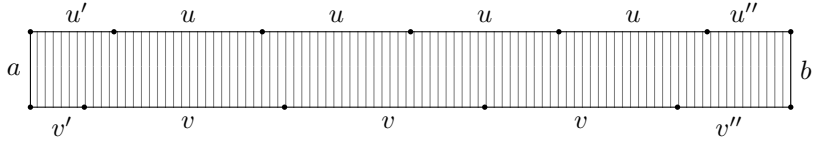
\begin{figure}[t]
\begin{tikzpicture}
\node (A) at (0,0) {};
\node (B) at (0,1) {};
\node (C) at (10,0) {};
\node (D) at (10,1) {};

\path[pattern=vertical lines, pattern color=gray] (A) rectangle (D);

\draw[shorten <= -3.5, shorten >= -3.5] (A) -- (B) node[midway,left] {$a$};
\draw[shorten <= -3.5, shorten >= -3.5] (D) -- (C) node[midway,right] {$b$};

\draw[shorten <= -3.5, shorten >= -3.5] (B) to
node[pos=0.05,above] (E) {$u'$}
node[pos=0.10] (F) {}
node[pos=0.2,above] (G) {$u$}
node[pos=0.3] (H) {}
node[pos=0.4,above] (I) {$u$}
node[pos=0.5] (J) {}
node[pos=0.6,above] (K) {$u$}
node[pos=0.7] (L) {}
node[pos=0.8,above] (M) {$u$}
node[pos=0.9] (N) {}
node[pos=0.95,above] (O) {$u''$}
(D);

\draw[shorten <= -3.5, shorten >= -3.5] (A) to
node[pos=0.03,below] (P) {$v'$}
node[pos=0.06] (Q) {}
node[pos=0.2,below] (R) {$v$}
node[pos=0.33] (S) {}
node[pos=0.47,below] (T) {$v$}
node[pos=0.6] (U) {}
node[pos=0.73,below] (V) {$v$}
node[pos=0.86] (W) {}
node[pos=0.93,below] (X) {$v''$}
(C);

 \fill (A) circle (1pt);
 \fill (B) circle (1pt);
 \fill (C) circle (1pt);
 \fill (D) circle (1pt);
 
 \fill (F) circle (1pt);
 \fill (H) circle (1pt);
 \fill (J) circle (1pt);
 \fill (L) circle (1pt);
 \fill (N) circle (1pt);
 
 \fill (Q) circle (1pt);
 \fill (S) circle (1pt);
 \fill (U) circle (1pt);
 \fill (W) circle (1pt);
\end{tikzpicture}
\caption{In the proof of Lemma~\ref{2dim} the automaton stores the connecting elements, i.e. checks the rectangles.}
\end{figure}

We now define $1$-reducible tuples for HNN-extensions similar to the case of graph products.
Again we identify tuples that can be obtained from each other by inserting/deleting $1$'s at
arbitrary positions.

\begin{defi} \label{def-reduction-HNN}
We define a reduction relation on tuples over $\BR(H)$ of arbitrary length.
Take $u_1, u_2, \ldots, u_m \in \BR(H)$. Then we have
\begin{itemize}
\item  $(u_1, u_2, \ldots, u_i, a, u_{i+1}, \ldots, u_m)  \to (u_1, \ldots, u_{i-1},b, u_{i+2}, \ldots, u_m)$  if both $u_i$ and $u_{i+1}$ contain
$t$ or $t^{-1}$ and  $u_i a u_{i+1} =_H b$ for $a,b \in A \cup B$ (a {\em generalized cancellation step}),
\item $(u_1, u_2, \ldots, u_m) \to (u_1, \ldots, u_{i-1}, g, u_{i+2}, \ldots, u_m)$ if
$u_i, u_{i+1} \in \Gamma$ and $g =_G u_i u_{i+1} \in G$ 
\end{itemize}
If $g \neq 1$ then we call the last rewrite step  an {\em atom creation}.
A concrete sequence of these rewrite steps leading to the empty tuple
is a {\em reduction} of $(u_1, u_2, \ldots, u_m)$. 
If such a sequence exists, the tuple is called {\em $1$-reducible}.
\end{defi}

A reduction of a tuple $(u_1, u_2, \ldots, u_m)$ can be seen as
a witness for the fact that $u_1 u_2 \cdots u_m =_H 1$. 
On the other hand, $u_1 u_2 \cdots u_m =_H 1$ does not necessarily imply
that $u_1, u_2, \ldots, u_m$ has a reduction (as seen for graph products). But we can show that every sequence which multiplies to $1$ in $H$ can be refined
(by factorizing the elements of the sequence) such that the resulting refined sequence has a reduction. We say that the tuple $(v_1, v_2, \dots, v_n)$ is a \emph{refinement} of the tuple $(u_1, u_2, \dots, u_m)$ if there exist factorizations $u_i=u_{i,1} \cdots u_{i,k_i}$ in $(\Gamma \cup \{t,t^{-1}\})^*$ such that
$(v_1, v_2, \dots, v_n)=(u_{1,1}, \dots, u_{1,k_1}, \dots, u_{m,1}, \dots, u_{m,k_m})$.

\begin{lemma} \label{Britton-lemma-reduction}
Let $m \geq 2$ and $u_1, u_2, \ldots, u_m \in \BR(H)$. If $u_1 u_2 \cdots u_m = 1$ in
$H$, then there exists a $1$-reducible refinement of $(u_1, u_2, \dots , u_m)$ that has length at most
$7m-12\leq 7m$ and there is a reduction of this refinement with at most $4m-8$ atom creations.
\end{lemma} 

\begin{proof}
We prove the lemma by induction on $m$. The induction is similar to the one of Lemma~\ref{lemma-reduction-free-product}.
The case $m=2$ is trivial (we must have $u_2 = u_1^{-1}$). If $m \geq 3$ then by Lemma~\ref{red uv} we can factorize $u_1$ and $u_2$ in
$(\Gamma \cup \{t,t^{-1}\})^*$ as
$u_1=u'_1 g_1 r$ and $u_2= s g_2 u'_2$ such that $rs=_H c \in A \cup B$, $g_1, g_2 \in G$ 
and $u_1u_2 =_H  u'_1 g u'_2\in \BR(H)$ for $g =g_1c g_2 \in G$. The words $r$ and $s$ are either both empty (in which case we have $c=1$)
or $r$ starts with some $t^{\varepsilon}$ and $s$ ends with $t^{-\varepsilon}$. 

By induction hypothesis, for the tuple $(u'_1g u'_2, u_3, \dots, u_m)$ there is a $1$-reducible refinement
\begin{equation}\label{seq-HNN}
(v_1, \dots, v_k, u_{3,1}, \dots, u_{3,k_3}, \dots, u_{m,1}, \dots, u_{m,k_m}),
\end{equation}
with $4(m-1)-8$ atom creations, where $k+\sum_{i=3}^m k_i \leq 7(m-1)-12$ and $u'_1gu'_2=v_1\cdots v_k$ in $(\Gamma \cup \{t,t^{-1}\})^*$.
Since $g \in G$, there exists $1\leq i \leq k$, such that $v_i=v_{i,1} g v_{i,2}$, $u'_1 = v_1 \cdots v_{i-1} v_{i,1}$ and $u'_2 = v_{i,2} v_{i+1} \cdots v_k$.
Now we replace $v_i$ by $v_{i,1}, g, v_{i,2}$ in the above refinement~\eqref{seq-HNN}.  
If there exists $u_{j,\ell}$ such that $v_i$ and $u_{j,\ell}$ cancel out in a generalized cancellation step in the $1$-reduction of \eqref{seq-HNN}
then there exist $a,b \in A \cup B$ such that $v_{i,1} g v_{i,2} a u_{j,\ell} = v_i a u_{j,\ell}   =_H b$. The generalized cancellation replaces $v_i,a,u_{j,\ell}$ by $b$.

Recall that $v_i$ and $u_{j,\ell}$ are both Britton-reduced.
By Lemma~\ref{boat} we can factorize $u_{j,\ell}$ in $(\Gamma \cup \{t,t^{-1}\})^*$ as $u_{j,\ell} = w_1 g' w_2$ 
such that there exist connecting elements $a',b' \in A \cup B$ with $v_{i,2} a w_1 =_H a'$, $v_{i,1} b' w_2 =_H b$, and $ga'g' =_G b'$; 
see Figure~\ref{fig-ujl-vi}. This yields the refined tuple
$$
(v_1, \ldots, v_{i-1}, v_{i,1}, g_1, r, s, g_2, v_{i,2}, v_{i+1}, \ldots, v_k, \tilde{u}_{3,1}, \dots, \tilde{u}_{3,k_3}, \dots, \tilde{u}_{m,1}, \dots, \tilde{u}_{m,k_m}),
$$
of $(u_1, u_2, \dots , u_m)$, where $\tilde{u}_{j,\ell} = w_1, g', w_2$ and $\tilde{u}_{p,q} = u_{p,q}$ in all other cases.
The length of this tuple is at most $k + 7 + \sum_{i=3}^m k_i \leq 7m-12$. 
The above tuple is also $1$-reducible: First, $r,s$ is replaced by $c$ in a generalized cancellation step. Then, after at most two atom creations\footnote{Note that each of $c, g_1, g_2$ can be $1$, in which case the number of atom creations is smaller than two.}
we obtain the tuple
$$
(v_1, \ldots, v_{i-1}, v_{i,1}, g, v_{i,2}, v_{i+1}, \ldots, v_k, \tilde{u}_{3,1}, \dots, \tilde{u}_{3,k_3}, \dots, \tilde{u}_{m,1}, \dots, \tilde{u}_{m,k_m}) .
$$
At this point, we can basically apply the fixed reduction of \eqref{seq-HNN}. The generalized cancellation $v_i,a,u_{j,\ell} \to b$
is replaced by the sequence 
$$
v_{i,1}, g, v_{i,2}, a, w_1, g', w_2 \to v_{i,1}, g, a', g', w_2 \to  v_{i,1}, ga', g', w_2 \to  v_{i,1}, b', w_2 \to b
$$
which contains at most two atom creations. Hence, the total number of atom creations is at most $4 + 4(m-1)-8 = 4m-8$.
This concludes the proof of the lemma.
\end{proof}

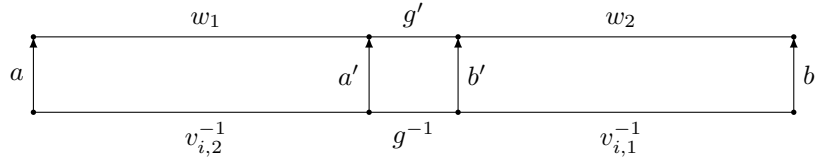
\begin{figure}
\begin{tikzpicture}
\node (A) at (0,0) {};
\node (B) at (0,1) {};
\node (C) at (10,0) {};
\node (D) at (10,1) {};

\tikzset{>=latex}
\draw[->,shorten <= -3.5, shorten >= -3.5] (A) -- (B) node[midway,left] {$a$};
\draw[<-,shorten <= -3.5, shorten >= -3.5] (D) -- (C) node[midway,right] {$b$};

\draw[shorten <= -3.5, shorten >= -3.5] (B) to
node[pos=0.22,above] (E) {$w_1$}
node[pos=0.44] (F) {}
node[pos=0.5,above] (G) {$g'$}
node[pos=0.56] (H) {}
node[pos=0.78,above] (I) {$w_2$}
(D);

\draw[shorten <= -3.5, shorten >= -3.5] (A) to
node[pos=0.22,below] (J) {$v_{i,2}^{-1}$}
node[pos=0.44] (K) {}
node[pos=0.5,below] (L) {$g^{-1}$}
node[pos=0.56] (M) {}
node[pos=0.78,below] (N) {$v_{i,1}^{-1}$}
(C);

\draw[->,shorten <= -3.5, shorten >= -3.5] (K) -- (F) node[midway,left] {$a'$};
\draw[->,shorten <= -3.5, shorten >= -3.5] (M) -- (H) node[midway,right] {$b'$};

 \fill (A) circle (1pt);
 \fill (B) circle (1pt);
 \fill (C) circle (1pt);
 \fill (D) circle (1pt);
 
 \fill (F) circle (1pt);
 \fill (H) circle (1pt);
 \fill (K) circle (1pt);
 \fill (M) circle (1pt);
\end{tikzpicture}
\caption{\label{fig-ujl-vi}The generalized cancellation step for $v_i$ and $u_{j,\ell}=w_1g'w_2$ in the proof of Lemma~\ref{Britton-lemma-reduction}.}
\end{figure}

Now we can prove the next main theorem.

\begin{theorem}\label{HNNbound}
Let $A,B$ be finite subgroups of $G$ and let $\varphi\colon A\to B$ be an isomorphism.
If $G$ is knapsack-semilinear, then the HNN-extension
$H$ of $G$ (with respect to the isomorphism $\varphi$) is knapsack-semilinear as well.
Moreover, we have $\mathsf{K}_{H,\Sigma}(n,m) \le \max\{\mathsf{K}_1, \mathsf{K}_2\}$ with 
\begin{equation*}
\mathsf{K}_1 = \mathsf{K}_{G,\Sigma}(24mn,m) \text{ and } 
\mathsf{K}_2 \le \mathcal{O}(\gamma^{2} m  n^{4})  ,
\end{equation*}
where $\gamma=|A|$.
\end{theorem}

\begin{proof}
We will follow the idea of the proof of Theorems~\ref{thm-main-technical} and~\ref{thm-cor-technical}, respectively. 
We first consider the case where every period $u_i$ is either an atom
or well-behaved and starts with $t^{\pm 1}$.
Again we are going through the six steps. 
For simplicity, we write $\mathsf{K}$ instead of $\mathsf{K}_{G,\Sigma}$.

\medskip
\noindent
{\em Step 1.} This step is carried out in the same way as in 
the proof of Theorem~\ref{thm-main-technical}. 

\medskip
\noindent
{\em Step 2.} Here we can use Lemma~\ref{Britton-lemma-reduction} instead of Lemma~\ref{lemma-reduction},
which yields the upper bound of $14m$ on the number of factors in our refinement of 
$u_1^{x_1} v_1 u_2^{x_2} v_2 \cdots u_m^{x_m} v_m$ (where powers $u_i^{x_i}$ with $i \in N_1$ have been removed). The number of atom creations is at most $8m-8$.

\medskip
\noindent
{\em Step 3.} 
This step is copied from the  proof of Theorem~\ref{thm-main-technical}. 
We obtain for the variables $x_i$ with $i \in N_a$ the semilinear constraint
$(x_i)_{i \in N_a} \in S_a$ where $S_a$ is of magnitude at most
$\mathsf{K}(8m \norm{\E},m)$.

\medskip
\noindent
{\em Step 4.} This step is analogous to the proof of Theorem~\ref{thm-cor-technical}. The only difference is that
the 2-dimensional knapsack instances are produced by the generalized cancellation steps from Definition~\ref{def-reduction-HNN}.
We arrive at a statement of the following form:
there exist integers $x_{i,j} \geq 0$ ($i \in N_{\overline{a}}$, $j \in K_i$) such that 
the following hold:
\begin{enumerate}[(a)]
\item $x_i = c'_i +  \sum_{j \in K_i} x_{i,j}$ for $i \in N_{\overline{a}}$,
\item $a_{i,j} p_{i,j} u_i^{x_{i,j}} s_{i,j} =_H s_{k,\ell}^{-1} (u_k^{-1})^{x_{k,\ell}} p_{k,\ell}^{-1} b_{i,j}$ for all $(i,j,k,\ell) \in M$,
\item $(x_i)_{i \in N_a} \in S_a$.
\end{enumerate}
Here, $K_i \subseteq [1,k_i]$ is a set of size at most $k_i \leq 14m$, 
$M$ is a matching relation (with $(i,j,k,\ell) \in M$ if and only if $(k,\ell,i,j) \in M$), $a_{i,j}, b_{i,j} \in A \cup B$,
and $c'_i \le \mathcal{O}(m \cdot \lambda(\E))$.
Every word $p_{i,j}$ is a suffix of $u_{i,j}$ and every $s_{i,j}$ is a prefix of $u_{i,j}$. In particular, $p_{i,j}$ and $s_{i,j}$ have length at most $\lambda(\E)$.

\medskip
\noindent
{\em Step 5.} The remaining two-dimensional knapsack equations from point (b)
are eliminated with Lemma~\ref{2dim}. Every equation 
$$
a_{i,j} p_{i,j} u_i^{x_{i,j}} s_{i,j} =_H s_{k,\ell}^{-1} (u_k^{-1})^{x_{k,\ell}} p_{k,\ell}^{-1} b_{i,j}
$$
can be nondeterministically replaced by a semilinear constraint for $x_{i,j}$ and $x_{k,\ell}$ of magnitude $\mathcal{O}(\gamma^{2} \lambda(\E)^{4})$.
By substituting these semilinear constraints in the above equations (a) for the $x_i$ (as we did in the proof of  Theorem~\ref{thm-main-technical}),
we obtain for the variables $x_i$ ($i \in N_{\overline{a}}$) a semilinear constraint of magnitude 
$\mathcal{O}(m \cdot \gamma^{2} \cdot \lambda(\E)^{4})$.
This leads to a semilinear representation for $\Sol_{H}(\E)$ of magnitude at most
$\max\{ \mathsf{K}_1, \mathsf{K}_2 \}$, where 
\begin{equation*}
\mathsf{K}_1 = \mathsf{K}(8m \norm{\E},m) \text{ and } 
\mathsf{K}_2 \le \mathcal{O}(m \cdot \gamma^{2} \cdot \norm{\E}^{4}) .
\end{equation*}
{\em Step 6.} For the preprocessing we can apply Lemma~\ref{Britton-reduced-powers} to each period $u_i$.
Hence, we  just have to replace $\norm{\E}$ by $3 \norm{\E}$ in the above bounds, which yields
the statement of the theorem.
\end{proof}

The following result is obtained analogously to Theorem~\ref{thm:free-product-NP-reduction}.
\begin{theorem} \label{thm:HNN-NP-reduction}
The knapsack problem for $\langle G, t \mid t^{-1}at=\varphi(a)~(a\in A)\rangle$ (with $A$ finite) is nondeterministically polynomial time reducible to the knapsack problem for $G$.
\end{theorem}
The consequence of Theorem~\ref{thm:HNN-NP-reduction} that solvability of the
knapsack problem in $\mathsf{NP}$ is passed on from $G$ to $H$ was shown using
different methods in the extended abstract~\cite{LohreyZetzsche2016a}.

\subsection{Amalgamated products preserve knapsack semilinearity}

Using our results for free products and HNN-extensions, we can easily deal with
amalgamated products. For
$i\in\{1,2\}$, let $G_i=\langle \Sigma_i \mid R_i\rangle$ be a finitely
generated group with $\Sigma_1 \cap \Sigma_2 = \emptyset$
and let $A$ be a finite group that is embedded in each $G_i$ via the injective morphism $\varphi_i\colon A\to G_i$ for
$i\in\{1,2\}$.
Then, the \emph{amalgamated product with
identified subgroup $A$} is the group
\[ \langle \Sigma_1\uplus \Sigma_2 \mid R_1\uplus R_2\cup \{\varphi_1(a)=\varphi_2(a) \mid a\in A\}\rangle. \]
This group is usually written as
$$
 \langle G_1*G_2 \mid \varphi_1(a)=\varphi_2(a)~(a\in A)\rangle.
 $$
 or just $G_1 *_A G_2$.
Note that the amalgamated product depends on the morphisms $\varphi_i$,
although they are omitted in the notation $G_1 *_A G_2$.  

From Theorem~\ref{HNNbound}, we can easily deduce a similar result for amalgamated products:

\begin{theorem}
Let $G_1$ and $G_2$ be finitely generated groups with a common subgroup $A$. Let $\mathsf{K}(n,m)$ be the 
pointwise maximum of the functions $\mathsf{K}_{G_1,\Sigma_1}$ and $\mathsf{K}_{G_2,\Sigma_2}$. Furthermore, let $\gamma=|A|$ and let $G$ be the amalgamated product $G_1 *_A G_2$. 
Then with $\Sigma=\Sigma_1\cup \Sigma_2$ we have $\mathsf{K}_{G,\Sigma}(n,m) \le \max\{\mathsf{K}_1, \mathsf{K}_2, \mathsf{K}_3\}$ where
\begin{equation*}
\mathsf{K}_1 = \mathsf{K}_{G,\Sigma}(144m^2n,m) ,
\mathsf{K}_2 \le \mathcal{O}(m^5 n^4)  \text{ and } 
\mathsf{K}_3 \le \mathcal{O}(m  \cdot \gamma^{2} \cdot n^{4}) .
\end{equation*}
\end{theorem}

\begin{proof}
For the proof we will make use of the previous theorem and Theorem~\ref{thm-cor-technical} for free products.

It is well-known~\cite[Theorem~2.6, p.~187]{LySch77} that $G_1 *_A G_2$ can be
embedded into the HNN-extension
\[ H = \langle G_1*G_2, t \mid t^{-1}\varphi_1(a)t=\varphi_2(a)~(a\in A) \rangle \]
by the morphism $\Phi\colon G_1 *_A G_2\to H$ with
\[ \Phi(g)=\begin{cases} t^{-1}gt & \text{if $g\in G_1$} \\ g & \text{if $g\in G_2$}. \end{cases} \]
Obviously we have $\mathsf{K}_{G,\Sigma}(n,m)\leq \mathsf{K}_{H,\Sigma}(n,m)$.
Hence we can calculate the bound by first getting the bound for the free product $G_1*G_2$ and then proceeding with the HNN-extension.
Theorem~\ref{thm-cor-technical} tells us that
$\mathsf{J}(n,m)=\mathsf{K}_{G_1*G_2,\Sigma}(n,m)\leq \max \{ \mathsf{K}'_1,\mathsf{K}'_2 \}$, where $\mathsf{K}'_1=\mathsf{K}(6mn,m)$ and $\mathsf{K}'_2 \le \mathcal{O}(m n^4)$. To obtain $\mathsf{K}_{H,\Sigma}(n,m)$, we make use of Theorem~\ref{HNNbound}. We have
$\mathsf{K}_{H,\Sigma}(n,m) \leq \max \{ \mathsf{J}_1,\mathsf{J}_2 \}$,
where $\mathsf{J}_1=\mathsf{J}(24mn,m)$ and $\mathsf{J}_2\leq \mathcal{O}(\gamma^{2} m n^{4})$.
Since the function $\mathsf{J}(n,m)$ appears in $\mathsf{J}_1$, we have to substitute by what we calculated before. More precisely we have to make the substitution $n \mapsto 24mn$ for the values $\mathsf{K}'_1$ and $\mathsf{K}'_2$.
This yields 
$\mathsf{K}_{H,\Sigma}(n,m) \leq \max \{ \mathsf{K}_1,\mathsf{K}_2, \mathsf{K}_3 \}$, where
\begin{align*}
\mathsf{K}_1 &= \mathsf{K}(144m^2n,m) ,\\
\mathsf{K}_2 &\leq \mathcal{O}(m^5 n^4) ,\\
\mathsf{K}_3 &= \mathsf{J}_2 \leq \mathcal{O}(\gamma^{2} m n^{4}) .
\end{align*}
This finishes the proof of the theorem.
\end{proof}
From Theorems~\ref{thm:free-product-NP-reduction} and \ref{thm:HNN-NP-reduction} 
and the above embedding of $G_1 *_A G_2$ in $\langle G_1*G_2, t \mid t^{-1}\varphi_1(a)t=\varphi_2(a)~(a\in A) \rangle$
we obtain:
\begin{theorem}\label{thm:amalgamated-product-NP-reduction}
The knapsack problem for $G_1*_A G_2$ (with $A$ finite) is nondeter\-ministically polynomial time reducible to the knapsack problems for $G_1$ and $G_2$.
\end{theorem}
As before, the consequence of Theorem~\ref{thm:amalgamated-product-NP-reduction} that
solvability of the knapsack problem in $\mathsf{NP}$ is passed on from $G_i$
for $i=1,2$ to the amalgamated product $G_1 *_A G_2$ was shown using different
methods in the extended abstract~\cite{LohreyZetzsche2016a}.

\section{Part 3: Knapsack in finite extensions}

We say that $H$ is a \emph{finite extension of $G$} if $G$ is a finite-index
subgroup of $H$. 

\begin{theorem} \label{thm-finite-index}
Let $G$ be a finitely generated group with a finite symmetric generating set $\Sigma$ and let $H$ be a finite extension of $G$ (hence, it is finitely generated too)
with the finite symmetric generating set $\Sigma' = \Sigma \cup (C \setminus \{1\}) \cup (C \setminus \{1\})^{-1}$, where $C$ is a set of coset representatives with $1 \in C$.
Let $\ell = |C|$ be the index of  $G$ in $H$.
 If $G$ is knapsack-semilinear then $H$ is knapsack-semilinear too and we have the bounds
\begin{eqnarray}
\mathsf{E}_{H,\Sigma'}(n,m) &\leq & \ell \cdot \mathsf{E}_{G,\Sigma}(\mathcal{O}(\ell^2 n),m) +2\ell, \label{eq-finite-index-E} \\
\mathsf{K}_{H,\Sigma'}(n,m) & \leq & \ell \cdot \mathsf{K}_{G,\Sigma}(\mathcal{O}(\ell^2 n),m) +2\ell \label{eq-finite-index-K}.
\end{eqnarray}
\end{theorem}

\begin{proof}
Suppose we are given an exponent expression
\begin{equation}
e=u_1^{x_1}v_1\cdots u_m^{x_m}v_m \label{fi:input}
\end{equation}
in $H$ where the $u_i$ and $v_i$ are words over $\Sigma'$.
Let $n$ be the length of $e$.
As a first step, we guess which of the variables $x_i$ assume a value smaller
than $\ell$.  For those that do, we can guess the value and merge the resulting power with 
the $v_i$ on the right. This increases the size of the instance by at most a factor
of $\ell$, which is a constant. 
At the end we will compensate this by applying the substitution $n\mapsto \ell n$. 
Hence, from now on, we only look for $H$-solutions $\sigma$ to $e=1$
where $\sigma(x_i) \ge \ell$ for $1\le i\le m$.

Next we guess the cosets of the prefixes of $u_1^{x_1}v_1\cdots u_m^{x_m}v_m$, i.e., 
we guess coset representatives $c_1,d_1,\ldots,c_{m-1},d_{m-1},c_m \in C$
and restrict to $H$-solutions $\sigma$ to $e=1$ such that $u_1^{\sigma(x_1)}v_1\cdots u_i^{\sigma(x_i)}\in Gc_i$
and $u_1^{\sigma(x_1)}v_1\cdots u_i^{\sigma(x_i)}v_i\in Gd_i$ for $1\le i\le m$. Here, we set $d_m=1$.
Equivalently, we only consider $H$-solutions $\sigma$ where 
$d_{i-1}u_i^{x_i}c_i^{-1}$ and $c_{i}v_id_i^{-1}$
all belong to $G$ for $1\le i\le m$. Here, we set $d_0 = 1$.
We can verify in polynomial time that all
$c_iv_id_i^{-1}$ ($1\le i\le m$) belong to $G$. It remains to describe the set of all $H$-solutions $\sigma$ for $e=1$
that fulfill the following constraints for all $1\le i\le m$:
\begin{equation} \label{constraints-finite-ext}
d_{i-1}u_i^{\sigma(x_i)}c_i^{-1}\in G \text{ and } \sigma(x_i) \ge \ell .
\end{equation}
For $1 \leq i \leq m$ consider the function $f_i\colon C\to C$, which is defined so that for each
$c\in C$, $f_i(c)$ is the unique element $d\in C$ with $cu_id^{-1}\in G$. Note
that we can compute $f_i$ in polynomial time if $G$ and $H$ are fixed groups
(all we need for this is a table that specifies for each $c \in C$ and $a \in \Sigma'$ the coset
representative of $ca$; this is a fixed table that does not depend on the input).
Then there are numbers $1\le k_i\le
\ell$ such that $f_i^{\ell+k_i}(d_{i-1})=f_i^\ell(d_{i-1})$.  With this notation, we
have $d_{i-1}u_i^{z}c_i^{-1}\in G$ if and only if $f_i^{z}(d_{i-1})=c_i$ for all $z \in \N$.

We may assume that there is a $z\ge \ell$ with $f_i^{z}(d_{i-1})=c_i$;
otherwise, there is no $H$-solution for $e=1$ fulfilling the above constraints  \eqref{constraints-finite-ext}
and we have a bad guess. Therefore,
there is a $0\le r_i<k_i$ such that $f_i^{\ell+r_i}(d_{i-1})=c_i$. This means that for all $z \ge \ell$, we
have $d_{i-1}u_i^{z}c_i^{-1}\in G$ if and only if  $f_i^{z}(d_{i-1})=c_i$ if and only if $z=\ell+k_i\cdot
y+r_i$ for some $y\ge 0$.  This allows us to construct an exponent expression
over $G$.

Let $e_i=f_i^\ell(d_{i-1})$. Then, the words $d_{i-1}u_i^\ell e_i^{-1}$,
$e_iu_i^{k_i}e_i^{-1}$, and $e_iu_i^{r_i}c_i^{-1}$ all represent elements  of $G$. Moreover,
for all $y_i \geq 0$ and $z_i=\ell+k_i\cdot y_i+r_i$ ($1 \leq i \leq m$), we have
\begin{align*}
u_1^{z_1}v_1\cdots u_m^{z_m}v_m &= \prod_{i=1}^m d_{i-1} u_i^{\ell+k_i\cdot y_i+r_i}c_i^{-1}c_i v_i d_i^{-1}  \\
&= \prod_{i=1}^m (d_{i-1} u_i^\ell e_i^{-1})(e_iu_i^{k_i} e_i^{-1})^{y_i} (e_i u_i^{r_i} c_i^{-1}c_i v_i d_i^{-1})
\end{align*}
and each word in parentheses represents an element of $G$. 
Hence, we can define the exponent expression
\begin{equation*} 
e' = \prod_{i=1}^m (d_{i-1} u_i^\ell e_i^{-1})(e_iu_i^{k_i} e_i^{-1})^{x_i} (e_i u_i^{r_i} v_i d_i^{-1})
\end{equation*}
over the group $G$. From the above consideration we obtain
\begin{align}
\Sol_H(e) \cap \{ \sigma : X_e \to H \mid \sigma \text{ satisfies the constraints } \eqref{constraints-finite-ext} \} & = \nonumber \\
\{ \sigma \mid \sigma(x_i) = k_i \cdot \sigma'(x_i) + (\ell+r_i) \text{ for some } \sigma' \in \Sol_G(e') \} . & \label{sigma'-sigma}
\end{align}
The set in \eqref{sigma'-sigma} is semilinear by assumption and 
since all $k_i$ and $r_i$ are bounded by $\ell$, we can bound its magnitude by
$\ell \cdot \magn(\Sol_G(e')) + 2\ell$. Moreover, we have $\deg(e') = \deg(e)$. It remains to bound $\norm{e'}$.
For this, we first have to rewrite the words $d_{i-1} u_i^\ell e_i^{-1}$, $e_iu_i^{k_i} e_i^{-1}$, and $e_i u_i^{r_i} v_i d_i^{-1}$
(which represent elements of $G$) into words over $\Sigma$. This increases the length of the words only by a constant factor:
for every $c \in C$ and every generator $a \in \Sigma'$ there exists a fixed word $w_{c,a} \in \Sigma^*$ and $d_{c,a} \in C$ 
such that $ca = w_{c,a} d_{c,a}$ holds in $H$. 
After this rewriting we have $\norm{e'} \leq \mathcal{O}(\ell n)$, which implies
$\magn(\Sol_G(e')) \leq \mathsf{E}_{G,\Sigma}(\mathcal{O}(\ell n),m)$. This yields the bound
$\ell \cdot \mathsf{E}_{G,\Sigma}(\mathcal{O}(\ell n),m) + 2\ell$
for the magnitude of the semilinear set in \eqref{sigma'-sigma}.
Applying the substitution $n \mapsto \ell n$ from the first step finally yields \eqref{eq-finite-index-E}.
The corresponding bound  \eqref{eq-finite-index-K} for 
knapsack expressions can be shown in the same way: Note that in the above transformation of $e$ into $e'$
we do not duplicate variables.
\end{proof}
From the above proof we also obtain the following result (similarly to Theorem~\ref{thm:free-product-NP-reduction}).

\begin{theorem} \label{thm:finite-NP-reduction}
The knapsack problem for a finite extension of a group $G$ is nondeterministically polynomial time reducible to the knapsack problem for $G$.
\end{theorem}
The consequence of Theorem~\ref{thm:finite-NP-reduction} that solvability of
the knapsack problem in $\mathsf{NP}$ is passed on from $G$ to finite
extensions of $G$ has also appeared in the extended
abstract~\cite{LohreyZetzsche2016a}.

\section{Future work}

In \cite{GanardiKLZ18} it was shown that the class of knapsack-semilinear groups is also closed under restricted wreath products.
It remains to bound the function $\mathsf{K}_{G \wr H}(n,m)$ in terms of the functions $\mathsf{K}_{G}(n,m)$  and $\mathsf{K}_{H}(n,m)$.
Looking into the proof in \cite{GanardiKLZ18}  reveals that the Presburger formula that describes the solution set for 
a knapsack equation over $G \wr H$ involves a quantifier alternation. One therefore has to investigate to what extent quantifier
alternations blow-up the magnitude of semilinear sets.

\medskip

\paragraph{\bf Acknowledgement.}
The first and second author were supported by the DFG research project LO 748/12-1. 
%\bibliographystyle{plainurl}
%\bibliography{bib}

\def\cprime{$'$} \def\cprime{$'$}

\bibliographystyle{plain}
\bibliography{bib}

\def\cprime{$'$} \def\cprime{$'$}
\begin{thebibliography}{10}

\bibitem{Adian2010}
Sergei~I Adian.
\newblock The {Burnside} problem and related topics.
\newblock {\em Russian Mathematical Surveys}, 65(5):805--855, 2010.

\bibitem{AllGre73}
Reg Allenby and Robert~John Gregorac.
\newblock On locally extended residually finite groups.
\newblock In {\em Conference on Group Theory (Univ. Wisconsin-Parkside,
  Kenosha, Wis., 1972)}, number 319 in Lecture Notes in Mathematics, pages
  9--17. Springer, Berlin, 1973.

\bibitem{BabaiBCIL96}
L{\'a}szl{\'o} Babai, Richard Beals, James Cai, G\'{a}bor Ivanyos, and Eugene
  Luks.
\newblock Multiplicative equations over commuting matrices.
\newblock In {\em Proceedings of SODA 1996}, pages 498--507. {ACM/SIAM}, 1996.

\bibitem{Beier}
Simon Beier, Markus Holzer, and Martin Kutrib.
\newblock On the descriptional complexity of operations on semilinear sets.
\newblock In {\em Proceedings of the 15th International Conference on Automata
  and Formal Languages, {AFL} 2017}, volume 252 of {\em {EPTCS}}, pages 41--55,
  2017.

\bibitem{BGZ21}
Pascal Bergstr{\"{a}}{\ss}er, Moses Ganardi, and Georg Zetzsche.
\newblock A characterization of wreath products where knapsack is decidable.
\newblock In {\em Proceedings of the 38th International Symposium on
  Theoretical Aspects of Computer Science, STACS 2021}, volume 187 of {\em
  LIPIcs}, pages 11:1--11:17. Schloss Dagstuhl - Leibniz-Zentrum f{\"{u}}r
  Informatik, 2021.

\bibitem{BeMaSa89}
Alberto Bertoni, Giancarlo Mauri, and Nicoletta Sabadini.
\newblock {M}embership problems for regular and context free trace languages.
\newblock {\em Information and Computation}, 82:135--150, 1989.

\bibitem{Bez98}
Vladimir~Nikolaevich Bezverkhnii.
\newblock On the intersection subgroups {HNN}-groups.
\newblock {\em Fundamentalnaya i Prikladnaya Matematika}, 4(1):199--222, 1998.

\bibitem{BoOt93}
Ronald~V. Book and Friedrich Otto.
\newblock {\em String--Rewriting Systems}.
\newblock Springer, 1993.

\bibitem{chistikov_et_al:LIPIcs:2016:6263}
Dmitry Chistikov and Christoph Haase.
\newblock {The Taming of the Semi-Linear Set}.
\newblock In {\em Proceedings of the 43rd International Colloquium on Automata,
  Languages, and Programming, ICALP 2016}, volume~55 of {\em Leibniz
  International Proceedings in Informatics (LIPIcs)}, pages 128:1--128:13.
  Schloss Dagstuhl--Leibniz-Zentrum f\"ur Informatik, 2016.

\bibitem{Dehn11}
Max Dehn.
\newblock \"{U}ber unendliche diskontinuierliche {Gruppen}.
\newblock {\em Mathematische Annalen}, 71:116--144, 1911.
\newblock In German.

\bibitem{DiDu89}
Warren Dicks and Martin~J. Dunwoody.
\newblock {\em Groups Acting on Graphs}.
\newblock Cambridge University Press, 1989.

\bibitem{Die90lncs}
Voker Diekert.
\newblock {\em Combinatorics on Traces}, volume 454 of {\em Lecture Notes in
  Computer Science}.
\newblock Springer, 1990.

\bibitem{DieRoz95}
Voker Diekert and Grzegorz Rozenberg, editors.
\newblock {\em The Book of Traces}.
\newblock World Scientific, 1995.

\bibitem{DiLo08IJAC}
Volker Diekert and Markus Lohrey.
\newblock Word equations over graph products.
\newblock {\em International Journal of Algebra and Computation},
  18(3):493--533, 2008.

\bibitem{Dro2003}
Carl Droms.
\newblock A complex for right-angled {Coxeter} groups.
\newblock {\em Proceedings of the American Mathematical Society},
  131(8):2305--2311, 2003.

\bibitem{DudTre18}
Fedor~Anatolievich Dudkin and Alexander~Victorovich Treyer.
\newblock Knapsack problem for {Baumslag}--{Solitar} groups.
\newblock {\em Siberian Journal of Pure and Applied Mathematics}, 18(4):43--55,
  2018.

\bibitem{EisenbrandS06}
Friedrich Eisenbrand and Gennady Shmonin.
\newblock Carath{\'{e}}odory bounds for integer cones.
\newblock {\em Operations Research Letters}, 34(5):564--568, 2006.

\bibitem{ElberfeldJT11}
Michael Elberfeld, Andreas Jakoby, and Till Tantau.
\newblock Algorithmic meta theorems for circuit classes of constant and
  logarithmic depth.
\newblock {\em Electronic Colloquium on Computational Complexity {(ECCC)}},
  18:128, 2011.

\bibitem{FigeliusGLZ20}
Michael Figelius, Moses Ganardi, Markus Lohrey, and Georg Zetzsche.
\newblock The complexity of knapsack problems in wreath products.
\newblock In {\em Proceedings of the 47th International Colloquium on Automata,
  Languages, and Programming, ICALP 2020}, volume 168 of {\em LIPIcs}, pages
  126:1--126:18. Schloss Dagstuhl - Leibniz-Zentrum f{\"{u}}r Informatik, 2020.

\bibitem{FrenkelNU16}
Elizaveta Frenkel, Andrey Nikolaev, and Alexander Ushakov.
\newblock Knapsack problems in products of groups.
\newblock {\em Journal of Symbolic Computation}, 74:96--108, 2016.

\bibitem{GanardiKLZ18}
Moses Ganardi, Daniel K{\"{o}}nig, Markus Lohrey, and Georg Zetzsche.
\newblock Knapsack problems for wreath products.
\newblock In {\em Proceedings of 35th Symposium on Theoretical Aspects of
  Computer Science, {STACS} 2018}, volume~96 of {\em LIPIcs}, pages
  32:1--32:13. Schloss Dagstuhl - Leibniz-Zentrum f{\"u}r Informatik, 2018.

\bibitem{GS66}
Seymour Ginsburg and Edwin~Henry Spanier.
\newblock Semigroups, {P}resburger formulas and languages.
\newblock {\em Pacific Journal of Mathematics}, 16(2):285--296, 1966.

\bibitem{Gre90}
Elisabeth~R. Green.
\newblock {\em Graph Products of Groups}.
\newblock PhD thesis, The University of Leeds, 1990.

\bibitem{Haa11}
Christoph Haase.
\newblock {\em On the complexity of model checking counter automata}.
\newblock PhD thesis, University of {Oxford}, St Catherine's College, 2011.

\bibitem{HauLo11}
Niko Haubold and Markus Lohrey.
\newblock Compressed word problems in {HNN}-extensions and amalgamated
  products.
\newblock {\em Theory of Computing Systems}, 49(2):283--305, 2011.

\bibitem{HauLohHau13}
Niko Haubold, Markus Lohrey, and Christian Mathissen.
\newblock Compressed decision problems for graph products of groups and
  applications to (outer) automorphism groups.
\newblock {\em International Journal of Algebra and Computation}, 22(8), 2013.

\bibitem{HiNeNe49}
Graham Higman, Bernhard~H. Neumann, and Hanna Neumann.
\newblock Embedding theorems for groups.
\newblock {\em Journal of the London Mathematical Society. Second Series},
  24:247--254, 1949.

\bibitem{HoltLS19}
Derek~F. Holt, Markus Lohrey, and Saul Schleimer.
\newblock Compressed decision problems in hyperbolic groups.
\newblock In {\em Proceedings of the 36th International Symposium on
  Theoretical Aspects of Computer Science, {STACS} 2019}, volume 126 of {\em
  LIPIcs}, pages 37:1--37:16. Schloss Dagstuhl - Leibniz-Zentrum f{\"u}r
  Informatik, 2019.

\bibitem{Huet1980}
G\'{e}rard Huet.
\newblock Confluent reductions: Abstract properties and applications to term
  rewriting systems: Abstract properties and applications to term rewriting
  systems.
\newblock {\em Journal of the ACM}, 27(4):797--821, 1980.

\bibitem{KaSiSt06}
Mark Kambites, Pedro~V. Silva, and Benjamin Steinberg.
\newblock On the rational subset problem for groups.
\newblock {\em Journal of Algebra}, 309(2):622--639, 2007.

\bibitem{KaWeMy05}
Ilya Kapovich, Richard Weidmann, and Alexei Myasnikov.
\newblock Foldings, graphs of groups and the membership problem.
\newblock {\em International Journal of Algebra and Computation},
  15(1):95--128, 2005.

\bibitem{Karp72}
Richard~M. Karp.
\newblock Reducibility among combinatorial problems.
\newblock In R.~E. Miller and J.~W. Thatcher, editors, {\em Complexity of
  Computer Computations}, pages 85--103. Plenum Press, New York, 1972.

\bibitem{KaSo70}
Abraham Karrass and Donald Solitar.
\newblock The subgroups of a free product of two groups with an amalgamated
  subgroup.
\newblock {\em Transactions of the American Mathematical Society},
  150:227--255, 1970.

\bibitem{KaSo71}
Abraham Karrass and Donald Solitar.
\newblock Subgroups of {HNN} groups and groups with one defining relation.
\newblock {\em Canadian Journal of Mathematics}, 23:627--643, 1971.

\bibitem{KoenigLohreyZetzsche2015a}
Daniel K{\"o}nig, Markus Lohrey, and Georg Zetzsche.
\newblock Knapsack and subset sum problems in nilpotent, polycyclic, and
  co-context-free groups.
\newblock In {\em Algebra and Computer Science}, volume 677 of {\em
  Contemporary Mathematics}, pages 138--153. American Mathematical Society,
  2016.

\bibitem{KuLo05ijac}
Dietrich Kuske and Markus Lohrey.
\newblock Logical aspects of {Cayley}-graphs: the monoid case.
\newblock {\em International Journal of Algebra and Computation},
  16(2):307--340, 2006.

\bibitem{LehSch07}
J{\"o}rg Lehnert and Pascal Schweitzer.
\newblock The co-word problem for the {Higman-Thompson} group is context-free.
\newblock {\em Bulletin of the London Mathematical Society}, 39(2):235--241,
  2007.

\bibitem{LOHREY2019}
Markus Lohrey.
\newblock Knapsack in hyperbolic groups.
\newblock {\em Journal of Algebra}, 2019.

\bibitem{LohSen06icalp}
Markus Lohrey and G{\'e}raud S{\'e}nizergues.
\newblock Theories of {HNN}-extensions and amalgamated products.
\newblock In {\em Proceedings of the 33st International Colloquium on Automata,
  Languages and Programming, ICALP 2006}, volume 4052 of {\em Lecture Notes in
  Computer Science}, pages 681--692. Springer, 2006.

\bibitem{LohSen08}
Markus Lohrey and G{\'e}raud S{\'e}nizergues.
\newblock Rational subsets in {HNN}-extensions and amalgamated products.
\newblock {\em International Journal of Algebra and Computation},
  18(1):111--163, 2008.

\bibitem{LohreyZetzsche2016a}
Markus Lohrey and Georg Zetzsche.
\newblock Knapsack in graph groups, {HNN}-extensions and amalgamated products.
\newblock In {\em Proceedings of the 33rd International Symposium on
  Theoretical Aspects of Computer Science, STACS 2016}, volume~47 of {\em
  Leibniz International Proceedings in Informatics (LIPIcs)}, pages
  50:1--50:14, Dagstuhl, Germany, 2016. Schloss Dagstuhl--Leibniz-Zentrum
  f{\"u}r Informatik.

\bibitem{LohreyZ18}
Markus Lohrey and Georg Zetzsche.
\newblock Knapsack in graph groups.
\newblock {\em Theory of Computing Systems}, 62(1):192--246, 2018.

\bibitem{LohreyZ20}
Markus Lohrey and Georg Zetzsche.
\newblock Knapsack and the power word problem in solvable baumslag-solitar
  groups.
\newblock In {\em Proceedings of the 45th International Symposium on
  Mathematical Foundations of Computer Science, {MFCS} 2020}, volume 170 of
  {\em LIPIcs}, pages 67:1--67:15. Schloss Dagstuhl - Leibniz-Zentrum f{\"{u}}r
  Informatik, 2020.

\bibitem{LySch77}
Roger Lyndon and Paul Schupp.
\newblock {\em Combinatorial Group Theory}.
\newblock Springer, 1977.

\bibitem{MeRa04}
Vasileios Metaftsis and Evagelos Raptis.
\newblock Subgroup separability of graphs of abelian groups.
\newblock {\em Proceedings of the American Mathematical Society},
  132(7):1873--1884, 2004.

\bibitem{MiTr17}
Alexei Mishchenko and Alexander Treier.
\newblock Knapsack problem for nilpotent groups.
\newblock {\em Groups Complexity Cryptology}, 9(1):87--98, 2017.

\bibitem{MyNiUs14}
Alexei Myasnikov, Andrey Nikolaev, and Alexander Ushakov.
\newblock Knapsack problems in groups.
\newblock {\em Mathematics of Computation}, 84:987--1016, 2015.

\bibitem{Stal71}
John~Robert Stallings.
\newblock {\em Group Theory and Three-Dimensional Manifolds}.
\newblock Number~4 in Yale Mathematical Monographs. Yale University Press,
  1971.

\bibitem{To09ipl}
Anthony~Widjaja To.
\newblock Unary finite automata vs. arithmetic progressions.
\newblock {\em Inf. Process. Lett.}, 109(17):1010--1014, 2009.

\end{thebibliography}

\end{document}